\documentclass[a4paper,12pt]{amsart}
\usepackage[margin=2cm]{geometry}
\usepackage[alphabetic]{amsrefs}

\makeatletter
\@namedef{subjclassname@2020}{\textup{2020} Mathematics Subject Classification}
\makeatother

\usepackage{indentfirst}
\usepackage{fancyhdr}
\usepackage{amssymb}

\usepackage{pifont}

\usepackage[usenames,dvipsnames]{xcolor}

\usepackage{amsmath,mathtools} 
\usepackage{latexsym}
\usepackage{amsthm} 
\usepackage{eucal} 
\usepackage{enumerate} 
\usepackage{pdfpages}
\usepackage{mathrsfs}
\usepackage{stackrel}
\usepackage{enumerate}
\usepackage{graphicx}

\renewcommand{\leq}{\leqslant}
\renewcommand{\le}{\leqslant}
\renewcommand{\geq}{\geqslant}
\renewcommand{\ge}{\geqslant}

\usepackage[colorlinks=true,urlcolor=blue,
citecolor=red,linkcolor=blue,linktocpage,pdfpagelabels,
bookmarksnumbered,bookmarksopen]{hyperref}

\numberwithin{equation}{section}
\theoremstyle{plain}

\usepackage{amsthm}
\theoremstyle{plain}
\newtheorem{thm}{Theorem}[section]
\newtheorem{lem}[thm]{Lemma}
\newtheorem{cor}[thm]{Corollary}

\newtheorem{prop}[thm]{Proposition}

\newtheorem{rem}[thm]{Remark}

\newcommand{\R}{\mathbb{R}}

\newcommand{\D}{\mathcal{D}_\vartheta}

\newcommand{\e}{\varepsilon}
\allowdisplaybreaks

\begin{document}

\title[Nonlocal capillarity for anisotropic kernels]{Nonlocal capillarity for anisotropic kernels}\thanks{
{\em Alessandra De Luca}: 
Dipartimento di Matematica e Applicazioni, Universit\`a di Milano - Bicocca, Via Cozzi 55, 20125 Milano, Italy.
{\tt a.deluca41@campus.unimib.it}\\
{\em Serena Dipierro}:
Department of Mathematics
and Statistics,
University of Western Australia,
35 Stirling Hwy, Crawley WA 6009, Australia.
{\tt serena.dipierro@uwa.edu.au}
\\
{\em Enrico Valdinoci}:
Department of Mathematics
and Statistics,
University of Western Australia,
35 Stirling Hwy, Crawley WA 6009, Australia. {\tt enrico.valdinoci@uwa.edu.au}\\
The first author is member of INdAM.
The second and third authors
are members of AustMS.
The second author is supported by
the Australian Research Council DECRA DE180100957
``PDEs, free boundaries and applications''. 
The third author is supported by
the Australian Laureate Fellowship
FL190100081
``Minimal surfaces, free boundaries and partial differential equations''.
Part of this
work was carried
out during a very pleasant and fruitful visit of the first author to the
University of Western Australia, which we thank for the warm hospitality.}

\author{Alessandra De Luca}
\author{Serena Dipierro}
\author{Enrico Valdinoci}


\begin{abstract}
We study a nonlocal capillarity problem
with
interaction kernels that are possibly anisotropic
and not necessarily invariant under scaling.
In particular, the lack of scale invariance will be modeled
via two different fractional exponents~$s_1, s_2\in (0,1)$
which take into account the possibility that
the container and the environment present different features with
respect to particle interactions.

We determine a nonlocal Young's law for the contact angle and discuss the unique solvability
of the corresponding equation in terms of the interaction kernels and of the relative adhesion coeﬃcient.
\end{abstract}

\maketitle

\tableofcontents

\section{Introduction and main results}

In the classical capillarity theory (see e.g.~\cite{MR3918991, ASDfMRDGB})
the contact angle is defined as the angle~$\vartheta$ at which a liquid interface meets a solid surface.
At the equilibrium, this angle is expressed by the Young's law
in terms of the relative adhesion coeﬃcient~$\sigma$ via the classical formula
$$ \cos(\pi-\vartheta)=\sigma.$$
The contact angle plays also an important role in the
fluid spreading on a solid surface, determining also the velocity of the moving contact lines
(see e.g.~\cite{MR1051327} and the references therein).

The contact angle is certainly the ``macroscopic''
outcome of several complex ``microscopic'' phenomena, involving physical chemistry,
statistical physics and fluid dynamics,
and ultimately relying on the effect of long-range and 
distance-dependent interactions between atoms or molecules
(such as van der Waals forces). It is therefore
of great interest to understand how the interplay between different microscopic
effects generates an effective contact angle at a macroscopic scale,
and to detect the proximal regions of the interfaces (likely, at a very small distance
from the contact line)
in which the effect of the singular long-range potentials may produce a significant effect,
see e.g.~\cite{MR556259, DUSSANV1983303}.\medskip

To further understand the role of long-range particle interactions in models
related to capillarity theory,
a modification of the classical Gau{\ss} free energy functional 
has been introduced in~\cite{MR3717439} that took into account
surface tension energies of nonlocal type and modeled on the fractional perimeter
presented in~\cite{MR2675483}.
These new variational principles lead to suitable equilibrium conditions that determine
a specific contact angle depending on the relative adhesion coeﬃcient
and on the properties of the interaction kernel.
The classical limit angle was then obtained
from this long-range prescription via a limit procedure, and precise asymptotics
have been provided in~\cite{MR3707346}.
Local minimizers in the fractional capillarity model have been studied in~\cite{REV}, 
where their blow-up limits at boundary points have been considered, showing, by means of
a new monotonicity formula, that these blow-up limits are cones, and giving a complete
characterization of such cones in the planar case.
\medskip

The main goal of this paper is to present a capillarity theory of nonlocal
type in which the long-range particle interactions are
possibly anisotropic
and not necessarily invariant under scaling.
This setting is specifically motivated by the case
in which 
the potential interactions of the droplet with
the container and those with the environment
are subject to different van der Waals forces.
These two different interactions will be modeled here by
two different fractional exponents.
In this setting, we determine a nonlocal Young's law for the contact angle,
which extends the known one in the nonlocal isotropic setting and
recovers the classical one as a limit case.

We now discuss in further detail the type of particle interactions that we take into account
and the variational structure of the corresponding anisotropic nonlocal capillarity theory.

\subsection{Interaction kernels}
Owing to~\cite{MR2675483}, the most widely studied interaction kernel of singular type
in problems related to nonlocal surface tension is
\begin{equation}\label{fractkernel}
K_s(\zeta):=\frac{1}{|\zeta|^{n+s}} \quad \text{for all $\zeta\in \mathbb{R}^n\setminus\{0\}$},
\end{equation}
with~$s\in (0,1)$.
Here, we aim at considering more general kernels than the one in~\eqref{fractkernel},
with a twofold objective: on the one hand, we wish to initiate and consolidate
a nonlocal capillarity theory in an {\em anisotropic} scenario; on the other hand,
we want to also model the case in which the particle interaction
of the container has a {\em different structure} with respect to the one
of the external environment.

The first of these two goals will be pursued by considering interaction kernels that are {\em not
necessarily invariant under rotation}, the second by taking into account {\em interactions
with different homogeneity} inside the container and in the external environment.\medskip

More specifically, the mathematical setting in which we work is the following.
Given~$n\geq 2$, $s\in (0,1)$, $\lambda\geq 1$ and $\varrho\in (0,\infty]$,
we consider the family of interaction kernels, denoted by $\mathbf{K}(n,s,\lambda,\varrho)$,
containing the even functions~$K\colon \R^n\setminus \{0\}\to [0,+\infty) $
such that, for all $\zeta\in\R^n\setminus\{0\}$,
\begin{equation}\label{stimakernel}
\frac{\chi_{B_\varrho}(\zeta)}{\lambda|\zeta|^{n+s}}\leq K(\zeta)\leq \frac{\lambda}{|\zeta|^{n+s}}. 
\end{equation}
Here, we are using the notation~$B_\varrho=\R^n$ when~$\varrho=\infty$.
Also, for every $h\in\mathbb{N}$, we consider the class $\mathbf{K}^h(n,s,\lambda,\varrho)$ of all the kernels $K\in \mathbf{K}(n,s,\lambda,\varrho)\cap C^h(\R^n\setminus \{0\})$ such that, for all $\zeta\in\R^n\setminus\{0\}$,
\begin{equation}\label{stimaderivatakernel}
|D^jK(\zeta)|\leq \frac{\lambda}{|\zeta|^{n+s+j}}\qquad{\mbox{for all }} 1\leq j\leq h.
\end{equation}
We also say that the kernel~$K$ 
admits a blow-up limit if
for every~$\zeta\in\R^n\setminus\{0\}$ the following limit exists:
\begin{equation}\label{Kast}
K^\ast(\zeta):=\lim_{r\rightarrow 0^+} r^{n+s} K(r\zeta).
\end{equation}

For each kernel $K$ we consider the interaction
induced by~$K$ between any two disjoint (measurable)
subsets $E,F$ of $\mathbb{R}^n$ defined by
\begin{equation*}
I_K(E,F):=\int_E\int_F K(x-y)\,dx\,dy.
\end{equation*}
For instance, with this definition,
the so-called $K$-nonlocal perimeter of a set~$E$
associated to $K$ is given by the quantity~$I_K(E,E^c)$,
which is
the interaction of the set $E$ with its complement with respect to $\mathbb{R}^n$
(here, as usual, we use the notation~$E^c:=\R^n\setminus E$). See~\cite{MR3981295}
for several results on the $K$-nonlocal perimeter.
In particular, if $K$ is the fractional kernel in~\eqref{fractkernel}, then
the notion of $K$-perimeter boils down to the one introduced by Caffarelli, Roquejoffre and Savin in~\cite{MR2675483}.
\medskip

Given an open set $\Omega\subseteq \mathbb{R}^n$,
$s_1$, $s_2\in(0,1)$ and~$\sigma\in \R$,
for every~$K_1\in
\mathbf{K}(n,s_1,\lambda,\varrho)$ and~$K_2\in
\mathbf{K}(n,s_2,\lambda,\varrho)$
and every set~$ E\subseteq\Omega$ we define the functional
\begin{equation}\label{energy}
\mathcal{E}(E):=I_1(E,E^c\cap\Omega)+\sigma\, I_2(E,\Omega^c).
\end{equation}
Here above and in what follows, we use\footnote{We observe that when~$\sigma>0$, one could reabsorb it
into the second interaction kernel
up to redefining~$K_2$ into~$\sigma K_2$. In general, one can think that~$\sigma$ ``simply plays the role of a sign'',
say it suffices to understand the matter for~$\sigma\in\{-1,+1\}$, up to changing~$K_2$ into~$|\sigma|\,K_2$:
indeed, if~$\widetilde{K}_2:=|\sigma|\,K_2$ we have that
$$ \sigma \,I_{K_2}(E,\Omega^c)={\rm sign}(\sigma)\,|\sigma|
\int_E\int_{\Omega^c} K_2(x-y)\,dx\,dy={\rm sign}(\sigma) \, I_{\widetilde K_2}(E,\Omega^c).
$$
However, we thought it was convenient to consider~$\sigma$ as an ``independent parameter'',
since this makes it easier to compare with the classical case.}
the short notation~$I_1:=I_{K_1}$
and~$I_2:=I_{K_2}$. 
Moreover, given a function~$g\in L^\infty(\Omega)$, 
we let
\begin{equation}\label{variationalproblem}
{\mathcal{C}}(E):=
\mathcal{E}(E)+\int_{E}g(x)\,dx
.\end{equation}
The setting that we take into account is general enough to include anisotropic
nonlocal perimeter functionals as in~\cite{MR3161386, MR3981295},
which, in turn, can be seen as nonlocal modifications of the classical anisotropic
perimeter functional. In this spirit, the functional in~\eqref{variationalproblem}
can be seen as a nonlocal generalization of classical anisotropic capillarity problems,
such as the ones in~\cite{MR3317808}.
As customary in the analysis of nonlocal problems arising from geometric functionals,
the long-range interactions involved in~\eqref{variationalproblem} produce
significant energy contributions which will give rise to structural differences
with respect to the classical case.
\medskip

The goal of this article is to study the minimizers
of the nonlocal capillarity functional~${\mathcal{C}}$ among all the sets~$E$ with a given volume.

The case in which~$K_1(\zeta)=K_2(\zeta)=K_s(\zeta)$
as in~\eqref{fractkernel} has been studied in~\cite{MR3717439, MR3707346, REV}.
Here instead we are specifically interested in the
nonlocal capillarity energy in~\eqref{variationalproblem}
with two different types of interactions between~$E$ and~$\Omega\cap E^c$ and between~$E$
and~$\Omega^c$, as modeled in~\eqref{energy}.

\subsection{Preliminary results: existence theory and Euler-Lagrange equation}

We now describe some basic features of the capillarity energy
functional~${\mathcal{C}}$ in~\eqref{variationalproblem}.
First of all, we have that the volume constrained minimization of this functional 
is well-posed, according to the following statement:

\begin{prop}[Existence of minimizers]\label{existenceofminimizer}
Let~$K_1\in \mathbf{K}(n,s_1,\lambda,\varrho)$ and~$K_2\in \mathbf{K}(n,s_2,\lambda,\varrho)$.
Let~$\Omega$ be an open and bounded set with~$I_1(\Omega,\Omega^c)+I_2(\Omega,\Omega^c)<+\infty$.

Let $m\in (0,|\Omega|)$ and $g\in L^\infty(\Omega)$.

Then, there exists a minimizer for the functional~${\mathcal{C}}$ in~\eqref{variationalproblem}
among all the sets~$E$ with Lebesgue measure equal to~$m$.

Moreover, $I_{1}(E,E^c\cap \Omega)<+\infty$ for every minimizer $E$.
\end{prop}

In the formulation given here, Proposition~\ref{existenceofminimizer}
is new in the literature, though its proof relies on an appropriate variation
of standard techniques, see e.g.~\cite{MR2675483, MR3717439}.
Nevertheless, we provide its proof in Appendix~\ref{APPENDI-A}, since here
we would like to point out
some modifications due to the facts that~$\sigma\in\R$ and the kernels have different homogeneities, differently from~\cite{MR3717439}.
\medskip

The volume constrained minimizers (and, more generally,
the volume constrained critical points)
obtained in Proposition~\ref{existenceofminimizer}
satisfy (under reasonable regularity assumptions on the domain and on the interaction kernels)
a suitable Euler-Lagrange equation, according to the following result.
To state it precisely, it is convenient to denote by~$
\mathrm{Reg}_E$ the collection of all those points~$x_0\in\overline{\Omega\cap\partial E}$
for which there exists~$\rho>0$
and~$\alpha\in(s_1,1)$ such that~$B_\rho(x_0)\cap\partial E$ is a
manifold of class~$C^{1,\alpha}$
possibly with boundary,
and the boundary (if any) is contained in~$\partial\Omega$,
see Figure~\ref{FIGURATBORDINP}.

\begin{figure}[h]
\includegraphics[width=0.55\textwidth]{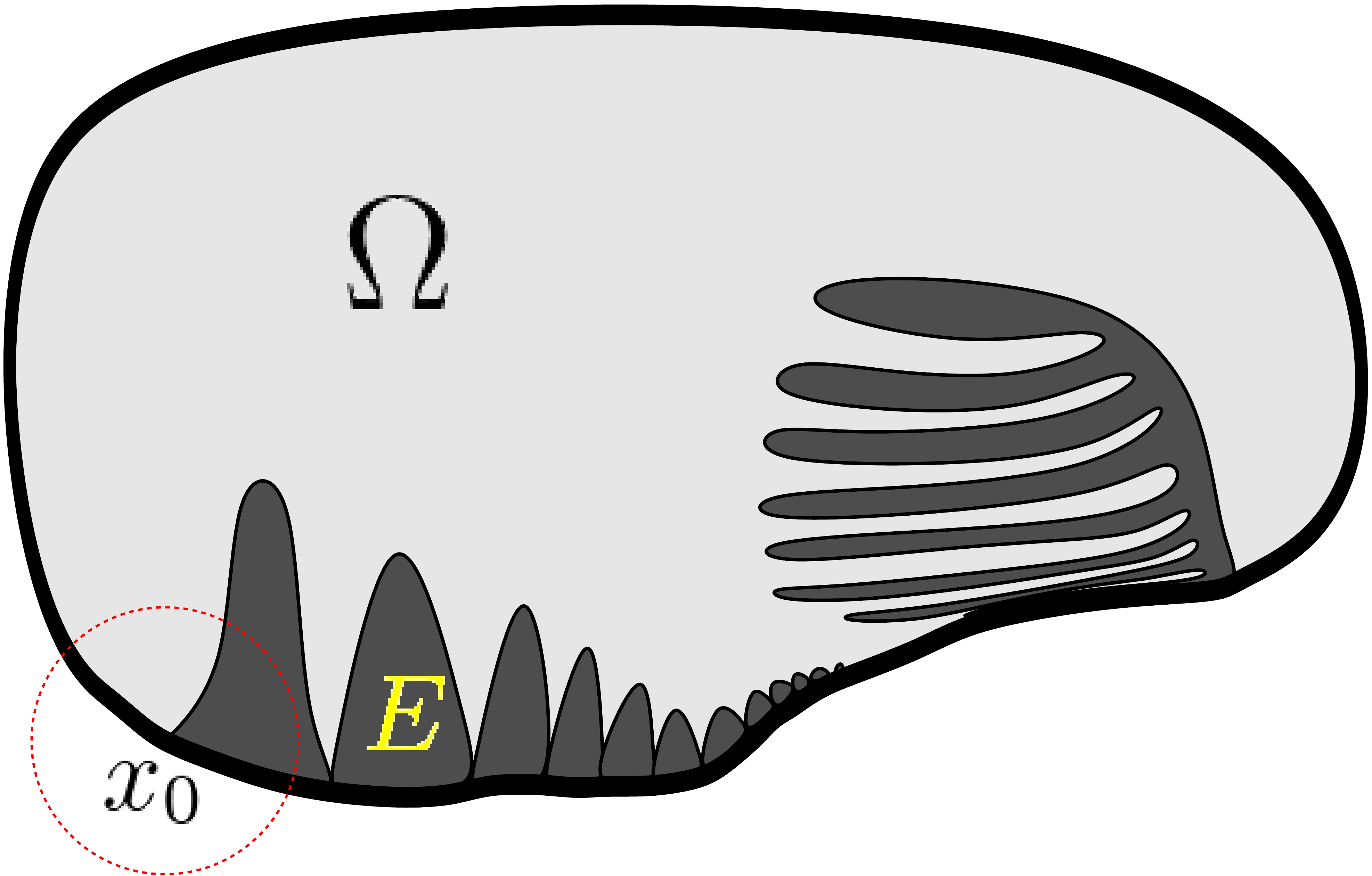}
\caption{The geometry involved in the definition of $\mathrm{Reg}_E$.}
        \label{FIGURATBORDINP}
\end{figure}

Given a kernel~$K\in \mathbf{K}(n,s_1,\lambda,\varrho)$, it is also convenient to recall the notion of
$K$-mean curvature, that is defined, for all $x\in \Omega\cap\mathrm{Reg}_E$, as
\begin{equation}\label{meancurvature}
\mathbf{H}^{K}_{\partial E}(x):=
\mathrm{p.v.}\int_{\R^n} K(x-y)\bigl(\chi_{E^c}(y)-\chi_E(y)\bigr) \,dy.
\end{equation}
Here $\mathrm{p.v.}$ stands for the principal value, that we omit from now on for the sake of
simplicity of notation.

We also say that~$E\subseteq\Omega$ is a critical point of~${\mathcal{C}}$ among sets
with prescribed Lebesgue measure
if
$$ \frac{d}{dt}\Big|_{t=0} {\mathcal{C}}\big(f_t(E)\big)= 0 ,$$
for every family of diffeomorphisms~$\{f_t\}_{|t|<\delta}$
such that, for each~$|t|<\delta$, one has that~$f_0={\rm{Id}}$, the support
of~$f_t-{\rm{Id}}$ is a compact set, $f_t(\Omega)=\Omega$ and~$|f_t(E)|=|E|$.

With this notation, we have the following result:

\begin{prop}[Euler-Lagrange equation]\label{teoremaEL}
Let~$K_1\in\mathbf{K}^1(n,s_1,\lambda, \varrho)$ and $K_2\in \mathbf{K}^1(n,s_2,\lambda,\varrho)$.
Let~$\Omega$ be an open bounded set with $C^1$-boundary,
$m\in (0,|\Omega|)$
and~$g\in C^1(\R^n)$.

Let~$E$ be a critical point of~${\mathcal{C}}$ in~\eqref{variationalproblem}
among all the sets with Lebesgue measure equal to~$m$. 

Then, there exists $c\in\R$ such that
\begin{equation}\label{E-Lw}
\begin{split}
&\iint_{E\times(E^c\cap \Omega)}\mathrm{div}_{(x,y)}\Big(K_1(x-y)\big(T(x),T(y)\big)\Big)\, dx\, dy\\
&+\sigma \iint_{E\times \Omega^c} \mathrm{div}_{(x,y)}\Big(K_2(x-y)\big(T(x),T(y)\big)\Big)\, dx\, dy+
\int_{E}\mathrm{div}(g\, T)=c\int_E \mathrm{div}T
\end{split}
\end{equation}
for every $T\in C^\infty_c(\mathbb{R}^n;\mathbb{R}^n)$ with
\[
T\cdot \nu_\Omega=0 \quad\text{on $\partial\Omega$}.
\]

Moreover, if $K_1\in \mathbf{K}^2(n,s_1,\lambda,\varrho)$ and $K_2\in \mathbf{K}^2(n,s_2,\lambda,\varrho)$, then
\begin{equation}\label{E-Ls}
\mathbf{H}^{K_{1}}_{\partial E}(x)-\int_{\Omega^c}K_1(x-y)\,dy+\sigma\int_{\Omega^c}K_2(x-y)\,dy+g(x)=c
\end{equation}
for all $x\in \Omega\cap\mathrm{Reg}_E$.
\end{prop}

The proof of Proposition~\ref{teoremaEL}
relies on a modification
of techniques previously exploited in~\cite{MR2675483, MR3322379, MR3717439}.
We omit the proof here since
one can follow precisely the proof of
Theorem~1.3 in~\cite{MR3717439} with obvious modifications due to the presence of different
kernels.
\medskip

We now present the main results of this paper, which are focused
on the determination of the contact angle.

\subsection{Main results: nonlocal Young's law}

One of the pivotal steps of any capillarity theory is the determination
of the contact angle between the droplet and the container
(in jargon, the Young's law). In our setting,
this Young's law is very sensitive to the relative homogeneity of the
interacting kernels. 
\medskip

Loosely speaking, when~$s_1< s_2$, at small scales (which are the ones
which we believe are more significant in the local determination of the contact angle),
the interaction between the droplet and the exterior of the container
prevails\footnote{For instance, when~$s_1<s_2$, $\Omega:=\{x_n>0\}$,
$E:=\{0<x_n<\beta\,|x'|\}$
and~$r\in(0,\varrho)$,
one sees from~\eqref{stimakernel} and the change of variables~$(X,Y):=\left(\frac{x}r,\frac{y}r\right)$ that
\begin{eqnarray*}&&
\frac{I_1(E\cap B_r,E^c\cap\Omega\cap B_r)}{I_2(E\cap B_r,\Omega^c\cap B_r)}\le\lambda^2\;
\frac{\displaystyle\iint_{(E\cap B_r)\times(E^c\cap\Omega\cap B_r)}\frac{dx\,dy}{|x-y|^{n+s_1}}}{
\displaystyle\iint_{(E\cap B_r)\times(\Omega^c\cap B_r)}\frac{dx\,dy}{|x-y|^{n+s_2}}}\\&&\qquad=\lambda^2\;r^{s_2-s_1}\;
\frac{\displaystyle\iint_{(E\cap B_1)\times(E^c\cap\Omega\cap B_1)}\frac{dX\,dY}{|X-Y|^{n+s_1}}}{
\displaystyle\iint_{(E\cap B_1)\times(\Omega^c\cap B_1)}\frac{dX\,dY}{|X-Y|^{n+s_2}}},
\end{eqnarray*}
which is infinitesimal when~$r\searrow0$. This suggests that in the small vicinity of contact points,
when~$s_1<s_2$, the effect of the kernel~$K_2$ in the determination of the energy minimizers
and of their geometric properties plays a dominant role with respect to that played by~$K_1$.} with respect to the one between the droplet and the interior of the container.
Thus, in this situation, the 
sign of the relative adhesion coeﬃcient~$\sigma$ becomes determinant:
in the hydrophilic regime~$\sigma<0$ the droplet is ``absorbed'' 
by the boundary of the container, thus producing a zero contact angle;
instead, in the hydrorepellent regime~$\sigma>0$
the droplet is ``held off'' the boundary of the container,
thus producing a contact angle equal to~$\pi$; finally,
in the neutral case~$\sigma=0$ the behavior of the second interaction kernel
becomes irrelevant. When~$\sigma=0$ and additionally the problem is isotropic, the contact angle becomes~$\pi/2$.
\medskip

Conversely, when~$s_1>s_2$,
the interaction between the droplet and the interior of the container
is, at small scales, significantly stronger
than that between the droplet and the exterior of the container.
In this situation, the relative adhesion coeﬃcient~$\sigma$
does not play any role and the contact angle is determined
by an integral cancellation condition (that will be explicitly provided in~\eqref{cancellation}).
When the first kernel is isotropic, this condition simplifies and the contact
angle is proved to be~$\pi/2$.

More precisely, the determination of the contact angle relies on a delicate cancellation
of the singular kernel contributions, which requires the determination
of an auxiliary angle which is ``symmetric'' (in a suitable sense of ``measuring singular
interactions'') with respect to the contact angle itself: this ``dual contact angle''
will be denoted by~$\widehat{\vartheta}$ and the cancellation property will be described
in detail in the forthcoming formula~\eqref{cancellation}.
\medskip

The detailed analysis of the contact angle when~$s_1\ne s_2$ is given in the forthcoming
Theorem~\ref{thmYounglaw}.
When instead~$s_1=s_2$, the internal and external interactions
equally contribute at small scales. This situation will be analyzed in Theorem~\ref{thmYounglawa1a2nonconst}
and will lead to a contact angle described by an integral condition (given explicitly in~\eqref{Ylaw}
and reformulated in~\eqref{sigmagenerale-0} below).

We now dive into the technicalities required for the determination
of the contact angle. Namely, using the Euler-Lagrange equation in~\eqref{E-Ls} and taking blow-ups
along sequences of interior points converging
to~$\partial\Omega\cap \mathrm{Reg}_E$, we derive
two versions of the nonlocal Young's law depending on whether $s_1\neq s_2$ or $s_1=s_2$. For this, we introduce the following notations that will be used throughout all this paper:
\begin{itemize}
\item given a set~$F\subseteq\R^n$, $x_0\in\R^n$
and~$r>0$, we let
\begin{equation}\label{blowupsets}
F^{x_0,r}:=\frac{F-x_0}{r};
\end{equation}
\item for any two angles $\vartheta_1$, $\vartheta_2\in [0,2\pi)$, with~$\vartheta_1<\vartheta_2$, we define
\begin{equation}\label{Jtheta1theta2}
J_{\vartheta_1,\vartheta_2}:=
\Big\{x\in \mathbb{R}^n\,:\, \exists\ \beta\in (\vartheta_1, \vartheta_2),\;
\rho>0\ \mathrm{such\ that\ }(x_1,x_n)=\rho(\cos\beta,\sin\beta) \Big\};
\end{equation}
\item for any angle $\alpha$, we set
\begin{equation}\label{etheta}
e(\alpha):=\cos\alpha \, e_1+\sin\alpha\, e_n.
\end{equation}
\end{itemize}
In order to establish the nonlocal Young's law, we
consider $K_1\in\mathbf{K}^2(n,s_1,\lambda,\varrho)$ and $K_2\in\mathbf{K}^2(n,s_2,\lambda,\varrho)$
such that the associated blow-up
kernels defined as in~\eqref{Kast} are well-defined and given by
\begin{equation}\label{adsyrytrjgtjhg697867}
K_1^\ast(\zeta)=\frac{a_1(\overrightarrow{\zeta})}{|\zeta|^{n+s_1}}\qquad{\mbox{and}}\qquad  K_2^\ast(\zeta)=\frac{a_2(\overrightarrow{\zeta})}{|\zeta|^{n+s_2}},
\end{equation}
where $\overrightarrow{\zeta}:=\frac{\zeta}{|\zeta|}$ and $a_1,a_2 $ are continuous functions on~$\partial B_1$,
bounded from above and below by two positive constants and satisfying
\begin{equation}\label{AIPA} a_i(\omega)=a_i(-\omega) \end{equation}
for all~$\omega\in\partial B_1$ and~$i\in\{1,2\}$.\medskip

Before exhibiting the main results of this paper, we premise the following result which has been thought in order to reproduce a cancellation of terms as in~\cite{MR3717439}.
This result points out that in this context a new construction is needed since the function~$a_1$ is anisotropic.

\begin{prop}\label{cancellazione} 
Given~$\vartheta\in(0,\pi)$, for every $\bar{\vartheta}\in (0,2\pi)$ let
\begin{equation}\label{D}
\mathcal{D}_\vartheta(\bar{\vartheta}):=
\int_{J_{\vartheta,\vartheta+\bar{\vartheta}}}\frac{a_1(\overrightarrow{x-e(\vartheta)})}{|x-e(\vartheta)|^{n+s_1}}\, dx-\int_{J_{0,\vartheta}}\frac{a_1(\overrightarrow{x-e(\vartheta)})}{|x-e(\vartheta)|^{n+s_1}}\, dx .
\end{equation}
Then, 
\begin{eqnarray}
&&{\mbox{$\mathcal{D}_\vartheta$ is well-defined in the principal value sense;}} \label{XRTS-1}\\
&&{\mbox{$\mathcal{D}_\vartheta$ is continuous in~$(0,2\pi)$;}} \label{XRTS-190}\\
&&\lim_{\bar{\vartheta}\searrow 0}\mathcal{D}_\vartheta(\bar{\vartheta}) = -\infty ; \label{limit1}\\
&&\lim_{\bar{\vartheta}\nearrow 2\pi}\mathcal{D}_\vartheta(\bar{\vartheta}) = +\infty.\label{limit2}
\end{eqnarray} 
Moreover,
for every~$c\in\R$ and every angle $\vartheta\in (0,\pi)$, there exists a  unique angle $\widehat{\vartheta}\in (0,2\pi)$ such that 
\begin{equation}\label{cancellation}
\mathcal{D}_\vartheta(\widehat{\vartheta})=c.
\end{equation}
\end{prop}

\begin{figure}[h]
\includegraphics[width=0.55\textwidth]{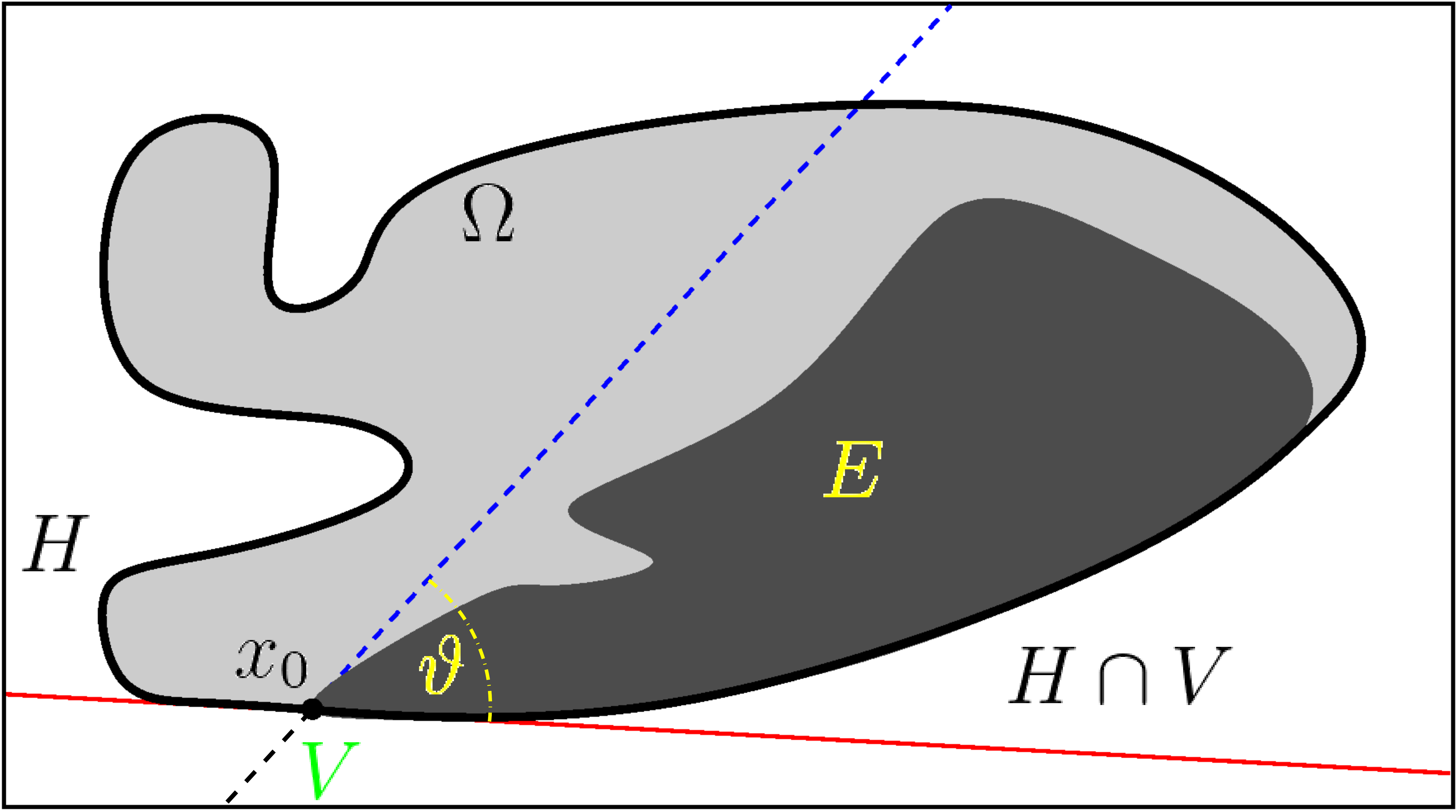}
\caption{The geometry involved in the asymptotics in~\eqref{NOHA}.}
        \label{FIGURATB}
\end{figure}

We showcase below a first version of the nonlocal Young's law corresponding to the case~$s_1\neq s_2$.

\begin{thm}\label{thmYounglaw}
Let~$K_1\in\mathbf{K}^2(n,s_1,\lambda,\varrho)$ and~$K_2\in\mathbf{K}^2(n,s_2,\lambda,\varrho)$.
Suppose that $K_1$, $K_2$ admit
blow-up limits~$K_1^\ast$, $K_2^\ast$
(according to~\eqref{Kast}) that satisfy assumption~\eqref{adsyrytrjgtjhg697867}. 

Let $g\in C^1(\R^n)$. Let $\Omega$ be an open bounded set with $C^1$-boundary and $E$ be a volume-constrained critical set of $\mathcal{C}$. 

Let $x_0\in \mathrm{Reg}_E\cap \partial\Omega$ and suppose that $H$ and $V$ are open
half-spaces such that
\begin{equation}\label{NOHA}
\Omega^{x_0,r}\rightarrow H\qquad {\mbox{and}}\qquad
E^{x_0,r}\rightarrow H\cap V\quad {\mbox{in }} L^1_\mathrm{loc}(\R^n)
\;{\mbox{as }}r\rightarrow 0^+.
\end{equation}
Let also~$\vartheta\in [0,\pi]$ be the angle between
the half-spaces $H$ and $V$, that is $H\cap V=J_{0,\vartheta}$ in the notation of~\eqref{Jtheta1theta2}.

Then, the following statements hold true.
\begin{itemize}
\item[1)] If~$s_1<s_2$ and~$\sigma<0$ then~$\vartheta=0$.

\item[2)] If~$s_1<s_2$ and~$\sigma>0$ then~$\vartheta=\pi$.

\item[3)] If
\begin{equation}\label{magvalimpo2}
{\mbox{either~$s_1<s_2$ and~$\sigma=0$, or~$s_1>s_2$,}}\end{equation} then~$\vartheta\in(0,\pi)$. 
Also, letting $\widehat{\vartheta}\in (0,2\pi)$ be as in~\eqref{cancellation}
with~$c=0$, we have that~$\widehat{\vartheta}=\pi-\vartheta$. 
Moreover,
for all $v\in H\cap \partial V$, 
\begin{equation}\label{Y:LS}
\mathbf{H}^{K_1^\ast}_{\partial (H\cap V)} (v)-\int_{H^c}K_1^\ast (v-y)\,dy=0.
\end{equation}
\end{itemize}
\end{thm}

The asymptotics in~\eqref{NOHA} are depicted in Figure~\ref{FIGURATB}.
As a particular case of Theorem~\ref{thmYounglaw}, we single out the special situation
in which the kernel~$K^*_1$ is isotropic. In this setting, 
the cancellation condition in~\eqref{cancellation}
boils down to an explicit condition for the contact angle,
and
we have:
 
\begin{cor}\label{corollario1}
Under the same assumptions of Theorem \ref{thmYounglaw},
we additionally suppose that~$a_1\equiv \mathrm{const}$.

Then, the following statements hold true.
\begin{itemize}
\item[1)] If~$s_1<s_2$ and~$\sigma<0$ then~$\vartheta=0$.

\item[2)] If~$s_1<s_2$ and~$\sigma>0$ then~$\vartheta=\pi$.

\item[3)] If either~$s_1<s_2$ and~$\sigma=0$, or~$s_1>s_2$, then~$\vartheta=\frac{\pi}2$. 
\end{itemize}
\end{cor}

We exhibit below the nonlocal Young's law in the case $s_1=s_2$, which was left out of Theorem~\ref{thmYounglaw}.

\begin{thm}\label{thmYounglawa1a2nonconst}
Let~$s\in(0,1)$ and~$K_1$, $K_2\in\mathbf{K}^2(n,s,\lambda,\varrho)$.
Suppose that $K_1$, $K_2$ admit
blow-up limits~$K_1^\ast$, $K_2^\ast$
(according to~\eqref{Kast}) that satisfy assumption~\eqref{adsyrytrjgtjhg697867}. 
Assume that there exists~$\varepsilon_0\in(0,1)$ such that
\begin{equation}\label{magvalimpo}
|\sigma|\, K_2(\zeta) \le (1-\varepsilon_0) \,K_1(\zeta) \qquad{\mbox{for all }}\zeta\in B_{\varepsilon_0}\setminus\{0\}.
\end{equation}

Let $g\in C^1(\R^n)$. Let~$\Omega$ be an open bounded set with $C^1$-boundary and $E$ be a volume-constrained critical set of $\mathcal{C}$. 

Let $x_0\in \mathrm{Reg}_E\cap \partial\Omega$ and suppose that $H$ and $V$ are open
half-spaces such that
\begin{equation*}
\Omega^{x_0,r}\rightarrow H\qquad {\mbox{and}}\qquad
E^{x_0,r}\rightarrow H\cap V\quad {\mbox{in }} L^1_\mathrm{loc}(\R^n)
\;{\mbox{as }}r\rightarrow 0^+.
\end{equation*}
Let also~$\vartheta\in [0,\pi]$ be the angle between
the half-spaces $H$ and $V$, that is $H\cap V=J_{0,\vartheta}$ in the notation of~\eqref{Jtheta1theta2},
and let~$\nu_E(x_0):=\nu_V(0)$.

Then, we have that~$\vartheta\in(0,\pi)$ and, for all $v\in H\cap \partial V$,
\begin{equation}\label{Ylaw}
\mathbf{H}^{K_1^\ast}_{\partial(H\cap V)}(v)-\int_{H^c}K_1^\ast (v-z)\, dz+\sigma\int_{H^c}K_2^\ast (v-z)\,dz=0.
\end{equation}
\end{thm}

Even in the very special situation in which~$K_1(\zeta)=K_2(\zeta)=\frac1{|\zeta|^{n+s}}$,
Theorem~\ref{thmYounglawa1a2nonconst}
here can be seen as a strengthening of Theorem~1.4 in~\cite{MR3717439}
(and, in particular, of formula~(1.24) there): indeed, the result here establishes
explicitly the nondegeneracy of the contact angle~$\vartheta$ by proving that~$\vartheta\in(0,\pi)$.\medskip

We point out that the case~$\sigma=0$
makes indistinguishable the setting~$s_1=s_2$ from that of~$s_1\ne s_2$:
consistently with this, we observe that the contact angle prescription when~$s_1=s_2$,
as given in~\eqref{Ylaw}, reduces to~\eqref{Y:LS}
when additionally~$\sigma=0$.\medskip

Also, we remark that when~$\sigma=0$ condition~\eqref{magvalimpo} is automatically satisfied. Furthermore, when~$K_1=K_2$, condition~\eqref{magvalimpo} reduces to~$\sigma\in(-1,1)$, which is precisely the assumption taken in~\cite{MR3717439}.\medskip

Besides, we think that the detection of a contact angle in a nonlocal capillarity setting
is an interesting feature in itself, especially when we compare this situation
with the stickiness phenomenon for the nonlocal minimal surfaces, as discovered in~\cite{MR3596708}.
More specifically, for nonlocal minimal surfaces, the long-range interactions
make it possible for the surface to stick to a domain (even if the domain is smooth and convex),
thus changing dramatically the boundary analysis (moreover, this phenomenon is essentially ``generic'', see~\cite{MR4104542}). The possible detection of the contact
angle for the nonlocal capillarity theory instead highlights the fact that the boundary analysis of this theory is somewhat ``sufficiently robust'' with respect to the classical case.
Roughly speaking, we believe that this important difference between nonlocal minimal surfaces
and nonlocal capillarity theory is due to the fact that in the latter the mass is always placed in a bounded
region, whence the energy contributions coming from infinity have a different nature than the ones occurring
for nonlocal perimeter functionals.\medskip

We also stress that conditions~\eqref{magvalimpo2} and~\eqref{magvalimpo} basically state that if the kernel~$K_2$ is ``too strong'', then one cannot expect nontrivial minimizers. Roughly speaking,
while Proposition~\ref{existenceofminimizer} always guarantees the existence
of a minimizer, when conditions~\eqref{magvalimpo2} and~\eqref{magvalimpo}
are violated the minimizer can ``detach from
the boundary'' (or ``completely stick to the boundary''),
hence the notion of contact angle becomes degenerate or void.
That is, while for the existence of minimizers
we do not need to require any bound on the
relative adhesion coeﬃcient~$\sigma$ in dependence of the interaction kernels,
to speak about a contact angle some quantitative condition is in order (roughly speaking,
otherwise the droplet does not meet the boundary of the container
with a nontrivial angle, rather preferring to either detach from the container
and float, or to completely stick at the boundary
by surrounding it).

\begin{figure}[h]
\includegraphics[width=0.35\textwidth]{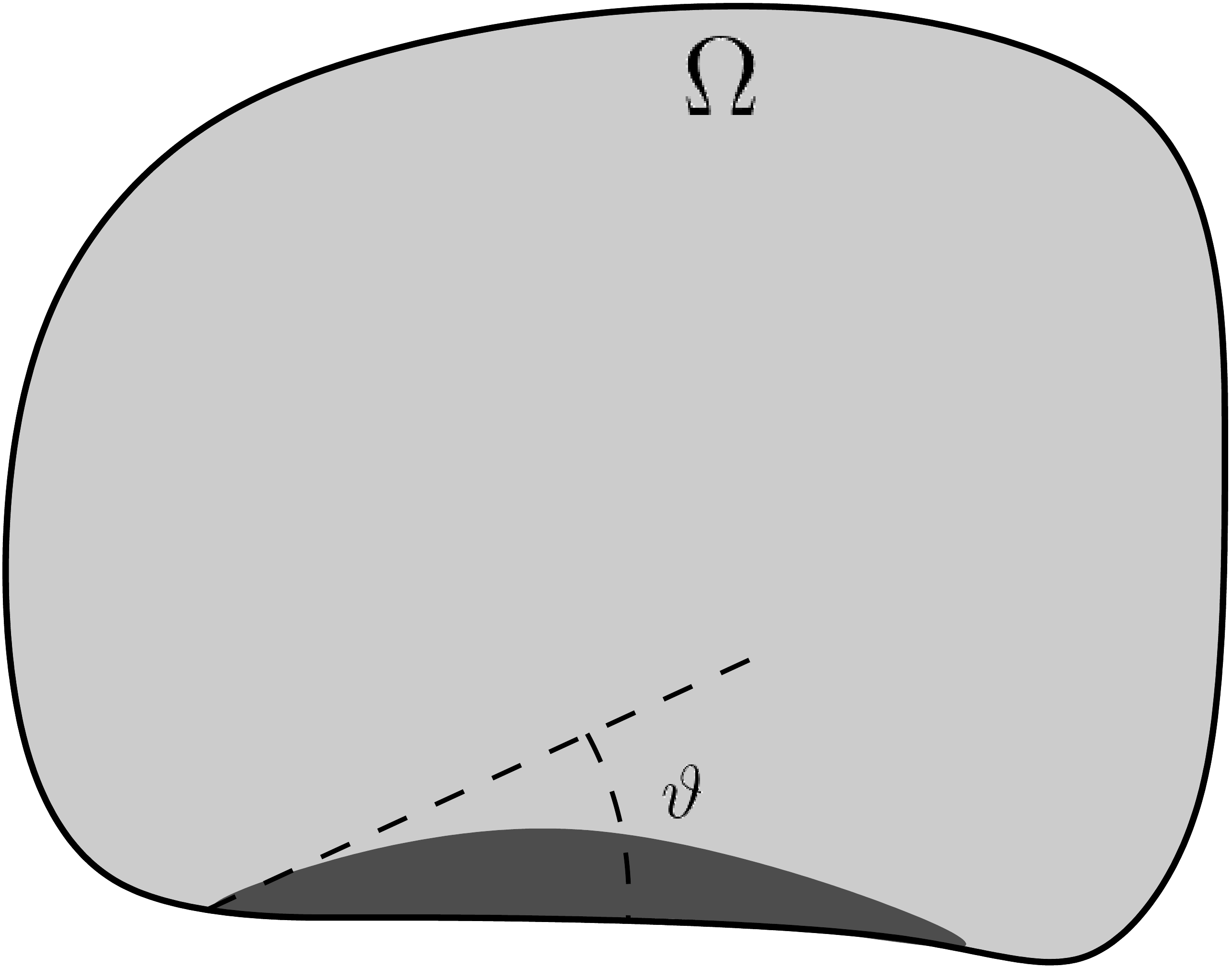}
\caption{The configuration in which the droplet tends to stick to the container.}
        \label{PICCOLO}
\end{figure}

The configuration in which the droplet tends to 
be squashed on the container,
thus producing a contact angle~$\vartheta$ close to zero, 
is sketched in Figure~\ref{PICCOLO}. The opposite situation in which the droplet tends to detach from the container, thus producing a contact angle~$\vartheta$ close to~$\pi$, is depicted in Figure~\ref{GRANDE}.

\begin{figure}[h]
\includegraphics[width=0.45\textwidth]{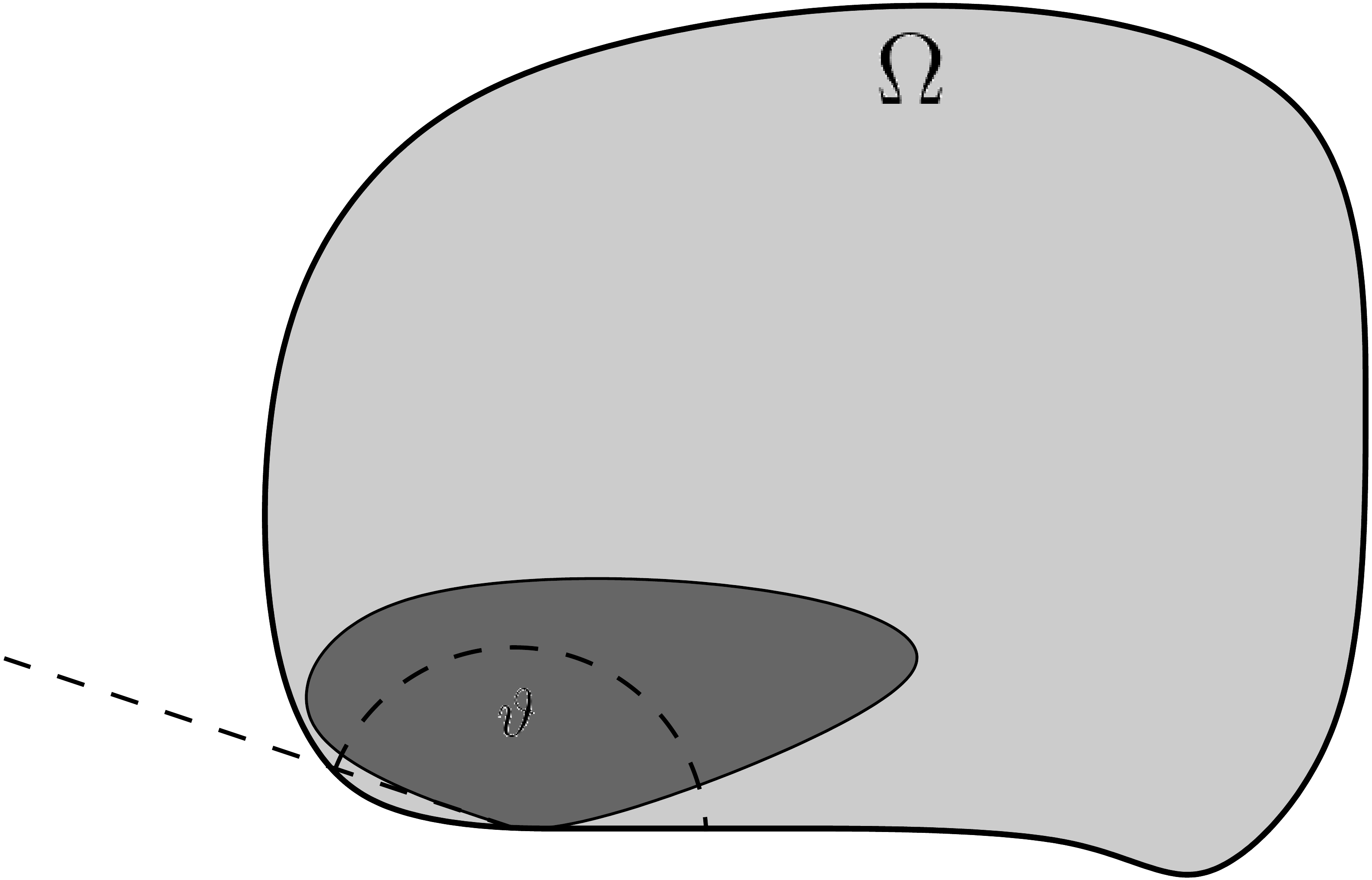}
\caption{The configuration in which the droplet tends to detach from the container.}
        \label{GRANDE}
\end{figure}

These concepts are made explicit in the following exemplifying observations:

\begin{thm}\label{ESEMPIOPTO}
Let~$\sigma>0$, $\Omega:=B_1$, $g:=0$, $K_1(\xi):=\frac{k_1}{|\xi|^{n+s_1}}$ and~$K_2(\xi):=\frac{k_2}{|\xi|^{n+s_2}}$, for some~$k_1$, $k_2>0$.

Let~$E$ be a volume-constrained minimizer of~$\mathcal{C}$. Assume that there exist~$p\in\partial B_1$ and~$\e_0>0$ such that~$B_{\e_0}(p)\cap B_1\subseteq E$. Assume also that~$\mathrm{Reg}_E\cap\Omega\ne\varnothing$.

Then, either~$s_1>s_2$, or~$s_1=s_2$ and~$k_1>\sigma k_2$. 
\end{thm}

\begin{thm}\label{ESEMPIOPTO:CO}
Let~$\sigma<0$, $\Omega:=B_1$, $g:=0$, $K_1(\xi):=\frac{k_1}{|\xi|^{n+s_1}}$ and~$K_2(\xi):=\frac{k_2}{|\xi|^{n+s_2}}$, for some~$k_1$, $k_2>0$.

Let~$E$ be a volume-constrained minimizer of~$\mathcal{C}$. Assume that there exist~$p\in\partial B_1$ and~$\e_0>0$ such that~$B_{\e_0}(p)\cap B_1\subseteq(\Omega\setminus E)$.
Assume also that~$\mathrm{Reg}_E\cap\Omega\ne\varnothing$.

Then, either~$s_1>s_2$, or~$s_1=s_2$ and~$-k_1<\sigma k_2$. 
\end{thm}

We now reformulate the condition of contact angle according to the following result:

\begin{prop}\label{chilosa22}
Let~$K_1^\ast$ and~$K_2^\ast$ be as in~\eqref{adsyrytrjgtjhg697867}.
Let~$\sigma\in\R$.
Assume that 
\begin{equation}\label{sjwicru348bt4v8b4576848poiuytrewqasdfghjkmnbvs}
{\mbox{either~$s_1=s_2$, or~$\sigma=0$.}}\end{equation}
Let~$H$ and~$V$ be open half-spaces and let~$\vartheta\in(0,\pi)$ 
be the angle between~$H$ and~$V$, that is~$H\cap V=J_{0,\vartheta}$ in the notation of~\eqref{Jtheta1theta2}. Let also~$\widehat{\vartheta}\in (0,2\pi)$ be as in~\eqref{cancellation}
with~$c:=0$

Suppose that there exists~$v\in H\cap\partial V$ such that
\begin{equation}\label{vfjoovvvv9990000043554647}
\mathbf{H}^{K_1^\ast}_{\partial(H\cap V)}(v)-\int_{H^c}K_1^\ast (v-z)\, dz+\sigma\int_{H^c}K_2^\ast (v-z)\,dz=0.
\end{equation}

Then, we have that~$\vartheta$ and~$\sigma$ satisfy the relation
\begin{equation}\label{sigmagenerale-0}\begin{split}
\int_{J_{\vartheta,\pi}}\frac{a_1(\overrightarrow{e(\vartheta)-x})}{|e(\vartheta)-x|^{n+s_1}}\, dx
-\int_{J_{0,\vartheta}}\frac{a_1(\overrightarrow{e(\vartheta)-x})}{|e(\vartheta)-x|^{n+s_1}}\, dx
+\sigma  \int_{H^c}\frac{a_2(\overrightarrow{e(\vartheta)-x})}{|e(\vartheta)-x|^{n+s_1}}\,dx=0.\end{split}
\end{equation}
\end{prop}

A topical question in view of Proposition~\ref{chilosa22}
is therefore to understand whether or not equation~\eqref{sigmagenerale-0}
identifies a unique contact angle~$\vartheta$. This is indeed the case, precisely under the natural condition in~\eqref{sjwicru348bt4v8b4576848poiuytrewqasdfghjkmnbvs}, according to the following result
in Theorem~\ref{CO:AN:CO}. To state it in full generality, it is convenient to introduce some notation.
Indeed, in the forthcoming computations, it comes in handy to reduce the problem
to a two-dimensional situation. For this, one revisits the setting in~\eqref{Jtheta1theta2}
by defining its two-dimensional projection onto the variables~$(x_1,x_n)$, namely one sets
\begin{equation}\label{Jtheta1theta2-Di2}
J_{\vartheta_1,\vartheta_2}^\star:=
\Big\{(x_1,x_n)\in \mathbb{R}^2\,:\, \exists\ \beta\in (\vartheta_1, \vartheta_2),\;
\rho>0\ \mathrm{such\ that\ }(x_1,x_n)=\rho(\cos\beta,\sin\beta) \Big\}.
\end{equation}
Let also~$e^\star(\vartheta):=(\cos\vartheta,\sin\vartheta)$ and, for every~$x=(x_1,x_2)\in\partial B_1\subseteq\R^2$
and~$j\in\{1,2\}$,
\begin{equation}\label{Jtheta1theta2-Di3}
a_j^\star(x):=\begin{cases}
a_j(x) & {\mbox{ if }}n=2,\\
\\
\displaystyle
\int_{\R^{n-2}}\frac{a_j\Big(\overrightarrow{
x_1 \,e_1+
x_2 \,e_n+|x|(0,\bar{y},0)}\Big)}{
\big(1+|\bar{y}|^2\big)^{\frac{n+s_j}2}
}\,d\bar{y}
& {\mbox{ if }}n\ge3.
\end{cases}
\end{equation}
Let also
\begin{equation}\label{Jtheta1theta2-Di4}
\phi_j(\vartheta):=a_j^\star(\cos\vartheta,\sin\vartheta).\end{equation}
We remark that, as a byproduct of~\eqref{AIPA},
\begin{equation}\label{AIPA2}
a_j^\star(x)=a^\star_j(-x)\qquad{\mbox{and}}\qquad\phi_j(\vartheta)=\phi_j(\pi+\vartheta).
\end{equation}

With this framework, we can state the existence and uniqueness result for
the contact angle equation as follows:

\begin{thm}\label{CO:AN:CO}
Let~$K_1^\ast$ and~$K_2^\ast$ be as in~\eqref{adsyrytrjgtjhg697867}.
Let~$\sigma\in\R$ and assume that~\eqref{sjwicru348bt4v8b4576848poiuytrewqasdfghjkmnbvs} holds true.

Then, there exists at most one~$\vartheta\in(0,\pi)$ satisfying
the contact angle condition in~\eqref{sigmagenerale-0}.

Furthermore, if 
\begin{equation}\label{0oujfn-29roh-32eirj-9034o5t-PK}
|\sigma|<\frac{\displaystyle\int_0^{\pi}\phi_1(\alpha)\,(\sin \alpha)^{s_1}\,d\alpha}{\displaystyle\int_0^{\pi}\phi_2(\alpha)\,(\sin \alpha)^{s_1}\,d\alpha}
\,,\end{equation}
then there exists a unique solution~$\vartheta\in(0,\pi)$ of~\eqref{sigmagenerale-0}.
\end{thm}

We stress once again that when~$a_1=a_2$ (and in particular for constant~$a_1=a_2$),
assumption~\eqref{0oujfn-29roh-32eirj-9034o5t-PK} reduces to the structural assumption~$|\sigma|<1$
that was taken in~\cite{MR3717439}.

Moreover, if~$K_1(\xi):=\frac{k_1}{|\xi|^{s_1}}$ and~$K_2(\xi):=\frac{k_2}{|\xi|^{s_2}}$ for some~$k_1$, $k_2>0$,
then assumption~\eqref{0oujfn-29roh-32eirj-9034o5t-PK} boils down to~$|\sigma|<\frac{k_1}{k_2}$, which is
precisely the condition for nontrivial minimizers obtained in Theorems~\ref{ESEMPIOPTO}
and~\ref{ESEMPIOPTO:CO}.

For these reasons, Theorem~\ref{CO:AN:CO} showcases the interesting fact that
the equation prescribing the contact angle in~\eqref{sigmagenerale-0}
admits one and only one solution precisely in the natural range of kernels given by~\eqref{sjwicru348bt4v8b4576848poiuytrewqasdfghjkmnbvs}
and~\eqref{0oujfn-29roh-32eirj-9034o5t-PK}.

Additionally, as we will point out in Remark~\ref{REM:barvartheta} at the end of Section~\ref{CO:AN:CO:S},
the uniqueness statement in Theorem~\ref{CO:AN:CO} heavily depends on the strict
positivity of the kernel and it fails for kernels that are merely nonnegative.

\subsection{Organization of the paper} 

The rest of the paper is organized as follows. In Section~\ref{SEC:CANCEL}
we provide the proof of the cancellation property stated in Proposition~\ref{cancellazione}.

In Section~\ref{SEC:YOUNG} we prove the nonlocal Young's law in Theorems \ref{thmYounglaw}
and~\ref{thmYounglawa1a2nonconst} and Proposition~\ref{chilosa22}, as well as Corollary~\ref{corollario1}.
Section~\ref{ELLANUEOD32} deals with the possible complete stickiness or detachment of the nonlocal
droplets and it presents the proofs of
Theorems~\ref{ESEMPIOPTO} and~\ref{ESEMPIOPTO:CO}.
Section~\ref{CO:AN:CO:S} is devoted to the
existence and uniqueness theory of the equation prescribing the contact angle and contains
the proof of Theorem~\ref{CO:AN:CO}.

Finally, 
the proof of Proposition~\ref{existenceofminimizer}
is contained in
Appendix~\ref{APPENDI-A}.

\section{The cancellation property in the anisotropic setting
and proof of Proposition \ref{cancellazione}}\label{SEC:CANCEL}

In this section we prove the desired cancellation property stated in Proposition~\ref{cancellazione}.
The argument relies on a delicate analysis of the geometric properties of the integrals involved
in the definition of the function in~\eqref{D}.

\begin{proof}[Proof of Proposition \ref{cancellazione}]
We focus on the proof of~\eqref{XRTS-1}, \eqref{XRTS-190}, \eqref{limit1} and~\eqref{limit2}:
once these statements are proved,
we can conclude that there exists an angle $\widehat{\vartheta}\in (0,2\pi)$ such that $\D(\widehat{\vartheta}) =0$, 
and this angle is unique since~$\mathcal{D}_\vartheta$
is strictly increasing,
thus establishing~\eqref{cancellation}. 

We start by proving \eqref{XRTS-1}. For this, we observe that
the definition in~\eqref{D} has to be interpreted in the principal-value sense, namely
\begin{equation}\label{limit}
\mathcal{D}_\vartheta(\bar{\vartheta})=\lim_{\rho\searrow 0}\left( \int_{J_{\vartheta, \vartheta+\bar{\vartheta}}\setminus B_\rho(e(\vartheta))}\frac{a_1(\overrightarrow{x-e(\vartheta)})}{|x-e(\vartheta)|^{n+s_1}}\,dx-\int_{J_{0, \vartheta}\setminus B_\rho(e(\vartheta))}\frac{a_1(\overrightarrow{x-e(\vartheta)})}{|x-e(\vartheta)|^{n+s_1}}\,dx\right).
\end{equation}
Hence, to establish~\eqref{XRTS-1}, we want to show that
the limit in~\eqref{limit} does exist and is finite. To this end, 
given~$\bar{\vartheta}\in(0,2\pi)$, we let~$\delta:=\min\{\sin\bar{\vartheta}, \sin\vartheta\}$
and we note that~$B_\delta(e(\vartheta))$ is contained in~$J_{0,\vartheta+\bar{\vartheta}}$.
Then, for every $\rho\in(0,\delta]$ we set 
\begin{equation*}\label{fepsilon}
f(\rho):= \int_{J_{\vartheta, \vartheta+\bar{\vartheta}}\setminus B_\rho(e(\vartheta))}\frac{a_1(\overrightarrow{x-e(\vartheta)})}{|x-e(\vartheta)|^{n+s_1}}\,dx-\int_{J_{0, \vartheta}\setminus B_\rho(e(\vartheta))}\frac{a_1(\overrightarrow{x-e(\vartheta)})}{|x-e(\vartheta)|^{n+s_1}}\,dx.
\end{equation*}
We also define~$A_{\delta,\rho}(e(\vartheta)):= B_\delta(e(\vartheta))\setminus B_\rho(e(\vartheta)) $,
see Figure~\ref{FIGURA3}.
By the change of variable~$x\mapsto 2e(\vartheta)-x$, we see that 
\begin{equation*}\label{intanello0}
\begin{split}
&\int_{J_{\vartheta, \vartheta+\bar{\vartheta}} \cap A_{\delta,\rho}(e(\vartheta))}\frac{a_1(\overrightarrow{x-e(\vartheta)})}{|x-e(\vartheta)|^{n+s_1}}\,dx-\int_{J_{0, \vartheta}\cap A_{\delta,\rho}(e(\vartheta))}\frac{a_1(\overrightarrow{x-e(\vartheta)})}{|x-e(\vartheta)|^{n+s_1}}\,dx\\
=&\int_{J_{0, \vartheta}\cap A_{\delta,\rho}(e(\vartheta))}\left[\frac{a_1(\overrightarrow{e(\vartheta)-x})}{|e(\vartheta)-x|^{n+s_1}}-\frac{a_1(\overrightarrow{x-e(\vartheta)})}{|x-e(\vartheta)|^{n+s_1}}\right]\,dx=0 ,
\end{split}
\end{equation*}
since $a_1$ is symmetric. {F}rom this, we deduce that for every $\rho\in(0,\delta]$
\begin{equation*}
f(\rho)-f(\delta)=\int_{J_{\vartheta, \vartheta+\bar{\vartheta}} \cap A_{\delta,\rho}(e(\vartheta))}\frac{a_1(\overrightarrow{x-e(\vartheta)})}{|x-e(\vartheta)|^{n+s_1}}\,dx-\int_{J_{0, \vartheta}\cap A_{\delta,\rho}(e(\vartheta))}\frac{a_1(\overrightarrow{x-e(\vartheta)})}{|x-e(\vartheta)|^{n+s_1}}\,dx=0. 
\end{equation*}

\begin{figure}[h]
\includegraphics[width=0.55\textwidth]{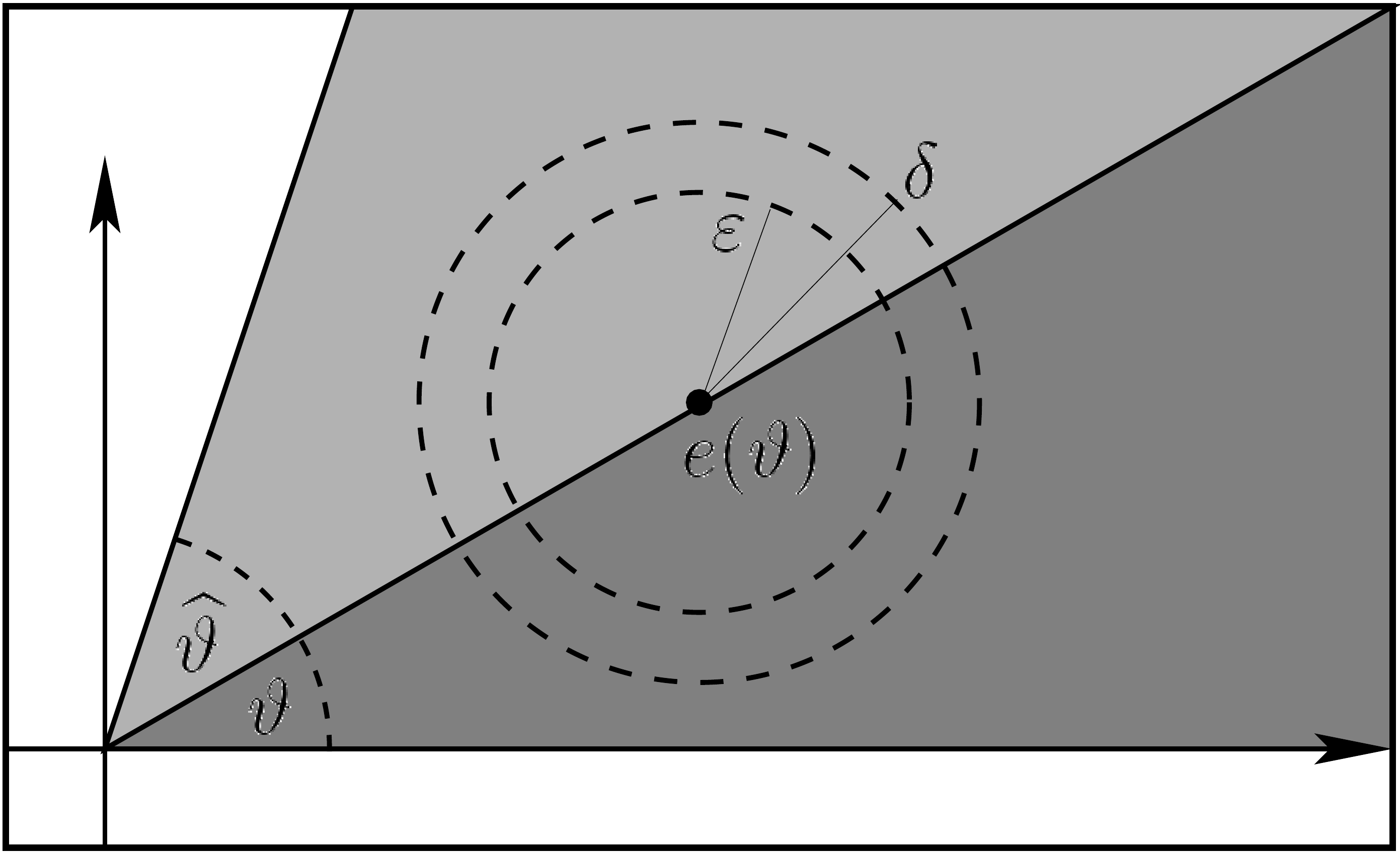}
\caption{The geometry involved in the proof of the existence and finiteness of the limit in~\eqref{limit}.}
        \label{FIGURA3}
\end{figure}

Hence we conclude that
\begin{equation}\label{limitenondipende}
\lim_ {\rho\searrow 0}f(\rho)=f(\delta),
\end{equation}
thus proving the existence and finiteness of the limit in~\eqref{limit}. 

This completes the proof of~\eqref{XRTS-1} and we now focus on the proof of~\eqref{XRTS-190}.

For this, we notice that, if~$\widetilde\vartheta$, $\bar\vartheta\in(0,2\pi)$
with~$\bar\vartheta\ge\widetilde\vartheta$,
\begin{eqnarray*}
\mathcal{D}_\vartheta(\bar{\vartheta})-\mathcal{D}_\vartheta(\widetilde{\vartheta})
&=&
\int_{J_{\vartheta, \vartheta+\bar{\vartheta}}}\frac{a_1(\overrightarrow{x-e(\vartheta)})}{|x-e(\vartheta)|^{n+s_1}}\,dx-
\int_{J_{\vartheta, \vartheta+\widetilde{\vartheta}}}\frac{a_1(\overrightarrow{x-e(\vartheta)})}{|x-e(\vartheta)|^{n+s_1}}\,dx
\\&=&\int_{J_{\vartheta+\widetilde{\vartheta}, \vartheta+\bar{\vartheta}}}\frac{a_1(\overrightarrow{x-e(\vartheta)})}{|x-e(\vartheta)|^{n+s_1}}\,dx
.\end{eqnarray*}
Since the denominator in the latter integral is bounded from below by a positive constant
(depending on~$\widetilde\vartheta$), the claim in~\eqref{XRTS-190} follows from the
Dominated Convergence Theorem.

We now deal with the proof of~\eqref{limit1} and~\eqref{limit2}. To this end, we first prove that
\begin{equation}
\label{AU:LIM}\begin{split}&
\lim_{\varepsilon\searrow0}\left(\int_{J_{\vartheta-\varepsilon, \vartheta}}\frac{a_1(\overrightarrow{x-e(\vartheta)})}{|x-e(\vartheta)|^{n+s_1}}\,dx-
\int_{J_{\vartheta, \vartheta+2\varepsilon}}\frac{a_1(\overrightarrow{x-e(\vartheta)})}{|x-e(\vartheta)|^{n+s_1}}\,dx\right)=-\infty\\
{\mbox{and }}\qquad&\lim_{\varepsilon\searrow0}\left(\int_{J_{\vartheta-2\varepsilon, \vartheta}}\frac{a_1(\overrightarrow{x-e(\vartheta)})}{|x-e(\vartheta)|^{n+s_1}}\,dx-
\int_{J_{\vartheta, \vartheta+\varepsilon}}\frac{a_1(\overrightarrow{x-e(\vartheta)})}{|x-e(\vartheta)|^{n+s_1}}\,dx\right)=+\infty.
\end{split}
\end{equation}

We focus on the proof of the first claim in~\eqref{AU:LIM} since a similar argument would take care of the second one.
For this, let~$\Xi$ be the first limit in~\eqref{AU:LIM} and~$
{\mathcal{R}}$ be the rotation by an angle~$\vartheta$
in the~$(x_1,x_n)$ plane that sends~$e(\vartheta)$ in~$e_1=(1,0,\dots,0)$. Let also~$a_{1,\vartheta}:=
a_1\circ {\mathcal{R}}$ and notice that~$a_{1,\vartheta}$ inherits the properties of~$a_1$, that is
$a_{1,\vartheta}$ is a continuous functions on~$\partial B_1$,
bounded from above and below by two positive constants and satisfying~$a_{1,\vartheta}(\omega)
=a_{1,\vartheta}(-\omega) $
for all~$\omega\in\partial B_1$.

With this notation, we have
\begin{equation} \label{KS-LIMSIEOS} \Xi=
\lim_{\varepsilon\searrow0}\left(\int_{J_{-\varepsilon, 0}}\frac{a_{1,\vartheta}(\overrightarrow{x-e_1})}{|x-e_1|^{n+s_1}}\,dx-
\int_{J_{0, 2\varepsilon}}\frac{a_{1,\vartheta}(\overrightarrow{x-e_1})}{|x-e_1|^{n+s_1}}\,dx\right).\end{equation}
We also remark that, in view of the boundedness of~$a_{1,\vartheta}$,
\begin{eqnarray*}
\int_{J_{-2\varepsilon, 2\varepsilon}\setminus B_2}\frac{a_{1,\vartheta}(\overrightarrow{x-e_1})}{|x-e_1|^{n+s_1}}\,dx
\le\int_{\R^n\setminus B_1}\frac{C}{|y|^{n+s_1}}\,dy\le C,
\end{eqnarray*}
for a suitable constant~$C\ge1$ possibly varying from step to step.

Combining this information with~\eqref{KS-LIMSIEOS} we find that
\begin{equation} \label{KS-LIMSIEOS-2} \Xi\le
\lim_{\varepsilon\searrow0}\left(\int_{J_{-\varepsilon, 0}\cap B_2}\frac{a_{1,\vartheta}(\overrightarrow{x-e_1})}{|x-e_1|^{n+s_1}}\,dx-
\int_{J_{0, 2\varepsilon}\cap B_2}\frac{a_{1,\vartheta}(\overrightarrow{x-e_1})}{|x-e_1|^{n+s_1}}\,dx+C\right).\end{equation}
Now we claim that, if~$\varepsilon$ is sufficiently small,
\begin{equation} \label{KS-LIMSIEOS-3} 
B_{\varepsilon/10}\left( e_1+\frac{3\varepsilon}{2}e_n\right)\subseteq
J_{\varepsilon, 2\varepsilon}\cap B_2.
\end{equation}
To check this, let~$y\in B_{\varepsilon/10}\left( e_1+\frac{3\varepsilon}{2}e_n\right)$. Then,
$$ \frac{\varepsilon^2}{100}\ge |y_1-1|^2+\left|y_n-\frac{3\varepsilon}{2}\right|^2$$
and accordingly~$y_1\in\left[1-\frac\varepsilon{10},1+\frac\varepsilon{10}\right]$
and~$y_n\in\left[\frac{7\varepsilon}{5},\frac{8\varepsilon}{5}\right]$.
As a consequence, if~$\varepsilon$ is conveniently small,
$$ \frac{y_n}{y_1}\in
\left[\frac{\frac{7\varepsilon}{5}}{1+\frac\varepsilon{10}},\frac{\frac{8\varepsilon}{5}}{1-\frac\varepsilon{10}}\right]\subseteq
\left[\frac{6\varepsilon}{5},\frac{9\varepsilon}{5}\right]\subseteq[\tan\varepsilon,\tan(2\varepsilon)],
$$
which, recalling~\eqref{Jtheta1theta2},
establishes~\eqref{KS-LIMSIEOS-3}.

Using~\eqref{KS-LIMSIEOS-3} and the assumption that~$a_{1,\vartheta}$ is bounded from below
away from zero, we obtain that
\begin{eqnarray*}
\int_{J_{\varepsilon, 2\varepsilon}\cap B_2}\frac{a_{1,\vartheta}(\overrightarrow{x-e_1})}{|x-e_1|^{n+s_1}}\,dx\ge
\frac{1}{C}
\int_{B_{\varepsilon/10}\left( e_1+\frac{3\varepsilon}{2}e_n\right)}\frac{dx}{|x-e_1|^{n+s_1}}\ge\frac{1}{C\varepsilon^{s_1}}.
\end{eqnarray*}
This and~\eqref{KS-LIMSIEOS-2} entail that
\begin{equation}\label{LMMDNSMANDNDIKJE} \Xi\le
\lim_{\varepsilon\searrow0}\left(\int_{J_{-\varepsilon, 0}\cap B_2}\frac{a_{1,\vartheta}(\overrightarrow{x-e_1})}{|x-e_1|^{n+s_1}}\,dx-
\int_{J_{0, \varepsilon}\cap B_2}\frac{a_{1,\vartheta}(\overrightarrow{x-e_1})}{|x-e_1|^{n+s_1}}\,dx-\frac{1}{C\varepsilon^{s_1}}+C\right).\end{equation}
Now we observe that
$$ J_{-\varepsilon,\varepsilon}=
\big\{x\in \mathbb{R}^n\,:\,|x_n|<\tan\varepsilon \,x_1\big\}$$
and we define
\begin{eqnarray*}&&
J^\sharp_\varepsilon:=2e_1-J_{-\varepsilon,\varepsilon},\\&&
R_{\varepsilon}:=J_{-\varepsilon,\varepsilon}\cap J^\sharp_\varepsilon,\\&&
J_\varepsilon^\star:=J_{0,\varepsilon}\setminus R_{\varepsilon}\\
{\mbox{and }}&& {\mathcal{B}}_{\varepsilon}:=
\left\{x\in J_\varepsilon^\star\,:\,x_n>\frac{x_1-1}{|\log\varepsilon|}\right\},
\end{eqnarray*}
see Figure~\ref{FIGURA00300}.

\begin{figure}[h]
\includegraphics[width=0.65\textwidth]{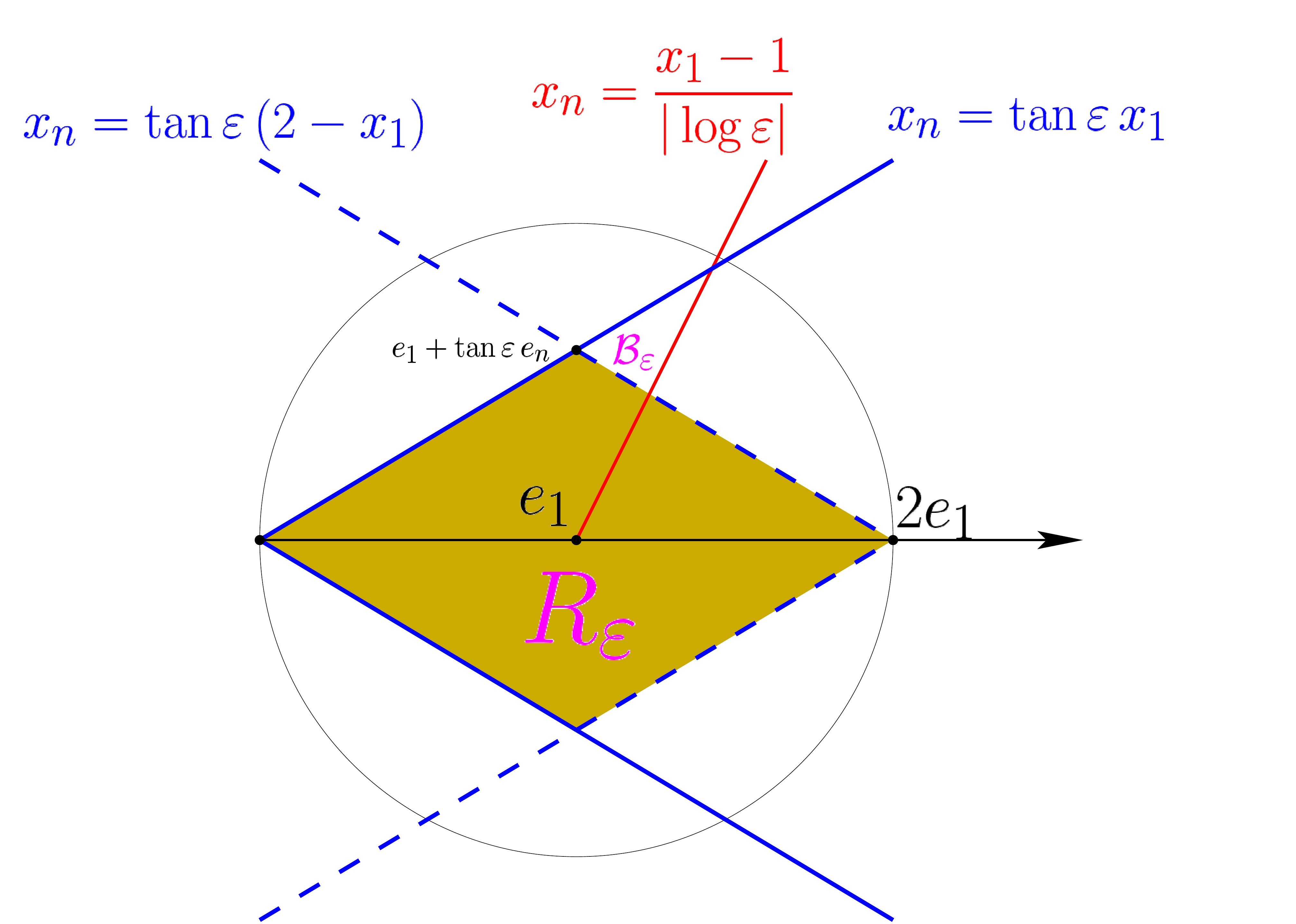}
\caption{The set decomposition involved in the proof of~\eqref{AU:LIM}.}
        \label{FIGURA00300}
\end{figure}

The intuition behind this set decomposition
is that, on the one hand, the set~$R_{\varepsilon}$
accounts for the cancellations due to the symmetry of~$a_{1,\vartheta}$
(corresponding to the reflection through~$e_1$, namely~$x\mapsto2e_1-x$);
on the other hand, the remaining integral contributions
in~$J_\varepsilon^\star$ cancel exactly when~$a_{1,\vartheta}$ is constant,
thanks to the reflection through the horizontal hyperplane~$x\mapsto(x',-x_n)$,
but they may provide additional terms when~$a_{1,\vartheta}$ is not constant.
To overcome this difficulty, our idea is to exploit the continuity
of~$a_{1,\vartheta}$ together with the reflection through the horizontal hyperplane
in order to ``approximately cancel'' as many contributions as possible.

This idea by itself however does not exhaust the complexity of the problem,
because two adjacent points can end up being projected far away from each other
on the sphere (for instance, if a point is close to~$e_1+\tan\varepsilon\,e_n$
and the other to~$e_1-\tan\varepsilon\,e_n$). To overcome this additional
complication, we exploit the set~${\mathcal{B}}_{\varepsilon}$: roughly
speaking, points outside~${\mathcal{B}}_{\varepsilon}$ remain sufficiently
close after they get projected
on the sphere (and here we can take advantage of the continuity of~$a_{1,\vartheta}$),
while the points in~${\mathcal{B}}_{\varepsilon}$
provide an additional, but small, correction, in view of
the location of~${\mathcal{B}}_{\varepsilon}$ and of its measure.

The details of the quantitative computation needed to
implement this combination of ideas go as follows.

We stress that 
\begin{equation}\label{SU78SL}
{\mbox{if~$x$ belongs to~$R_{\varepsilon}$, then so does~$2e_1-x$.}}\end{equation}
Indeed, if~$x\in R_{\varepsilon}$ then~$x\in J_{-\varepsilon,\varepsilon}$
and~$x\in 2e_1-J_{-\varepsilon,\varepsilon}$, and consequently~$2e_1-x\in 2e_1-J_{-\varepsilon,\varepsilon}$
and~$2e_1-x\in J_{-\varepsilon,\varepsilon}$, which gives~\eqref{SU78SL}.

Also, we see that~$J_{-\varepsilon, 0}\cap R_{\varepsilon}=R_{\varepsilon}\cap\{x_n<0\}$
and~$J_{0,\varepsilon}\cap R_{\varepsilon}=R_{\varepsilon}\cap\{x_n>0\}$.
Thus, using~\eqref{SU78SL}, the change of variable~$x\mapsto 2e_1-x$ and the symmetry of~$a_{1,\vartheta}$, taking into account that under this transformation some vectors end up outside the ball $B_2$,
\begin{equation*}
\begin{split}
\int_{J_{-\varepsilon, 0}\cap B_2\cap R_{\varepsilon}}\frac{
a_{1,\vartheta}(\overrightarrow{x-e_1})}{|x-e_1|^{n+s_1}}\,dx&=
\int_{R_{\varepsilon}\cap\{x_n<0\}\cap B_2}\frac{a_{1,\vartheta}(\overrightarrow{x-e_1})}{|x-e_1|^{n+s_1}}\,dx
\\
&\le\int_{R_{\varepsilon}\cap\{x_n>0\}\cap B_2}\frac{a_{1,\vartheta}(\overrightarrow{e_1-x})}{
|e_1-x|^{n+s_1}}\,dx+C\int_{\R^n\setminus B_2}\frac{dx}{|x-e_1|^{n+s_1}}\\
&\le\int_{J_{0,\varepsilon}\cap B_2\cap R_{\varepsilon}}\frac{a_{1,\vartheta}(\overrightarrow{x-e_1})}{|x-e_1|^{n+s_1}}\,dx+C.
\end{split}
\end{equation*}
Plugging this cancellation into~\eqref{LMMDNSMANDNDIKJE}, we conclude that
\begin{equation*} \Xi\le
\lim_{\varepsilon\searrow0}\left(\int_{(J_{-\varepsilon, 0}\cap B_2)\setminus R_\varepsilon}\frac{a_{1,\vartheta}(\overrightarrow{x-e_1})}{|x-e_1|^{n+s_1}}\,dx-
\int_{(J_{0, \varepsilon}\cap B_2)\setminus R_\varepsilon}\frac{a_{1,\vartheta}(\overrightarrow{x-e_1})}{|x-e_1|^{n+s_1}}\,dx-\frac{1}{C\varepsilon^{s_1}}+C\right).\end{equation*}
Using the change of variable~$x\mapsto(x',-x_n)$ and noticing that~$|(x',-x_n)-e_1|=|x-e_1|$, we thus find that
\begin{equation} \label{KS-LIMSIEOS-4} \begin{split}\Xi\,&\le\,
\lim_{\varepsilon\searrow0}\left(
\int_{(J_{0, \varepsilon}\cap B_2)\setminus R_\varepsilon}\frac{
a_{1,\vartheta}(\overrightarrow{(x',-x_n)-e_1})-
a_{1,\vartheta}(\overrightarrow{x-e_1})}{|x-e_1|^{n+s_1}}\,dx-\frac{1}{C\varepsilon^{s_1}}+C\right)\\
&=\,
\lim_{\varepsilon\searrow0}\left(
\int_{J^\star_{\varepsilon}\cap B_2}\frac{
a_{1,\vartheta}(\overrightarrow{(x',-x_n)-e_1})-
a_{1,\vartheta}(\overrightarrow{x-e_1}) }{|x-e_1|^{n+s_1}}\,dx-\frac{1}{C\varepsilon^{s_1}}+C\right).
\end{split}\end{equation}
We point out that
\begin{equation}\label{DENT5-009}
J_\varepsilon^\star\subseteq\{ x_1\ge1\}.
\end{equation}
Indeed, if~$x\in J_\varepsilon^\star$, then~$x\in
J_{0,\varepsilon}$, whence
\begin{equation}\label{DENT5-009-1}
x_n\in\left(0,\tan\varepsilon x_1\right).\end{equation}
Also, we have that~$x\not\in R_\varepsilon$ and therefore
either~$x\not\in J_{-\varepsilon,\varepsilon}$ or~$x\not\in J^\sharp_\varepsilon$.
In fact, since~$J_{0,\varepsilon}\subseteq J_{-\varepsilon,\varepsilon}$,
we have that necessarily~$x\not\in J^\sharp_\varepsilon$,
and, as a result, $2e_1-x\not\in J_{-\varepsilon,\varepsilon}$.
This gives that~$|x_n|\ge\tan\varepsilon\,(2-x_1)$. Therefore, by~\eqref{DENT5-009-1},
\begin{equation}\label{211BIS}
2-x_1\le\frac{|x_n|}{\tan\varepsilon}=\frac{ x_n }{\tan\varepsilon}\le x_1,
\end{equation}
and this entails~\eqref{DENT5-009}.

Now we claim that
\begin{equation}\label{k-s-PKSMD-0}
{\mathcal{B}}_{\varepsilon}\subseteq
\left\{x\in \R^n\,:\, |x_1-1|\le 2\varepsilon\,|\log\varepsilon|,\,
|x_n-\tan\varepsilon|\le2\varepsilon^2\,|\log\varepsilon|
\right\}.
\end{equation}
To check this, let~$x\in{\mathcal{B}}_{\varepsilon}$. Then,
\begin{equation}\label{k-s-PKSMD-1} \frac{x_1-1}{|\log\varepsilon|}\le x_n\le\tan\varepsilon\,x_1=
\tan\varepsilon\,(x_1-1)+\tan\varepsilon.
\end{equation}
Recalling~\eqref{DENT5-009}, we thus find that
$$
\left( \frac{1}{|\log\varepsilon|}-\tan\varepsilon\right)|
x_1-1|=
\left( \frac{1}{|\log\varepsilon|}-\tan\varepsilon\right)(
x_1-1)\le\tan\varepsilon.$$
Consequently, if~$\varepsilon$ is conveniently small,
\begin{equation} \label{k-s-PKSMD}\frac{9}{10}\,|x_1-1|\le
\left( 1-\tan\varepsilon\,|\log\varepsilon|\right)|
x_1-1|\le\tan\varepsilon\,|\log\varepsilon|\le
\frac{11}{10}\,\varepsilon\,|\log\varepsilon|.\end{equation}
Furthermore, by the second inequality in~\eqref{k-s-PKSMD-1}
and~\eqref{k-s-PKSMD},
\begin{equation} \label{k-s-PKSMD-8} x_n-\tan\varepsilon\le\tan\varepsilon\,( x_1-1) \le
\tan\varepsilon\,|x_1-1|\le\frac{11}9\,
\varepsilon\,\tan\varepsilon\,|\log\varepsilon|\le2\varepsilon^2\,|\log\varepsilon|.
\end{equation}
Moreover, from~\eqref{211BIS},
$$ x_n\ge\tan\varepsilon\,(2-x_1),$$
whence, utilizing again~\eqref{k-s-PKSMD},
$$ \tan\varepsilon-x_n\le\tan\varepsilon+
\tan\varepsilon\,(x_1-2)\le \tan\varepsilon\,|x_1-1|\le
\frac{11}9\,
\varepsilon\,\tan\varepsilon\,|\log\varepsilon|\le2\varepsilon^2\,|\log\varepsilon|.
$$
{F}rom this, \eqref{k-s-PKSMD} and~\eqref{k-s-PKSMD-8}
we obtain~\eqref{k-s-PKSMD-0}, as desired.

Now, using~\eqref{k-s-PKSMD-0} and the changes of variable~$y:=\frac{x-e_1}{\tan\varepsilon}$
and~$z:=\left(\frac{y'}{|y_n|},y_n\right)$,
we see that
\begin{equation}\label{CABMSFIDMNLAMDXVA}
\begin{split}
\int_{{\mathcal{B}}_{\varepsilon}}\frac{dx}{|x-e_1|^{n+s_1}}\,&\leq\,
\int_{ {\{|x_1-1|\le 2\varepsilon\,|\log\varepsilon|\}}\atop{\{
|x_n-\tan\varepsilon|\le2\varepsilon^2\,|\log\varepsilon|\}}}\frac{dx}{|x-e_1|^{n+s_1}}\\&
=\,\frac{1}{(\tan\varepsilon)^{s_1}}
\int_{ {\{|y_1|\le 2\varepsilon\,|\log\varepsilon|/\tan\varepsilon\}}\atop{\{
|y_n-1|\le2\varepsilon^2\,|\log\varepsilon|/\tan\varepsilon\}}}\frac{dy}{|y|^{n+s_1}}
\\&
\le\,\frac{2}{ \varepsilon^{s_1}}
\int_{ \{
|y_n-1|\le4\varepsilon\,|\log\varepsilon|\}}\frac{dy}{|y|^{n+s_1}}\\&
=\,\frac{2}{ \varepsilon^{s_1}}
\int_{ \{
|z_n-1|\le4\varepsilon\,|\log\varepsilon|\}}\frac{dz}{|z_n|^{1+s_1}\big(|z'|^2+1\big)^{\frac{n+s_1}2}}\\
&\le\,\frac{C}{ \varepsilon^{s_1}}
\int_{ 1-4\varepsilon\,|\log\varepsilon|}^{ 1+4\varepsilon\,|\log\varepsilon|}\frac{dz_n}{z_n^{1+s_1}}\\&\le\,
C\varepsilon^{1-s_1}|\log\varepsilon|,
\end{split}\end{equation}
up to renaming the positive constant~$C$
line after line.

We also recall that~$|(x',-x_n)-e_1|=|x-e_1|$
and accordingly
\begin{equation}\label{petehitgherutu7y672223}
\left|\overrightarrow{(x',-x_n)-e_1}- \overrightarrow{x-e_1}\right|=
\frac{\left| ((x',-x_n)-e_1)-(x-e_1)\right|}{|x-e_1|}=\frac{2|x_n|}{|x-e_1|}.
\end{equation}
As a result, recalling~\eqref{DENT5-009} and~\eqref{DENT5-009-1},
if~$x\in J^\star_{\varepsilon}\setminus {\mathcal{B}}_{\varepsilon}$ then
$$|x_n|=x_n\le\frac{x_1-1}{|\log\varepsilon|}=\frac{|x_1-1|}{|\log\varepsilon|}.$$
This and~\eqref{petehitgherutu7y672223} give that
\begin{equation*}
\left|\overrightarrow{(x',-x_n)-e_1}- \overrightarrow{x-e_1}\right|\le
\frac{2}{|\log\varepsilon|}.
\end{equation*}
Consequently, if we consider the modulus of continuity of~$a_{1,\vartheta}$, namely
$$ \sigma(t):=\sup_{{v,w\in\partial B_1}\atop{|v-w|\le t}}|a_{1,\vartheta}(v)-a_{1,\vartheta}(w)|,$$
we deduce that
if~$x\in J^\star_{\varepsilon}\setminus {\mathcal{B}}_{\varepsilon}$ then
$$ \big|a_{1,\vartheta}(\overrightarrow{(x',-x_n)-e_1})-
a_{1,\vartheta}(\overrightarrow{x-e_1})\big|\le\sigma\left(\frac{2}{|\log\varepsilon|}\right)$$
and thus
\begin{equation}\label{KMD:9876tg:SODK:OCKMSax4XZXX}
\int_{J^\star_{\varepsilon}\setminus {\mathcal{B}}_{\varepsilon}}\frac{
a_{1,\vartheta}(\overrightarrow{(x',-x_n)-e_1})-
a_{1,\vartheta}(\overrightarrow{x-e_1}) }{|x-e_1|^{n+s_1}}\,dx\le
\sigma\left(\frac{2}{|\log\varepsilon|}\right)
\int_{J^\star_{\varepsilon}\setminus {\mathcal{B}}_{\varepsilon}}\frac{dx }{|x-e_1|^{n+s_1}}.
\end{equation}
Notice also that
\begin{equation}\label{KMD:9876tg:SODK:OCKMSax4XZ}
J^\star_{\varepsilon}\subseteq\R^n\setminus B_{\varepsilon/100}(e_1).
\end{equation}
Indeed, if~$x\in J^\star_{\varepsilon}$, from~\eqref{211BIS} we have that
\begin{equation}\label{KMD:9876tg:SODK:OCKMSax4XZ2}
|x_n|=x_n\ge\tan\varepsilon(2-x_1).\end{equation}
Now, if~$x_1\ge \frac{19}{10}$, then~$|x-e_1|\ge|x_1-1|\ge\frac{9}{10}$
and~\eqref{KMD:9876tg:SODK:OCKMSax4XZ} plainly follows; if instead~$
x_1< \frac{19}{10}$ we deduce from~\eqref{KMD:9876tg:SODK:OCKMSax4XZ2} that
$$ |x-e_1|\ge|x_n|\ge\frac{\tan\varepsilon}{10},$$
from which~\eqref{KMD:9876tg:SODK:OCKMSax4XZ} follows in this case too.

By combining~\eqref{KMD:9876tg:SODK:OCKMSax4XZXX}
and~\eqref{KMD:9876tg:SODK:OCKMSax4XZ} we deduce that
\begin{eqnarray*} \int_{J^\star_{\varepsilon}\setminus {\mathcal{B}}_{\varepsilon}}\frac{
a_{1,\vartheta}(\overrightarrow{(x',-x_n)-e_1})-
a_{1,\vartheta}(\overrightarrow{x-e_1}) }{|x-e_1|^{n+s_1}}\,dx&\le&
\sigma\left(\frac{2}{|\log\varepsilon|}\right)
\int_{\R^n\setminus B_{\varepsilon/100}}\frac{dy }{|y|^{n+s_1}}\\&\le&
\frac{C}{\varepsilon^{s_1}}\,
\sigma\left(\frac{2}{|\log\varepsilon|}\right),
\end{eqnarray*}
which together with~\eqref{CABMSFIDMNLAMDXVA} leads to
$$ \int_{J^\star_{\varepsilon}}\frac{
a_{1,\vartheta}(\overrightarrow{(x',-x_n)-e_1})-
a_{1,\vartheta}(\overrightarrow{x-e_1}) }{|x-e_1|^{n+s_1}}\,dx
\le \frac{C}{\varepsilon^{s_1}}\,
\sigma\left(\frac{2}{|\log\varepsilon|}\right)
+C\varepsilon^{1-s_1}|\log\varepsilon|.$$
Joining this information with~\eqref{KS-LIMSIEOS-4} we find that
\begin{eqnarray*}
\Xi&\le&
\lim_{\varepsilon\searrow0}\left[
\frac{C}{\varepsilon^{s_1}}\,
\sigma\left(\frac{2}{|\log\varepsilon|}\right)
+C\varepsilon^{1-s_1}|\log\varepsilon|
-\frac{1}{C\varepsilon^{s_1}}+C\right]
\\&=&
\lim_{\varepsilon\searrow0}
\frac{1}{\varepsilon^{s_1}}\,\left[
C\sigma\left(\frac{2}{|\log\varepsilon|}\right)
+C\varepsilon\,|\log\varepsilon|
-\frac{1}{C}+C\varepsilon^{s_1}\right]
\\&\le&
\lim_{\varepsilon\searrow0}\left(
-\frac{1}{2C\varepsilon^{s_1}}\right)
\\&=&-\infty
\end{eqnarray*}
This completes the proof of~\eqref{AU:LIM}.

Now, using~\eqref{AU:LIM},
\begin{eqnarray*}&&
\lim_{\bar{\vartheta}\searrow 0}\mathcal{D}_\vartheta(\bar{\vartheta})\\&=&\lim_{\bar{\vartheta}\searrow 0}
\left(\int_{J_{\vartheta,\vartheta+\bar{\vartheta}}}\frac{
a_1(\overrightarrow{x-e(\vartheta)})}{|x-e(\vartheta)|^{n+s_1}}\, dx
-\int_{J_{\vartheta-2\bar{\vartheta},\vartheta}}\frac{a_1(\overrightarrow{x-e(\vartheta)})}{|x-e(\vartheta)|^{n+s_1}}\, dx-
\int_{J_{0,\vartheta-2\bar{\vartheta}} }\frac{a_1(\overrightarrow{x-e(\vartheta)})}{|x-e(\vartheta)|^{n+s_1}}\, dx\right)
\\&
\le&\lim_{\bar{\vartheta}\searrow 0}
\left(\int_{J_{\vartheta,\vartheta+\bar{\vartheta}}}\frac{
a_1(\overrightarrow{x-e(\vartheta)})}{|x-e(\vartheta)|^{n+s_1}}\, dx
-\int_{J_{\vartheta-2\bar{\vartheta},\vartheta}}\frac{a_1(\overrightarrow{x-e(\vartheta)})}{|x-e(\vartheta)|^{n+s_1}}\, dx\right)
=-\infty,
\end{eqnarray*}
which proves~\eqref{limit1}, and
\begin{eqnarray*}&&
\lim_{\bar{\vartheta}\nearrow 2\pi}
\mathcal{D}_\vartheta(\bar{\vartheta})
=\lim_{\alpha\searrow 0}
\mathcal{D}_\vartheta(2\pi-\alpha)
=\lim_{\alpha\searrow 0}\left(
\int_{J_{\vartheta,\vartheta+2\pi-\alpha}}
\frac{a_1(\overrightarrow{x-e(\vartheta)})}{|x-e(\vartheta)|^{n+s_1}}\, dx
-\int_{J_{0,\vartheta}}\frac{a_1(\overrightarrow{x-e(\vartheta)})}{
|x-e(\vartheta)|^{n+s_1}}\, dx\right)\\&&\qquad =\lim_{\alpha\searrow 0}\left(
\int_{J_{\vartheta, 2\pi}}
\frac{a_1(\overrightarrow{x-e(\vartheta)})}{|x-e(\vartheta)|^{n+s_1}}\, dx
-\int_{J_{\vartheta-\alpha,\vartheta}}\frac{a_1(\overrightarrow{x-e(\vartheta)})}{
|x-e(\vartheta)|^{n+s_1}}\, dx\right)
\\&&\qquad \geq
\lim_{\alpha\searrow 0}\left(
\int_{J_{\vartheta, \vartheta+2\alpha}}
\frac{a_1(\overrightarrow{x-e(\vartheta)})}{|x-e(\vartheta)|^{n+s_1}}\, dx
-\int_{J_{\vartheta-\alpha,\vartheta}}\frac{a_1(\overrightarrow{x-e(\vartheta)})}{
|x-e(\vartheta)|^{n+s_1}}\, dx\right)=+\infty,
\end{eqnarray*}
which proves~\eqref{limit2}.
\end{proof}

\section{Nonlocal Young's law and proofs of Theorems \ref{thmYounglaw} and~\ref{thmYounglawa1a2nonconst}, of Corollary~\ref{corollario1} and of Proposition~\ref{chilosa22},
}\label{SEC:YOUNG}

In order to prove Theorems \ref{thmYounglaw}
and~\ref{thmYounglawa1a2nonconst},
Corollary~\ref{corollario1} and Proposition~\ref{chilosa22}, we first
recall an ancillary result on the continuity of the nonlocal $K$-mean curvature defined
in~\eqref{meancurvature}
(for the usual fractional mean curvature, that is when the kernel~$K$ is as in~\eqref{fractkernel},
similar continuity results were presented in~\cite{MR3322379, MR3393254}).

{F}rom now on,
we denote points~$x\in\R^n$ as~$x=(x',x_n)\in\R^{n-1}\times\R$ and we set
\begin{equation*}
\begin{split}
\mathbf{C}&:=\{x=(x',x_n)\in\R^n:|x'|<1, |x_n|<1\}\\
{\mbox{and }}\quad \mathbf{D}&:=\{z\in\R^{n-1}:|z|<1\}.
\end{split}
\end{equation*}

\begin{lem}\label{lemmadiconv}
Let $\lambda\geq 1$, $s\in (0,1)$ and~$\alpha\in (s,1)$.
Let $\{F_k\}_{k\in\mathbb{N}}$ be a sequence of Borel sets in $\R^n$ such that $0\in \partial F_k$ and
\begin{equation*}
\text{$F_k\rightarrow F$ in $L^1_{\mathrm{loc}}(\R^n)$ for some $F\subseteq \R^n$}.
\end{equation*}
and $u_k$, $u\in C^{1,\alpha}(\R^{n-1})$ be such that
\begin{equation*}
\mathbf{C}\cap F_k=\{x\in \mathbf{C}:x_n\leq u_k(x')\}
\end{equation*}
and
\begin{equation*}
\lim_{k\rightarrow +\infty}\Vert u_k-u\Vert_{C^{1,\alpha}(\mathbf{D})}=0.
\end{equation*}
Let $K_k$, $K\in \mathbf{K}(n,s,\lambda,0)$ be
such that~$K_k\rightarrow K$ pointwise in $\R^n\setminus\{0\}$ as $k\rightarrow +\infty$.

Then 
\begin{equation*}
\lim_{k\rightarrow +\infty}\mathbf{H}^{K_k}_{\partial F_k}(0)=\mathbf{H}^{K}_{\partial F}(0).
\end{equation*}
\end{lem}

For the proof of Lemma~\ref{lemmadiconv} here, see Lemma~4.1 in~\cite{MR3717439}.
\medskip

We will also need a technical lemma to distinguish between the nondegenerate case~$\vartheta\in(0,\pi)$
and the
particular cases in which~$\vartheta\in\{0,\pi\}$.

\begin{lem}\label{angolettilemma}
Let~$K_1\in\mathbf{K}^2(n,s_1,\lambda,\varrho)$ be such that
it admits a
blow-up limit~$K_1^\ast$
(according to~\eqref{Kast}).
Let $\Omega$ be an open bounded set with $C^1$-boundary and $E$ be a volume-constrained critical set of $\mathcal{C}$. 

Let $x_0\in \mathrm{Reg}_E\cap \partial\Omega$, $x_k\in\mathrm{Reg}_E\cap \Omega$
such that~$x_k\to x_0$ as~$k\to+\infty$
and~$r_k>0$ such that~$r_k\to0$ as~$k\to+\infty$.

Suppose that $H$ and $V$ are open
half-spaces such that
\begin{equation}\label{ndeotuerunfsdjkd34748768947068}
\Omega^{x_0,r_k}\rightarrow H\qquad {\mbox{and}}\qquad
E^{x_0,r_k}\rightarrow H\cap V\quad {\mbox{in }} L^1_\mathrm{loc}(\R^n)
\;{\mbox{as }}k\to +\infty.
\end{equation}

Set~$v_k:=\frac{x_k-x_0}{r_k}$ and suppose that there exists~$v\in H\cap\partial V$ such
that~$v_k\to v$ as~$k\to+\infty$.

Let~$\vartheta\in [0,\pi]$ be the angle between
the half-spaces $H$ and $V$, that is $H\cap V=J_{0,\vartheta}$ in the notation of~\eqref{Jtheta1theta2}.

Then, 
\begin{itemize}
\item[i)] if~$\vartheta=0$ then
$$ \lim_{k\to+\infty}r_k^{s_1}
\left[\mathbf{H}^{K_1}_{\partial E} (x_k)-\int_{\Omega^c}K_1(x_k-y)\,dy\right]=+\infty;
$$
\item[ii)] if~$\vartheta=\pi$ then
$$ \lim_{k\to+\infty}r_k^{s_1}
\left[\mathbf{H}^{K_1}_{\partial E} (x_k)-\int_{\Omega^c}K_1(x_k-y)\,dy\right]=-\infty;
$$
 \item[iii)] if~$\vartheta\in(0,\pi)$ then
$$ \lim_{k\to+\infty}r_k^{s_1}
\left[\mathbf{H}^{K_1}_{\partial E} (x_k)- \int_{\Omega^c}K_1(x_k-y)\,dy\right]=
\mathbf{H}^{K_1^\ast}_{\partial (H\cap V)} (v)-\int_{H^c}K_1^\ast (v-y)\,dy \in\R.
$$
\end{itemize}
\end{lem}

\begin{proof}
We start by proving~i). For this, we notice that
\begin{eqnarray*}
\Xi_k &:=& r_k^{s_1}\left[\mathbf{H}^{K_1}_{\partial E} (x_k)-\int_{\Omega^c}K_1(x_k-y)\,dy\right]\\
&=& r_k^{s_1}\left[\int_{E^c\cap\Omega}K_1(x_k-y)\,dy
-\int_{E}K_1(x_k-y)\,dy\right]\\
&=&r_k^{n+s_1}\left[\int_{(E^{x_0,r_k})^c\cap\Omega^{x_0,r_k}}K_1(x_k-x_0-r_k z)\,dz
-\int_{E^{x_0,r_k}}K_1(x_k-x_0-r_k z)\,dz\right]\\&=&
r_k^{n+s_1}\left[\int_{(E^{x_0,r_k})^c\cap\Omega^{x_0,r_k}}K_1\big(r_k(v_k- z)\big)\,dz
-\int_{E^{x_0,r_k}}K_1\big(r_k(v_k- z)\big)\,dz\right]
,\end{eqnarray*}
where the change of variable~$z=\frac{y-x_0}{r_k}$ has been used.

Now we point out that
\begin{equation*}
r_k^{n+s_1}
\int_{\R^n\setminus B_{1/2}(v_k)}K_1\big(r_k(v_k- z)\big)\,dz\le 
\lambda \int_{\R^n\setminus B_{1/2}(v_k)}\frac{dz}{|v_k- z|^{n+s_1}}\le C,
\end{equation*}
thanks to~\eqref{stimakernel}, for some positive constant~$C$, depending only on~$n$, $s_1$ and~$\lambda$.

{F}rom these observations we conclude that
\begin{equation}\begin{split}
\Xi_k \ge\,& 
r_k^{n+s_1}\left[\int_{(E^{x_0,r_k})^c\cap\Omega^{x_0,r_k} \cap B_{1/2}(v_k)}K_1\big(r_k(v_k- z)\big)\,dz\right.
\\&\qquad\qquad\left.
-\int_{E^{x_0,r_k}\cap B_{1/2}(v_k)}K_1\big(r_k(v_k- z)\big)\,dz\right] -C.\label{sjiwhweuiheirbger00}
\end{split}\end{equation}
Now we notice that~$E^{x_0,r_k} \cap B_{1/2}(v_k)$ can be written as a portion of space
included between the graphs of the functions describing~$\partial \Omega^{x_0,r_k}$
and~$\partial E^{x_0,r_k}$, that we denote respectively by~$\psi_k$ and~$u_k$. More precisely,
recalling that~$x_0\in\mathrm{Reg}_E\cap \partial\Omega$, in the
vicinity of~$x_0$ we can describe~$\partial \Omega$
and~$\partial E$ by the graphs of two functions~$\psi$ and~$u$, respectively, with~$\psi$
of class~$C^1$ and~$u$ of class~$C^{1,\alpha}$ with~$\alpha\in(s_1,1)$, and~$\psi(x_0')=u(x_0')=x_{0,n}$.
Up to a rotation, we also assume that~$\nabla\psi(x_0')=0$. 
In this way, 
\begin{equation}\label{sjwircbuy47uopoiuytredgjtrey}
\psi_k(x')=\frac{\psi(x_0'+r_k x')-x_{0,n}}{r_k} \qquad{\mbox{ and }}\qquad
u_k(x')=\frac{u(x_0'+r_k x')-x_{0,n}}{r_k}.
\end{equation}
Moreover, 
\begin{eqnarray*}&&
E^{x_0,r_k}\cap B_{1/2}(v_k)
= \Big\{x\in B_{1/2}(v_k)\,:\, x_n\in\big(\psi_k(x'), u_k(x')\big)
\Big\}
\end{eqnarray*}
and notice that, since~$E\subseteq\Omega$, it follows that~$\psi\le u$ and so~$\psi_k\le u_k$. As a result,
$$ \big\{ x\in B_{1/2}(v_k)\,:\, x_n> u_k(x')\big\}\subseteq (E^{x_0,r_k})^c\cap \Omega^{x_0,r_k}
\cap B_{1/2}(v_k).
$$
Hence, from~\eqref{sjiwhweuiheirbger00} we obtain that
\begin{equation}\begin{split}
\Xi_k \ge\,& 
r_k^{n+s_1}\left[\int_{
B_{1/2}(v_k)\cap \{ x_n> u_k(x')\}
}K_1\big(r_k(v_k- z)\big)\,dz
\right.\\&\qquad\qquad\left.-\int_{ B_{1/2}(v_k)\cap\{ x_n\in(\psi_k(x'), u_k(x'))\}}K_1\big(r_k(v_k- z)\big)\,dz\right]-C.
\label{sjiwhweuiheirbger}
\end{split}\end{equation}

We now define
$$ \widetilde{u}_k(x'):=u_k(v'_k) +\nabla u_k(v_k')\cdot(x'-v_k') 
$$
and we point out that, if~$|x'-v_k'|\le3$,
\begin{eqnarray*}
| u_k(x')-\widetilde{u}_k(x')|&=&\left|\frac{u(x_0'+r_k x')-u(x_0'+r_k v_k')}{r_k}-
\nabla u(x_0'+r_k v_k')\cdot (x'-v_k')\right|\\&=&
\left|\frac{u\big(x_k'+r_k (x'-v_k')\big)-u(x_k')}{r_k}-
\nabla u(x_k')\cdot (x'-v_k')\right|\\&=&
\left|\int_0^1 \nabla u\big(x_k'+t r_k (x'-v_k')\big)\cdot(x'-v_k')\,dt-
\nabla u(x_k')\cdot (x'-v_k')\right|\\&\le& \|u\|_{C^{1,\alpha}(B'_\rho(x'_0))}\, r_k^\alpha\,|x'-v_k'|^{1+\alpha},
\end{eqnarray*}
for a suitable~$\rho>0$. As a consequence, 
\begin{eqnarray*}
&& r_k^{n+s_1}\int_{(\{x_n> u_k(x')\}
\Delta \{x_n>\widetilde{u}_k(x')\})\cap B_{1/2}(v_k)}K_1\big(r_k(v_k- z)\big)\,dz\\&&\qquad\qquad\le
\lambda\int_{(\{x_n> u_k(x')\}
\Delta \{x_n>\widetilde{u}_k(x')\})\cap B_{1/2}(v_k)}\frac{dz}{|v_k- z|^{n+s_1}}\\&&\qquad
\qquad \le \lambda  \|u\|_{C^{1,\alpha}(B'_\rho(x'_0))}\, r_k^\alpha
\int_{B'_{1/2}(v'_k)}\frac{|v_k'-z'|^{1+\alpha}}{|v_k'- z'|^{n+s_1}}\,dz'
\\&&\qquad\qquad\le C \,r_k^\alpha,
\end{eqnarray*}
up to renaming~$C$, possibly in dependence of~$u$ as well.

Plugging this information into~\eqref{sjiwhweuiheirbger}, and possibly renaming~$C$ again, we obtain that
\begin{equation}\begin{split}
\Xi_k \ge\,& 
r_k^{n+s_1}\left[\int_{
B_{1/2}(v_k)\cap \{ x_n> \widetilde{u}_k(x')\}
}K_1\big(r_k(v_k- z)\big)\,dz
\right.\\&\qquad\qquad\left.-\int_{ B_{1/2}(v_k)\cap\{ x_n\in(\psi_k(x'), \widetilde{u}_k(x'))\}}K_1\big(r_k(v_k- z)\big)\,dz\right]-C.
\label{swfryuuytpuotuioiu76986978690}
\end{split}\end{equation}

Now, from~\eqref{sjwircbuy47uopoiuytredgjtrey} we see that~$\psi_k(x')\to\nabla\psi(x'_0)\cdot x'$
and~$u_k(x')\to\nabla u(x'_0)\cdot x'$ as~$k\to+\infty$.
Hence, if~$\vartheta=0$ it follows that~$\nabla\psi(x'_0)=\nabla u(x'_0)$.
Consequently, if~$x'\in B'_{1/2}(v'_k)$ then
\begin{equation}\label{CON2STDELTAk}
\begin{split}&\,
|\widetilde{u}_k(x')-\psi_k(x')|\\
=\,& \left|u_k(v'_k) +\nabla u_k(v_k')\cdot(x'-v_k') -
\frac{\psi(x_0'+r_k x')-\psi(x_{0}')}{r_k}
\right|\\ =\,&
\left| \frac{u(x_0'+r_k v_k')-u(x_{0}')}{r_k}+\nabla u(x_0'+r_k v_k')\cdot(x'-v_k') 
- \int_0^1 \nabla\psi(x_0'+tr_k x')\cdot x'\,dt
\right|\\=\,&
\left|\int_0^1\nabla u(x_0'+tr_k v_k')\cdot v_k'\,dt+\nabla u(x_0'+r_k v_k')\cdot(x'-v_k') 
- \int_0^1 \nabla\psi(x_0'+tr_k x')\cdot x'\,dt
\right|\\ \le\,&\left|\int_0^1\nabla u(x_0')\cdot v_k'\,dt+\nabla u(x_0')\cdot(x'-v_k') 
- \int_0^1 \nabla\psi(x_0')\cdot x'\,dt
\right|+\delta_k\\=\,&\delta_k,
\end{split}\end{equation}
for a suitable~$\delta_k$ such that~$\delta_k\to0$ as~$k\to+\infty$.

This and~\eqref{swfryuuytpuotuioiu76986978690} give that
\begin{equation}\begin{split}
\Xi_k \ge\,& 
r_k^{n+s_1}\left[\int_{
B_{1/2}(v_k)\cap \{ x_n> \widetilde{u}_k(x')\}
}K_1\big(r_k(v_k- z)\big)\,dz
\right.\\&\qquad\qquad\left.-\int_{ B_{1/2}(v_k)\cap\{ x_n\in(\widetilde{u}_k(x')-\delta_k, \widetilde{u}_k(x'))\}}K_1\big(r_k(v_k- z)\big)\,dz\right]-C.
\label{swfryuuytpuotuioiu76986978690BIS}
\end{split}\end{equation}
Now we define the map~$Y(z):=2v_k-z$ and we show that
\begin{equation}\label{flokki}
Y\Big( B_{1/2}(v_k)\cap\{ x_n\in(\widetilde{u}_k(x')-\delta_k, \widetilde{u}_k(x'))\}\Big)\subseteq
B_{1/2}(v_k)\cap\{ x_n\in(\widetilde{u}_k(x'), \widetilde{u}_k(x')+\delta_k)\}.
\end{equation}
Indeed, let~$z\in B_{1/2}(v_k)\cap\{ x_n\in(\widetilde{u}_k(x')-\delta_k, \widetilde{u}_k(x'))\}$
and call~$y:=Y(z)$. We have that~$ |y-v_k|=|v_k-z|<1/2$. Moreover,
\begin{eqnarray*}
y_n-\widetilde{u}_k(y')&=&2v_{k,n}-z_n-\widetilde{u}_k\big(2v_k'-z'\big)\\&=&
2u_k(v_k')-z_n-\widetilde{u}_k\big(2v_k'-z'\big)\\&\in &\Big(2u_k(v_k')-\widetilde{u}_k(z')
-\widetilde{u}_k\big(2v_k'-z'\big),\,
2u_k(v_k')-\widetilde{u}_k(z')-\widetilde{u}_k\big(2v_k'-z'\big)+\delta_k
\Big)\\&=&\Big(2u_k(v_k')
-2\widetilde{u}_k(v_k'),\,
2u_k(v_k')-2\widetilde{u}_k(v_k')+\delta_k
\Big)\\&=&\big(0,\,\delta_k\big)
\end{eqnarray*}
and the proof of~\eqref{flokki} is thus complete.

Using~\eqref{flokki} and changing variable~$y=Y(z)$ we see that
\begin{eqnarray*}&&
\int_{ B_{1/2}(v_k)\cap\{ x_n\in(\widetilde{u}_k(x')-\delta_k, \widetilde{u}_k(x'))\}}K_1\big(r_k(v_k- z)\big)\,dz\\
&\le& \int_{ B_{1/2}(v_k)\cap\{ x_n\in(\widetilde{u}_k(x'), \widetilde{u}_k(x')+\delta_k)\}}K_1\big(r_k(y-v_k)\big)\,dy\\&=&\int_{ B_{1/2}(v_k)\cap\{ x_n\in(\widetilde{u}_k(x'), \widetilde{u}_k(x')+\delta_k)\}}K_1\big(r_k(v_k-y)\big)\,dy.
\end{eqnarray*}
Combining this and~\eqref{swfryuuytpuotuioiu76986978690BIS}, and recalling~\eqref{stimakernel},
we arrive at
\begin{equation}\begin{split}\label{KSJLDboqkwfmCCijdf1a}
\Xi_k \ge\,& 
r_k^{n+s_1}\int_{
B_{1/2}(v_k)\cap \{ x_n> \widetilde{u}_k(x')+\delta_k\}
}K_1\big(r_k(v_k- z)\big)\,dz-C\\ \ge\,& \frac1\lambda\int_{
B_{1/2}(v_k)\cap \{ x_n> \widetilde{u}_k(x')+\delta_k\}
}\frac{dz}{|v_k- z|^{n+s_1}}\,dz-C.
\end{split}\end{equation}

Now we define
\begin{equation}\label{KSJLDboqkwfmCCijdf1b-prepa} \nu_k:=\frac{(-\nabla u_k(v'_k),1)}{\sqrt{1+|\nabla u_k(v'_k)|^2}}\qquad{\mbox{and}}\qquad
\zeta_k:=v_k+3\delta_k\nu_k\end{equation} and we
claim that, if~$k$ is sufficiently large,
\begin{equation}\begin{split}\label{KSJLDboqkwfmCCijdf1b}
B_{\delta_k}(\zeta_k)\subseteq B_{1/2}(v_k)\cap \{ x_n> \widetilde{u}_k(x')+\delta_k\}.
\end{split}\end{equation}
To check this, we observe that
\begin{equation*}
\lim_{k\to+\infty}|\nabla u_k(v_k')|=\lim_{k\to+\infty}|\nabla u(x_k')|=|\nabla u(x_0')|=|\nabla\psi(x_0')|=0
\end{equation*}
and consequently
\begin{equation}\label{KSJLDboqkwfmCCijdf1c}
\lim_{k\to+\infty}
\frac{3}{\sqrt{1+|\nabla u_k(v'_k)|^2}}-4|\nabla u_k(v_k')|-2=
1.\end{equation}
Now, pick~$w\in B_{\delta_k}(\zeta_k)$. We have that
$$ |w-v_k|\le|w-\zeta_k|+|\zeta_k-v_k|<\delta_k+3\delta_k=4\delta_k$$
and thus~$w\in B_{1/2}(v_k)$ as long as~$k$ is large enough.

Moreover,
\begin{eqnarray*}
w_n-\widetilde u_k(w')-\delta_k&\ge& (\zeta_{k,n}-\delta_k)
-u_k(v_k')-\nabla u_k(v_k')(w'-v_k')-\delta_k\\&=&
\left(v_{k,n}+\frac{3\delta_k}{\sqrt{1+|\nabla u_k(v'_k)|^2}}-\delta_k\right)
-v_{k,n}-\nabla u_k(v_k')(w'-v_k')-\delta_k\\&=&
\frac{3\delta_k}{\sqrt{1+|\nabla u_k(v'_k)|^2}}-\nabla u_k(v_k')(w'-v_k')-2\delta_k\\&\ge&
\frac{3\delta_k}{\sqrt{1+|\nabla u_k(v'_k)|^2}}-|\nabla u_k(v_k')|\,|w'-v_k'|-2\delta_k\\&\ge&
\left(\frac{3 }{\sqrt{1+|\nabla u_k(v'_k)|^2}}-4|\nabla u_k(v_k')|-2\right)\delta_k\\&>&0,
\end{eqnarray*}
thanks to~\eqref{KSJLDboqkwfmCCijdf1c}.

The proof of~\eqref{KSJLDboqkwfmCCijdf1b} is thereby complete.

Thus, exploiting~\eqref{KSJLDboqkwfmCCijdf1a} and~\eqref{KSJLDboqkwfmCCijdf1b}, we find that
\begin{equation*}\begin{split}
\Xi_k \ge \int_{B_{\delta_k}(\zeta_k)}\frac{dz}{|v_k- z|^{n+s_1}}\,dz-C.
\end{split}\end{equation*}
Notice also that if~$z\in B_{\delta_k}(\zeta_k)$ then~$|v_k-z|\le|v_k-\zeta_k|+|\zeta_k-z|\le3\delta_k+\delta_k=4\delta_k$ and accordingly
\begin{equation*}\begin{split}
\Xi_k \ge \int_{B_{\delta_k}(\zeta_k)}\frac{dz}{(4\delta_k)^{n+s_1}}\,dz-C=\frac{c}{\delta_k^{s_1}}-C,
\end{split}\end{equation*}
for some~$c>0$. This establishes the claim in~i), as desired.

The claim in~ii) can be proved similarly.

As for the claim in~iii), we suppose that~$\vartheta\in(0,\pi)$
and, for every $k\in \mathbb{N}$, we denote by~$F_k$ the set obtained
by a suitable rigid motion of the set~$E^{x_0,r_k}-v_k$ so as to have that $0\in \partial F_k$ and 
\begin{equation}\label{310BISBISBIS}
\mathbf{C}\cap F_k=\left\{x\in \mathbf{C}: x_n\leq u_k(x')\right\},
\end{equation}
for some $u_k\in C^{1,\alpha}(\mathbb{R}^{n-1})$.
Let also~$u$ be the linear function such that $u_k\rightarrow u$ in $C^{1,\alpha}(\mathbf{D})$
as~$k\to+\infty$.
We notice that, by~\eqref{ndeotuerunfsdjkd34748768947068}, up to a rigid motion,
\begin{equation}\label{convinL1bBIS}
F_k\rightarrow F:= H\cap V-v\quad \text{in $L^1_{\mathrm{loc}}(\mathbb{R}^n)$ as~$k\to+\infty$}.
\end{equation}
 
Furthermore, recalling the definition of
mean curvature in~\eqref{meancurvature} and exploiting the change of variable~$y=x_0+r_kz$,
we see that
\begin{equation}\label{meancurv2BIS}\begin{split}
\mathbf{H}^{K_1}_{\partial E} (x_k)=\,& 
\int_{\R^n} K_1(x_k-y)\bigl(\chi_{E^c}(y)-\chi_E(y)\bigr)\, dy
\\ =\,&
r_k^{-s_1}\int_{\mathbb{R}^n}r_k^{n+s_1}K_1(x_k-x_0-r_k z)\left(\chi_{(E^{x_0,r_k})^c}(z)-\chi_{E^{x_0,r_k}}(z)\right)\,dz.\end{split}
\end{equation}
We also introduce, for every $\zeta\in \mathbb{R}^n\setminus\{0\}$, the kernel
\begin{equation*}
K_{1,k}(\zeta):=r_k^{n+s_1}K_1(r_k\zeta),
\end{equation*}
and we observe that, in light of~\eqref{meancurv2BIS},
\begin{equation}\label{hdu3yv765b4385BIS}
\mathbf{H}^{K_1}_{\partial E} (x_k)=r_k^{-s_1}\mathbf{H}^{K_{1,k}}_{\partial F_k} (0).
\end{equation}
Furthermore, we recall that~$K_{1,k}\rightarrow K^\ast_1$ pointwise in $\mathbb{R}^n\setminus\{0\}$, hence one can infer from~\eqref{310BISBISBIS}, \eqref{convinL1bBIS}, \eqref{hdu3yv765b4385BIS}
and Lemma~\ref{lemmadiconv} that
\begin{equation}\label{meancurvconvaddoritrthBIS}
\lim_{k\to +\infty} r_k^{s_1}\mathbf{H}^{K_1}_{\partial E} (x_k)=\mathbf{H}^{K_1^\ast}_{\partial (H\cap V)} (v).
\end{equation}

Moreover, since~$\vartheta\in(0,\pi)$,
one can use the Lebesgue's Dominated Convergence Theorem and find that
\begin{equation*}
\begin{split}\lim_{k\to+\infty}
r_k^{s_1}\int_{\Omega^c} K_1(x_k-y)\,dy=\,&\lim_{k\to+\infty}\int_{(\Omega^{x_0,r_k})^c} r_k^{n+s_1} K_1(r_k(v_k-y))\,dy\\=\,&\int_{H^c}K_1^\ast(v-y)\,dy .
\end{split}
\end{equation*}
{F}rom this and~\eqref{meancurvconvaddoritrthBIS} we obtain the desired result in~iii).
\end{proof}

Now we showcase a refinement of Lemma~\ref{angolettilemma}
which will be needed to exclude the degenerate blow-up limits~$\vartheta\in\{0,\pi\}$
in the case~$s_1>s_2$.

\begin{lem}\label{lemmaangoletti2}
Let~$s_1>s_2$, $K_1\in\mathbf{K}^2(n,s_1,\lambda,\varrho)$ and~$K_2\in\mathbf{K}^2(n,s_2,\lambda,\varrho)$.
Let $\Omega$ be an open bounded set with $C^1$-boundary and $E$ be a volume-constrained critical set of $\mathcal{C}$. 

Let $x_0\in \mathrm{Reg}_E\cap \partial\Omega$,
$x_k\in\mathrm{Reg}_E\cap \Omega$
such that~$x_k\to x_0$ as~$k\to+\infty$
and~$r_k>0$ such that~$r_k\to0$ as~$k\to+\infty$.

Suppose that $H$ and $V$ are open
half-spaces such that
\begin{equation*}
\Omega^{x_0,r_k}\rightarrow H\qquad {\mbox{and}}\qquad
E^{x_0,r_k}\rightarrow H\cap V\quad {\mbox{in }} L^1_\mathrm{loc}(\R^n)
\;{\mbox{as }}k\to +\infty.
\end{equation*}
Let~$\vartheta\in [0,\pi]$ be the angle between
the half-spaces $H$ and $V$, that is $H\cap V=J_{0,\vartheta}$ in the notation of~\eqref{Jtheta1theta2}.

Then, 
\begin{itemize}
\item[i)] if~$\vartheta=0$ then
$$ \lim_{k\to+\infty}r_k^{s_1}
\left[\mathbf{H}^{K_1}_{\partial E} (x_k)-\int_{\Omega^c}K_1(x_k-y)\,dy\right]
+\sigma\, r_k^{s_1-s_2}
\,r_k^{s_2}\int_{\Omega^c}K_2(x_k-y)\,dy
=+\infty;
$$
\item[ii)] if~$\vartheta=\pi$ then
$$ \lim_{k\to+\infty}r_k^{s_1}
\left[\mathbf{H}^{K_1}_{\partial E} (x_k)-\int_{\Omega^c}K_1(x_k-y)\,dy\right]+\sigma\, r_k^{s_1-s_2}
\,r_k^{s_2}\int_{\Omega^c}K_2(x_k-y)\,dy=-\infty.
$$
\end{itemize}
\end{lem}

\begin{proof} We focus on the proof of~i), since the proof of~ii) is similar, up to sign changes.
To this end, we exploit
the notation introduced in Lemma~\ref{angolettilemma}, and specifically~\eqref{swfryuuytpuotuioiu76986978690}, 
and we set~$v_k:=\frac{x_k-x_0}{r_k}$,
to see that
\begin{equation}\label{okmsa0rhAyhHcoakmda}\begin{split}
\Upsilon_k:=\,&
r_k^{s_1}
\left[\mathbf{H}^{K_1}_{\partial E} (x_k)-\int_{\Omega^c}K_1(x_k-y)\,dy\right]
+\sigma\, r_k^{s_1-s_2}
\,r_k^{s_2}\int_{\Omega^c}K_2(x_k-y)\,dy\\
\ge\,&\Xi_k -|\sigma|\, r_k^{s_1-s_2}
\,r_k^{s_2}\int_{\Omega^c}K_2(x_k-y)\,dy
\\
\ge\,& 
r_k^{n+s_1}\left[\int_{
B_{1/2}(v_k)\cap \{ x_n> \widetilde{u}_k(x')\}
}K_1\big(r_k(v_k- z)\big)\,dz
\right.\\&\qquad\qquad\left.-\int_{ B_{1/2}(v_k)\cap\{ x_n\in(\psi_k(x'), \widetilde{u}_k(x'))\}}K_1\big(r_k(v_k- z)\big)\,dz\right]\\&\qquad\qquad
-|\sigma|\, r_k^{s_1-s_2}
\,r_k^{n+s_2}\int_{\R^n\setminus\Omega^{x_0,r_k}}K_2\big(r_k(v_k-z)\big)\,dz
-C
\\
\ge\,& 
r_k^{n+s_1}\left[\int_{
B_{1/2}(v_k)\cap \{ x_n> \widetilde{u}_k(x')\}
}K_1\big(r_k(v_k- z)\big)\,dz
\right.\\&\qquad\qquad\left.-\int_{ B_{1/2}(v_k)\cap\{ x_n\in(\psi_k(x'), \widetilde{u}_k(x'))\}}K_1\big(r_k(v_k- z)\big)\,dz\right]\\&\qquad\qquad
-|\sigma|\, r_k^{s_1-s_2}
\,r_k^{n+s_2}\int_{B_{1/2}(v_k)\cap\{ x_n<\psi_k(x')\}}K_2\big(r_k(v_k-z)\big)\,dz
-C,
\end{split}\end{equation}
up to changing~$C>0$ from line to line.

Also, by~\eqref{CON2STDELTAk},
\begin{eqnarray*}&& \int_{B_{1/2}(v_k)\cap\{ x_n<\psi_k(x')\}}K_2\big(r_k(v_k-z)\big)\,dz\\&=&
\int_{B_{1/2}(v_k)\cap\{ x_n\in(\widetilde{u}_k(x')-\delta_k,\psi_k(x'))\}}K_2\big(r_k(v_k-z)\big)\,dz
+\int_{B_{1/2}(v_k)\cap\{ x_n<\widetilde{u}_k(x')-\delta_k\}}K_2\big(r_k(v_k-z)\big)\,dz.
\end{eqnarray*}
Therefore, we can write~\eqref{okmsa0rhAyhHcoakmda} as
\begin{equation}\label{okmsa0rhAyhHcoakmda1}\begin{split}
\Upsilon_k\ge\,&
r_k^{n+s_1}\left[\int_{
B_{1/2}(v_k)\cap \{ x_n> \widetilde{u}_k(x')\}
}K_1\big(r_k(v_k- z)\big)\,dz
\right.\\&\qquad\qquad\left.-\int_{ B_{1/2}(v_k)\cap\{ x_n\in(\psi_k(x'), \widetilde{u}_k(x'))\}}K_1\big(r_k(v_k- z)\big)\,dz\right]\\&\qquad\qquad
-|\sigma|\, r_k^{s_1-s_2}\,r_k^{n+s_2}
\int_{B_{1/2}(v_k)\cap\{ x_n\in(\widetilde{u}_k(x')-\delta_k,\psi_k(x'))\}}K_2\big(r_k(v_k-z)\big)\,dz\\&\qquad\qquad
-|\sigma|\, r_k^{s_1-s_2}\,r_k^{n+s_2}\int_{B_{1/2}(v_k)\cap\{ x_n<\widetilde{u}_k(x')-\delta_k\}}K_2\big(r_k(v_k-z)\big)\,dz
-C.
\end{split}\end{equation}
Now we set
\begin{equation} \label{OJKHSN9q0woirfhwoiehgf49uiefgbv-2-r0ije}
{\mathcal{Z}}_k(x):=\max\Big\{
r_k^{n+s_1}\,K_1(x),\;\,
|\sigma|\, r_k^{s_1-s_2}
\,r_k^{n+s_2} \,K_2(x)
\Big\}.\end{equation}
In this way, we deduce from~\eqref{okmsa0rhAyhHcoakmda1} that
\begin{equation}\label{FRE0101}\begin{split}
\Upsilon_k\ge\,&
r_k^{n+s_1}\int_{
B_{1/2}(v_k)\cap \{ x_n> \widetilde{u}_k(x')\}
}K_1\big(r_k(v_k- z)\big)\,dz\\&\qquad\qquad
-\int_{ B_{1/2}(v_k)\cap\{ x_n\in(\widetilde{u}_k(x')-\delta_k, \widetilde{u}_k(x'))\}}{\mathcal{Z}}_k\big(r_k(v_k- z)\big)\,dz\\&\qquad\qquad
-|\sigma|\, r_k^{s_1-s_2}\,r_k^{n+s_2}\int_{B_{1/2}(v_k)\cap\{ x_n<\widetilde{u}_k(x')-\delta_k\}}K_2\big(r_k(v_k-z)\big)\,dz
-C.
\end{split}\end{equation}
Let~$Y(z):=2v_k-z$. We also use the short notation
\begin{eqnarray*}&&
{\mathcal{P}}_k:=B_{1/2}(v_k)\cap \{ x_n> \widetilde{u}_k(x')\},\\&&
{\mathcal{Q}}_k:=B_{1/2}(v_k)\cap\{ x_n\in(\widetilde{u}_k(x')-\delta_k, \widetilde{u}_k(x'))\}\\
{\mbox{and }}&&{\mathcal{R}}_k:=B_{1/2}(v_k)\cap\{ x_n<\widetilde{u}_k(x')-\delta_k\}
.\end{eqnarray*}
We know from~\eqref{flokki} that
\begin{equation}\label{FRE0102}
Y({\mathcal{Q}}_k)\subseteq
B_{1/2}(v_k)\cap\{ x_n\in(\widetilde{u}_k(x'), \widetilde{u}_k(x')+\delta_k)\}
\subseteq{\mathcal{P}}_k.
\end{equation}
We also claim that
\begin{equation}\label{FRE0103}
Y({\mathcal{R}}_k)\subseteq
{\mathcal{P}}_k\setminus Y({\mathcal{Q}}_k).
\end{equation}
Indeed, if there were a point~$y\in Y({\mathcal{Q}}_k)\cap Y({\mathcal{R}}_k)$
we would have that~$y=2v_k-Q=2v_k-R$ for some~$Q\in{\mathcal{Q}}_k$ and~$R\in{\mathcal{R}}_k$, but this would entail that~$Q=R\in{\mathcal{Q}}_k\cap{\mathcal{R}}_k=\varnothing$, which is a contradiction. This shows that~$Y({\mathcal{R}}_k)$ lies in the complement of~$Y({\mathcal{Q}}_k)$, thus, to complete the proof of~\eqref{FRE0103}, it only remains to show that~$Y({\mathcal{R}}_k)\subseteq
{\mathcal{P}}_k$. To this end, we observe that if~$ z_n<\widetilde{u}_k(z')-\delta_k$
and~$y=Y(z)$, then
\begin{eqnarray*}&&
y_n-\widetilde{u}_k(y')=2v_{k,n}-z_n-\widetilde{u}_k(y')=2\widetilde{u}_k(v'_k)-z_n-\widetilde{u}_k(2v'_k-z')\\&&\qquad>
2\widetilde{u}_k(v'_k)-\widetilde{u}_k(z')+\delta_k-\widetilde{u}_k(2v'_k-z')=\delta_k>0.
\end{eqnarray*}
This completes the proof of~\eqref{FRE0103}.

Hence, by~\eqref{FRE0101}, \eqref{FRE0102} and~\eqref{FRE0103},
\begin{equation}\label{FRE01044}\begin{split}
\Upsilon_k\ge\,&
r_k^{n+s_1}\int_{
{\mathcal{P}}_k
}K_1\big(r_k(v_k- z)\big)\,dz
-\int_{ {\mathcal{Q}}_k}{\mathcal{Z}}_k\big(r_k(v_k- z)\big)\,dz\\&\qquad\qquad
-|\sigma|\, r_k^{s_1-s_2}\,r_k^{n+s_2}\int_{{\mathcal{R}}_k}K_2\big(r_k(v_k-z)\big)\,dz
-C\\
=\,&
r_k^{n+s_1}\int_{
{\mathcal{P}}_k
}K_1\big(r_k(v_k- z)\big)\,dz
-\int_{ Y({\mathcal{Q}}_k)}{\mathcal{Z}}_k\big(r_k(v_k- y)\big)\,dy\\&\qquad\qquad
-|\sigma|\, r_k^{s_1-s_2}\,r_k^{n+s_2}\int_{Y({\mathcal{R}}_k)}K_2\big(r_k(v_k-y)\big)\,dy
-C\\
=\,&r_k^{n+s_1}
\int_{
{\mathcal{P}}_k\setminus( Y({\mathcal{Q}}_k)\cup Y({\mathcal{R}}_k))
}K_1\big(r_k(v_k- z)\big)\,dz
+
\int_{
Y({\mathcal{Q}}_k)
}\alpha_k(z)\,dz
+
\int_{Y({\mathcal{R}}_k)
}\beta_k(z)\,dz-C,
\end{split}\end{equation}
where
\begin{eqnarray*}
&&\alpha_k(z):=r_k^{n+s_1}\,K_1\big(r_k(v_k- z)\big)-{\mathcal{Z}}_k\big(r_k(v_k- z)\big)\\
{\mbox{and }}&&\beta_k(z):=r_k^{n+s_1}\,K_1\big(r_k(v_k- z)\big)-
|\sigma|\, r_k^{s_1-s_2}\,r_k^{n+s_2}\,K_2\big(r_k(v_k-z)\big).
\end{eqnarray*}
We stress that up to now the condition~$s_1>s_2$ has not been used.
We are going to exploit it now to bound~$\alpha_k$ and~$\beta_k$.
For this, we note that, if~$z\in B_{1/2}(v_k)$ and~$k$ is large enough, then
\begin{eqnarray*}&&
|\sigma|\, r_k^{s_1-s_2}\,r_k^{n+s_2} \,K_2\big(r_k(v_k-z)\big)\le
\frac{\lambda\,|\sigma|\, r_k^{s_1-s_2} }{|v_k-z|^{n+s_2}}\le
\frac{\lambda\,|\sigma|\, r_k^{s_1-s_2} }{|v_k-z|^{n+s_1}}=
\frac{\lambda\,|\sigma|\, r_k^{s_1-s_2}\,r_k^{n+s_1} }{\big|r_k(v_k-z)\big|^{n+s_1}}\\&&\qquad\qquad\le
\lambda^2\,|\sigma|\, r_k^{s_1-s_2}\,r_k^{n+s_1} \,K_1\big(r_k(v_k-z)\big)
\le \frac12\,r_k^{n+s_1} \,K_1\big(r_k(v_k-z)\big).
\end{eqnarray*}
This and~\eqref{OJKHSN9q0woirfhwoiehgf49uiefgbv-2-r0ije}
entail that if~$z\in B_{1/2}(v_k)$ and~$k$ is large enough, then~${\mathcal{Z}}_k\big(r_k(v_k-z)\big)=r_k^{n+s_1}\,K_1\big(r_k(v_k-z)\big)$, and therefore~$\alpha_k(z)=0$.
In addition,
$$ \beta_k(z)\ge\frac12\,r_k^{n+s_1} \,K_1\big(r_k(v_k-z)\big) .$$
{F}rom these observations and~\eqref{FRE01044} we arrive at
\begin{equation}\label{24hmYo2q453ty435wrtfg2qwqadf43eg33unglaw}
\begin{split}
\Upsilon_k&\,\ge\,r_k^{n+s_1}
\int_{
{\mathcal{P}}_k\setminus( Y({\mathcal{Q}}_k)\cup Y({\mathcal{R}}_k))
}K_1\big(r_k(v_k- z)\big)\,dz+\frac12\,r_k^{n+s_1}\,\int_{Y({\mathcal{R}}_k)
}K_1\big(r_k(v_k- z)\big)\,dz-C\\
&\ge\,\frac12\,r_k^{n+s_1}
\int_{
{\mathcal{P}}_k\setminus Y({\mathcal{Q}}_k)
}K_1\big(r_k(v_k- z)\big)\,dz-C.
\end{split}
\end{equation}

Now we utilize the notation in~\eqref{KSJLDboqkwfmCCijdf1b-prepa},
the inclusion in~\eqref{KSJLDboqkwfmCCijdf1b} and the first inclusion in~\eqref{FRE0102} to see that
\begin{equation}\label{24hmYo2q453ty435wrtfg2qwqadf43eg33unglawBIS}\begin{split}
{\mathcal{P}}_k\setminus Y({\mathcal{Q}}_k) \supseteq\,&
{\mathcal{P}}_k\setminus\Big( B_{1/2}(v_k)\cap\{ x_n\in(\widetilde{u}_k(x'), \widetilde{u}_k(x')+\delta_k)\}\Big)\\=\,&
B_{1/2}(v_k)\cap \{ x_n\ge \widetilde{u}_k(x')+\delta_k\}\\
\supseteq\,&B_{\delta_k}(\zeta_k).
\end{split}\end{equation}
By plugging this information into~\eqref{24hmYo2q453ty435wrtfg2qwqadf43eg33unglaw}, we thereby conclude that
\begin{equation}\label{24hmYo2q453ty435wrtfg2qwqadf43eg33unglawTRIS}\begin{split}
\Upsilon_k\ge\,&\frac12\,r_k^{n+s_1}
\int_{B_{\delta_k}(\zeta_k)
}K_1\big(r_k(v_k- z)\big)\,dz-C
\\ \ge\,&
\frac12\,
\int_{B_{\delta_k}(\zeta_k)
}\frac{dz}{|v_k- z|^{n+s_1}}-C\\=\,&
\frac{c}{\delta_k^{s_1}}-C,\end{split}\end{equation}
for some~$c>0$.
{F}rom this, the desired result in~i) plainly follows.
\end{proof}

With this, we are in the position of providing the proof of Theorem~\ref{thmYounglaw},
where we suppose that~$a_1$ and~$a_2$ are anisotropic functions and then, as a special case, we exhibit the proof of~Corollary \ref{corollario1} where we take~$a_1\equiv \mathrm{const}$.

\begin{proof}[Proof of Theorem \ref{thmYounglaw}]
We fix a point~$x_0\in \partial\Omega\cap \mathrm{Reg}_E$ and a sequence
of points~$x_k\in \Omega\cap \mathrm{Reg}_E$ such that $x_k\rightarrow x_0$ as~$k\to+\infty$. We also set~$r_k:=|x_k-x_0|$
and we observe that~$r_k\to0$ as~$k\to+\infty$.

{F}rom~\eqref{E-Ls} evaluated at~$x_k$, we get 
\begin{equation*}
\mathbf{H}^{K_1}_{\partial E} (x_k) -\int_{\Omega^c}K_1(x_k-y)\,dy+\sigma\int_{\Omega^c}K_2(x_k-y)\,dy+g(x_k)=c,
\end{equation*}
where $c$ does not depend on $k$. Multiplying both sides by $r_k^{s_1}$, we thereby obtain that
\begin{equation*}
r_k^{s_1}\,\mathbf{H}^{K_1}_{\partial E} (x_k)-r_k^{s_1}\int_{\Omega^c}K_1(x_k-y)\,dy+\sigma\, r_k^{s_1-s_2}
\,r_k^{s_2}\int_{\Omega^c}K_2(x_k-y)\,dy+r_k^{s_1}\,g(x_k)
=c\,r_k^{s_1}.
\end{equation*}
Notice that, since $g$ is locally bounded, we have that~$r_k^{s_1}g(x_k)\to 0$ as $k\to+\infty$. 
As a consequence,
\begin{equation}\label{importnbtf0000}
\lim_{k\to+\infty} r_k^{s_1}
\left[\mathbf{H}^{K_1}_{\partial E} (x_k)-\int_{\Omega^c}K_1(x_k-y)\,dy\right]+\sigma
r_k^{s_1-s_2}\,r_k^{s_2}\int_{\Omega^c}K_2(x_k-y)\,dy=0.
\end{equation}

Now, we prove the statement in 1) of Theorem~\ref{thmYounglaw}. For this,
we suppose that~$s_1<s_2$ and~$\sigma<0$. In this case,
$$ \sigma \,r_k^{s_1-s_2}\,r_k^{s_2}\int_{\Omega^c}K_2(x_k-y)\,dy\le 0,$$
and therefore by ii) in Lemma~\ref{angolettilemma} and~\eqref{importnbtf0000}
we deduce that~$\vartheta\neq\pi$. Hence, to prove 1)
it remains to check that
\begin{equation}\label{djiwehtuithbvsdjkghierout85yu5897568}
\vartheta\not\in(0,\pi).\end{equation}
To this end, we suppose by contradiction that~$\vartheta\in(0,\pi)$.
Then, by the Lebesgue's Dominated Convergence Theorem,
\begin{equation}\label{jiwhrwutui4ty34ui}
\begin{split}\lim_{k\to+\infty}
r_k^{s_2}\int_{\Omega^c} K_2(x_k-y)\,dy=\,&\lim_{k\to+\infty}\int_{(\Omega^{x_0,r_k})^c} r_k^{n+s_2} K_2(r_k(v_k-y))\,dy\\=\,&\int_{H^c}K_2^\ast(v-y)\,dy 
\end{split}
\end{equation}
and this limit is finite.
Consequently,
$$ \lim_{k\to+\infty} \sigma \,r_k^{s_1-s_2}\,r_k^{s_2}\int_{\Omega^c}K_2(x_k-y)\,dy=-\infty.$$
This and~iii) in Lemma~\ref{angolettilemma} contradict~\eqref{importnbtf0000},
and thus~\eqref{djiwehtuithbvsdjkghierout85yu5897568} is proved. 

Accordingly, if~$s_1<s_2$ and~$\sigma<0$, then necessarily~$\vartheta=0$, which establishes~1).

We now prove the statement in 2). Namely we consider the case
in which~$s_1<s_2$ and~$\sigma>0$, and thus
$$ \sigma \,r_k^{s_1-s_2}\,r_k^{s_2}\int_{\Omega^c}K_2(x_k-y)\,dy\ge 0.$$
{F}rom this, i) in Lemma~\ref{angolettilemma} and~\eqref{importnbtf0000}
we infer that~$\vartheta\neq0$. Hence, to establish 2)
we show that
\begin{equation}\label{djiwehtuithbvsdjkghierout85yu5897568BIS}
\vartheta\not\in(0,\pi).\end{equation}
We argue as before and we suppose by contradiction that~$\vartheta\in(0,\pi)$.
Then, exploiting~\eqref{jiwhrwutui4ty34ui} we see that
$$ \lim_{k\to+\infty} \sigma \,r_k^{s_1-s_2}\,r_k^{s_2}\int_{\Omega^c}K_2(x_k-y)\,dy=+\infty.$$
This and~iii) in Lemma~\ref{angolettilemma} contradict~\eqref{importnbtf0000},
and thus~\eqref{djiwehtuithbvsdjkghierout85yu5897568BIS} is proved. 

As a consequence, if~$s_1<s_2$ and~$\sigma>0$, then~$\vartheta=\pi$, hence we have established~2)
as well. Hence, we now focus on the statement in~3).

For this, we first suppose that~$s_1<s_2$ and~$\sigma=0$. Then, \eqref{importnbtf0000} becomes
\begin{equation}\label{djweiorucbiotuebtovb5u864y58}
\lim_{k\to+\infty} r_k^{s_1}
\left[\mathbf{H}^{K_1}_{\partial E} (x_k)-\int_{\Omega^c}K_1(x_k-y)\,dy\right]=0.
\end{equation}
This and Lemma~\ref{angolettilemma} give that~$\vartheta\in(0,\pi)$ in this case.

In the case in which~$s_1>s_2$, if~$\vartheta\in\{0,\pi\}$ then we would use
Lemma~\ref{lemmaangoletti2} to find a contradiction with~\eqref{importnbtf0000}, hence
we conclude that necessarily~$\vartheta\in(0,\pi)$ in this case as well.

Now, in order to prove~\eqref{Y:LS}, we take~$v\in H\cap \partial V$, then by~\eqref{NOHA}
we have that, for every~$k$, there exists~$v_k\in \Omega^{x_0,r_k}\cap\partial E^{x_0,r_k}$
such that~$v_k\to v$ as~$k\to+\infty$, where~$r_k$ is an infinitesimal sequence as~$k\to+\infty$. As a consequence, for every~$k$,
there exists~$x_k\in \mathrm{Reg}_E\cap\Omega$ 
such that~$v_k=\frac{x_k-x_0}{r_k}$ and~$x_k\to x_0$ as~$k\to+\infty$.
Then, we are in the position to apply iii) in Lemma~\ref{angolettilemma} and conclude that
\begin{equation}\label{swiru4t94kijhygfmjnhbgvf89765trewt54e5q4gy5hyuj}
\lim_{k\to+\infty}r_k^{s_1}
\left[\mathbf{H}^{K_1}_{\partial E} (x_k)- \int_{\Omega^c}K_1(x_k-y)\,dy\right]=
\mathbf{H}^{K_1^\ast}_{\partial (H\cap V)} (v)-\int_{H^c}K_1^\ast (v-y)\,dy .\end{equation}

Also, if~$s_1>s_2$,
we recall that the limit in~\eqref{jiwhrwutui4ty34ui} is finite (since~$\vartheta\in(0,\pi)$)
and that~$r_k$ is infinitesimal to infer that
\begin{equation*}
\lim_{k\to+\infty}
r_k^{s_1-s_2}\,r_k^{s_2}\int_{\Omega^c}K_2(x_k-y)\,dy=0.
\end{equation*}
This, together with~\eqref{importnbtf0000}, gives that~\eqref{djweiorucbiotuebtovb5u864y58}
holds true in this case as well.

Accordingly, from~\eqref{djweiorucbiotuebtovb5u864y58} and~\eqref{swiru4t94kijhygfmjnhbgvf89765trewt54e5q4gy5hyuj} we deduce that
$$\mathbf{H}^{K_1^\ast}_{\partial (H\cap V)} (v)-\int_{H^c}K_1^\ast (v-y)\,dy=0,$$
which establishes~\eqref{Y:LS}.

Hence, to complete the proof of the statement in~3), it remains
to check that~$\widehat{\vartheta}=\pi-\vartheta$, being~$\widehat{\vartheta}\in (0,2\pi)$ the angle given in~\eqref{cancellation} with~$c=0$.

For this, we exploit the notation in~\eqref{etheta}, the assumption in~\eqref{adsyrytrjgtjhg697867}
and the change of variable~$z=y/|v|$, to see that,
for all~$v\in  H\cap  \partial V$,
the left hand side of~\eqref{Y:LS} can be written as
\begin{equation*}
\begin{split}&
\mathbf{H}^{K_1^\ast}_{\partial (H\cap V)} (v)-\int_{H^c}K_1^\ast (v-y)\,dy
=\int_{\R^n}K_1^\ast(v-y)\big(\chi_{(H\cap V)^c\cap H}(y)
-\chi_{H\cap V}(y)
\big)\,dy
\\&\qquad\qquad=\int_{\R^n}\frac{a_1(\overrightarrow{v-y})}{|v-y|^{n+s_1}}\big(\chi_{(H\cap V)^c\cap H}(y)
-\chi_{H\cap V}(y)
\big)\,dy
\\
&\qquad\qquad =|v|^{-s_1}
\int_{\mathbb{R}^n}\frac{a_1(\overrightarrow{e(\vartheta)-z})\big(\chi_{J_{0,\vartheta}^c\cap H}(z)
-\chi_{J_{0,\vartheta}}(z)\big)}{|e(\vartheta)-z|^{n+s_1}}\,dz\\
&\qquad
\qquad=|v|^{-s_1}
\int_{J_{\vartheta,\pi}}\frac{a_1(\overrightarrow{e(\vartheta)-z})}{|e(\vartheta)-z|^{n+s_1}}\, dz
-|v|^{-s_1}\int_{J_{0,\vartheta}}\frac{a_1(\overrightarrow{e(\vartheta)-z})}{|e(\vartheta)-z|^{n+s_1}}\, dz.
\end{split}
\end{equation*}
Therefore, by~\eqref{Y:LS},
\begin{equation}\label{equivbof54yib59vyn}
\int_{J_{\vartheta,\pi}}\frac{a_1(\overrightarrow{e(\vartheta)-z})}{|e(\vartheta)-z|^{n+s_1}}\, dz-\int_{J_{0,\vartheta}}\frac{a_1(\overrightarrow{e(\vartheta)-z})}{|e(\vartheta)-z|^{n+s_1}}\, dz
=0.
\end{equation}
Consequently, recalling the notation in~\eqref{D} and exploting~\eqref{cancellation} with~$c=0$,
we have that
$$
 \mathcal{D}_\vartheta(\pi-{\vartheta})=
\int_{J_{\vartheta,\pi}}\frac{a_1(\overrightarrow{e(\vartheta)-z})}{|e(\vartheta)-z|^{n+s_1}}\, dz-\int_{J_{0,\vartheta}}\frac{a_1(\overrightarrow{e(\vartheta)-z})}{|e(\vartheta)-z|^{n+s_1}}\, dz
=0=\mathcal{D}_\vartheta(\widehat{\vartheta}).$$
By the uniqueness claim in Proposition~\ref{cancellazione}, we conclude that~$\pi-{\vartheta}=\widehat\vartheta$, as desired.

This completes the proof of~3), and in turn of Theorem~\ref{thmYounglaw}.
\end{proof}

As a consequence of Theorem~\ref{thmYounglaw} we now obtain the particular case
in which~$a_1\equiv\mathrm{const}$ dealt with in
Corollary~\ref{corollario1}.

\begin{proof}[Proof of Corollary \ref{corollario1}] 
We point out that 1) and~2) in Corollary \ref{corollario1} follow from~1) and~2)
in Theorem~\ref{thmYounglaw}, respectively.

To prove 3) of Corollary \ref{corollario1}, we first notice that~$\vartheta\in(0,\pi)$ in these cases.
Also, if~$a_1\equiv\mathrm{const}$, then the cancellation property
in~\eqref{cancellation} boils down to~$\mathcal{D}_\vartheta(\vartheta)=0$,
and therefore, by the uniqueness claim in Proposition~\ref{cancellazione}
we obtain that~$\widehat{\vartheta}=\vartheta$.

Furthermore, we recall that~\eqref{Y:LS} holds true in this case, thanks to~3) of
Theorem~\ref{thmYounglaw}, and therefore, using the equivalent formulation of~\eqref{Y:LS}
given in~\eqref{equivbof54yib59vyn} (with~$a_1\equiv\mathrm{const}$ in this case), we find that
\begin{equation*}
 \mathcal{D}_\vartheta(\pi-{\vartheta})=
\int_{J_{\vartheta,\pi}}\frac{a_1}{|e(\vartheta)-z|^{n+s_1}}\, dz
-\int_{J_{0,\vartheta}}\frac{a_1}{|e(\vartheta)-z|^{n+s_1}}\, dz
=0=\mathcal{D}_\vartheta({\vartheta}).\end{equation*}
Hence, using again the uniqueness claim in Proposition~\ref{cancellazione}
we conclude that~$\pi-\vartheta=\vartheta$, which gives that~$\vartheta=\frac{\pi}2$, as desired.
\end{proof}  

We now deal with the case~$s_1=s_2$, as given by Theorem~\ref{thmYounglawa1a2nonconst}. For this,
we need a variation of Lemma~\ref{lemmaangoletti2} that takes into account the situation
in which~$s_1=s_2$.

\begin{lem}\label{lemmaangoletti3}
Let~$s\in(0,1)$ and $K_1$, $K_2\in\mathbf{K}^2(n,s,\lambda,\varrho)$.
Assume that there exists~$\varepsilon_0\in(0,1)$ such that
\begin{equation}\label{condimprotante}
|\sigma|\, K_2(\zeta) \le (1-\varepsilon_0) \,K_1(\zeta) \qquad{\mbox{for all }}\zeta\in B_{\varepsilon_0}\setminus\{0\}.
\end{equation}

Let $\Omega$ be an open bounded set with $C^1$-boundary and $E$ be a volume-constrained critical set of $\mathcal{C}$. 

Let $x_0\in \mathrm{Reg}_E\cap \partial\Omega$,
$x_k\in\mathrm{Reg}_E\cap \Omega$
such that~$x_k\to x_0$ as~$k\to+\infty$ and~$r_k>0$ such that~$r_k\to0$ as~$k\to+\infty$.

Suppose that $H$ and $V$ are open
half-spaces such that
\begin{equation*}
\Omega^{x_0,r_k}\rightarrow H\qquad {\mbox{and}}\qquad
E^{x_0,r_k}\rightarrow H\cap V\quad {\mbox{in }} L^1_\mathrm{loc}(\R^n)
\;{\mbox{as }}k\to +\infty.
\end{equation*}
Let~$\vartheta\in [0,\pi]$ be the angle between
the half-spaces $H$ and $V$, that is $H\cap V=J_{0,\vartheta}$ in the notation of~\eqref{Jtheta1theta2}.

Then,  
\begin{itemize}
\item[i)] if~$\vartheta=0$ then
$$ \lim_{k\to+\infty}r_k^{s}
\left[\mathbf{H}^{K_1}_{\partial E} (x_k)-\int_{\Omega^c}K_1(x_k-y)\,dy
+\sigma\int_{\Omega^c}K_2(x_k-y)\,dy\right]
=+\infty;
$$
\item[ii)] if~$\vartheta=\pi$ then
$$ \lim_{k\to+\infty}r_k^{s}
\left[\mathbf{H}^{K_1}_{\partial E} (x_k)-\int_{\Omega^c}K_1(x_k-y)\,dy+\sigma\int_{\Omega^c}K_2(x_k-y)\,dy\right]=-\infty.
$$
\end{itemize}
\end{lem}

\begin{proof}
We establish~i), being the proof of~ii) analogous. 
For this, we use the notation introduced in the proof of Lemma~\ref{lemmaangoletti2}, 
and specifically we recall formula~\eqref{FRE01044}, to be used here with~$s_1=s_2=s$.
In this case, we use~\eqref{condimprotante} to see that, if~$k$ is large enough, for all~$z\in B_{1/2}(v_k)$
we have that
\begin{equation}\label{sgorbrador22}
|\sigma| \,K_2\big(r_k(v_k-z)\big)\le (1-\varepsilon_0)\,K_1\big(r_k(v_k-z)\big).\end{equation}
This and~\eqref{OJKHSN9q0woirfhwoiehgf49uiefgbv-2-r0ije} give that
\begin{eqnarray*}
{\mathcal{Z}}_k\big(r_k(v_k-z)\big)&=&r_k^{n+s}\max\Big\{
K_1\big(r_k(v_k-z)\big),\;\,
|\sigma| \,K_2\big(r_k(v_k-z)\big)
\Big\}\\&=&
r_k^{n+s}K_1\big(r_k(v_k-z)\big),
\end{eqnarray*}
which entails that~$\alpha_k(z)=0$.

Also, using again~\eqref{sgorbrador22}, it follows that
\begin{eqnarray*}
\beta_k(z)=r_k^{n+s}\Big(K_1\big(r_k(v_k- z)\big)-
|\sigma|\,K_2\big(r_k(v_k-z)\big)\Big)\ge \varepsilon_0\, r_k^{n+s}\,K_1\big(r_k(v_k- z)\big).
\end{eqnarray*}
In light of these observations, \eqref{FRE01044} in this framework reduces to
$$ \Upsilon_k\ge \varepsilon_0
\, r_k^{n+s}\int_{{\mathcal{P}}_k\setminus Y({\mathcal{Q}}_k)}
K_1\big(r_k(v_k- z)\big)\,dz-C.
$$
We have thus recovered the last inequality in~\eqref{24hmYo2q453ty435wrtfg2qwqadf43eg33unglaw},
with~$1/2$ replaced by the constant~$\varepsilon_0$. Then it suffices to proceed as in~\eqref{24hmYo2q453ty435wrtfg2qwqadf43eg33unglawBIS}
and~\eqref{24hmYo2q453ty435wrtfg2qwqadf43eg33unglawTRIS} to complete the proof.
\end{proof}

With this additional result, we are now in the position of giving the proof 
of Theorem~\ref{thmYounglawa1a2nonconst}.

\begin{proof}[Proof of Theorem \ref{thmYounglawa1a2nonconst}] 
We fix a point~$x_0\in \partial\Omega\cap \mathrm{Reg}_E$ and a sequence
of points~$x_k\in \Omega\cap \mathrm{Reg}_E$ such that $x_k\rightarrow x_0$ as~$k\to+\infty$. We also set~$r_k:=|x_k-x_0|$
and we observe that~$r_k\to0$ as~$k\to+\infty$.

{F}rom~\eqref{E-Ls} evaluated at~$x_k$, we get 
\begin{equation*}
\mathbf{H}^{K_1}_{\partial E} (x_k) -\int_{\Omega^c}K_1(x_k-y)\,dy+\sigma\int_{\Omega^c}K_2(x_k-y)\,dy+g(x_k)=c,
\end{equation*}
where $c$ does not depend on $k$. Thus, multiplying both sides by $r_k^{s}$, we find that
\begin{equation*}
r_k^{s}\,\mathbf{H}^{K_1}_{\partial E} (x_k)-r_k^{s}\int_{\Omega^c}K_1(x_k-y)\,dy+\sigma\, r_k^{s_1}
\int_{\Omega^c}K_2(x_k-y)\,dy+r_k^{s}\,g(x_k)
=c\,r_k^{s}.
\end{equation*}
Since $g$ is locally bounded, we have that~$r_k^{s}g(x_k)\to 0$ as $k\to+\infty$, and therefore
\begin{equation}\label{spwt5u0789786543we676iyut}
\lim_{k\to+\infty} r_k^{s}
\left[\mathbf{H}^{K_1}_{\partial E} (x_k)-\int_{\Omega^c}K_1(x_k-y)\,dy+\sigma
\int_{\Omega^c}K_2(x_k-y)\,dy\right]=0.
\end{equation}
In light of Lemma~\ref{lemmaangoletti3} (which can be exploited here thanks
to assumption~\eqref{magvalimpo}), this gives that the angle~$\vartheta$ between~$H$ and~$V$
lies in~$(0,\pi)$.

Thus, in order to prove~\eqref{Ylaw}, we can take~$v\in H\cap \partial V$ and
we see that, for every~$k$, there exists~$v_k\in \Omega^{x_0,r_k}\cap\partial E^{x_0,r_k}$
such that~$v_k\to v$ as~$k\to+\infty$, where~$r_k$ is an infinitesimal sequence as~$k\to+\infty$. As a consequence, for every~$k$,
there exists~$x_k\in \mathrm{Reg}_E\cap\Omega$ 
such that~$v_k=\frac{x_k-x_0}{r_k}$ and~$x_k\to x_0$ as~$k\to+\infty$.
Then, we are in the position to apply iii) in Lemma~\ref{angolettilemma} and conclude that
\begin{equation*}
\lim_{k\to+\infty}r_k^{s_1}
\left[\mathbf{H}^{K_1}_{\partial E} (x_k)- \int_{\Omega^c}K_1(x_k-y)\,dy\right]=
\mathbf{H}^{K_1^\ast}_{\partial (H\cap V)} (v)-\int_{H^c}K_1^\ast (v-y)\,dy .\end{equation*}

Also, by Lebesgue's Dominated Convergence Theorem,
\begin{equation*}
\begin{split}\lim_{k\to+\infty}
r_k^{s}\int_{\Omega^c} K_2(x_k-y)\,dy=\,&\lim_{k\to+\infty}\int_{(\Omega^{x_0,r_k})^c} r_k^{n+s}
K_2(r_k(v_k-y))\,dy\\=\,&\int_{H^c}K_2^\ast(v-y)\,dy 
\end{split}
\end{equation*}
and this limit is finite.

These considerations and~\eqref{spwt5u0789786543we676iyut} give the
desired result in~\eqref{Ylaw}.
\end{proof}

We are now in the position of establishing
Proposition~\ref{chilosa22}.

\begin{proof}[Proof of Proposition~\ref{chilosa22}]
We exploit the notation in~\eqref{etheta}, the assumption in~\eqref{adsyrytrjgtjhg697867}
and the change of variable~$z=y/|v|$, to see that~\eqref{vfjoovvvv9990000043554647}
can be written as
\begin{equation*}
\begin{split}0=\,&
\mathbf{H}^{K_1^\ast}_{\partial (H\cap V)} (v)-\int_{H^c}K_1^\ast (v-y)\,dy+\sigma
\int_{H^c}K_2^\ast (v-y)\,dy 
\\
=\,&\int_{\R^n}K_1^\ast(v-y)\big(\chi_{(H\cap V)^c\cap H}(y)
-\chi_{H\cap V}(y)
\big)\,dy+\sigma
\int_{H^c}K_2^\ast (v-y)\,dy 
\\=\,&\int_{\R^n}\frac{a_1(\overrightarrow{v-y})}{|v-y|^{n+s_1}}\big(\chi_{(H\cap V)^c\cap H}(y)
-\chi_{H\cap V}(y)
\big)\,dy+\sigma
\int_{H^c}\frac{a_2(\overrightarrow{v-y})}{|v-y|^{n+s_2}}\,dy \\
=\,&|v|^{-s_1}
\int_{\mathbb{R}^n}\frac{a_1(\overrightarrow{e(\vartheta)-z})\big(\chi_{J_{0,\vartheta}^c\cap H}(z)
-\chi_{J_{0,\vartheta}}(z)\big)}{|e(\vartheta)-z|^{n+s_1}}\,dz
+
\sigma |v|^{-s_2}\int_{H^c}\frac{a_2(\overrightarrow{e(\vartheta)-z})}{|e(\vartheta)-z|^{n+s_2}}\,dz
\\
=\,&|v|^{-s_1}
\int_{J_{\vartheta,\pi}}\frac{a_1(\overrightarrow{e(\vartheta)-z})}{|e(\vartheta)-z|^{n+s_1}}\, dz
-|v|^{-s_1}\int_{J_{0,\vartheta}}\frac{a_1(\overrightarrow{e(\vartheta)-z})}{|e(\vartheta)-z|^{n+s_1}}\, dz
\\&\qquad
+\sigma |v|^{-s_2}\int_{H^c}\frac{a_2(\overrightarrow{e(\vartheta)-z})}{|e(\vartheta)-z|^{n+s_2}}\,dz.
\end{split}
\end{equation*}
Hence, recalling the assumption in~\eqref{sjwicru348bt4v8b4576848poiuytrewqasdfghjkmnbvs},
this gives the desired result in~\eqref{sigmagenerale-0}.\end{proof}

\section{Proofs of Theorems~\ref{ESEMPIOPTO} and~\ref{ESEMPIOPTO:CO}}\label{ELLANUEOD32}

We now deal with the possibly degenerate cases in which the nonlocal droplets
either detach from the container or adhere completely to its surfaces. These cases depend on the strong
attraction or repulsion of the second kernel and are described in the examples provided
in Theorems~\ref{ESEMPIOPTO} and~\ref{ESEMPIOPTO:CO}, which we are now going to prove.
For this, we need some auxiliary integral estimates to detect the interaction between ``thin sets''.
This is formalized in Lemmata~\ref{LE:AUSJM-1aux1} and~\ref{LE:AUSJM-1aux2} here below:

\begin{lem}\label{LE:AUSJM-1aux1}Let~$r$, $t>0$, $s\in(0,1)$ and
$$ D:=\big\{x=(x',x_n)\in\R^n\,:\,|x'|<r {\mbox{ and }}x_n\in(0,t)\big\}.$$
Then,
$$ \iint_{D\times\{y_n<0\}}\frac{dx\,dy}{|x-y|^{n+s}}=c_\star\, r^{n-1}\,t^{1-s},$$
for a suitable~$c_\star>0$, depending only on~$n$ and~$s$.
\end{lem}

\begin{proof} We recall that the surface area of the $(n-1)$-dimensional unit sphere is equal to~${\frac {2\pi ^{\frac {n}{2}}}{\Gamma \left({\frac {n}{2}}\right)}}$,
where~$\Gamma$ is the Gamma Function. Furthermore,
$$ \int_0^{+\infty}\frac{\ell^{n-2}\, d\ell}{\big( \ell^2+1\big)^{\frac{n+s}2}}=
\frac{ \Gamma\left(\frac{n-1}2\right) \Gamma\left(\frac{1+s}2\right)}{2 \Gamma\left(\frac{n+s}2\right)} .$$
Hence,
we use the substitution~$\xi:=\frac{y'-x'}{x_n-y_n}$ to see that
\begin{eqnarray*}
&&\iint_{D\times\{y_n<0\}}\frac{dx\,dy}{|x-y|^{n+s}}\\&&\qquad=
\int_0^t\left[\int_{\{|x'|<r\}} \left[\int_{-\infty}^0\left[\int_{\R^{n-1}}
\frac{d\xi}{ (x_n-y_n)^{1+s}\big( |\xi|^2+1\big)^{\frac{n+s}2}}
\right]\,dy_n
\right]
\,dx'\right]\,dx_n\\&&\qquad={\frac{2\pi ^{\frac {n-1}{2}}}{\Gamma \left({\frac {n-1}{2}}\right)}}
\int_0^t\left[\int_{\{|x'|<r\}} \left[\int_{-\infty}^0\left[\int_0^{+\infty}
\frac{\ell^{n-2}\, d\ell}{ (x_n-y_n)^{1+s}\big( \ell^2+1\big)^{\frac{n+s}2}}
\right]\,dy_n
\right]
\,dx'\right]\,dx_n\\
\\&&\qquad={\frac{\pi ^{\frac {n-1}{2}} \Gamma\left(\frac{1+s}2\right)}{\Gamma\left(\frac{n+s}2\right)}}
\int_0^t\left[\int_{\{|x'|<r\}} \left[\int_{-\infty}^0\frac{dy_n}{ (x_n-y_n)^{1+s}}
\right]
\,dx'\right]\,dx_n
\\&&\qquad=
{\frac{2 \pi ^{\frac {2n-1}{2}} \Gamma\left(\frac{1+s}2\right)}{\Gamma \left({\frac {n}{2}}\right)\Gamma\left(\frac{n+s}2\right)}}\,r^{n-1}\,
\int_0^t\left[\int_{-\infty}^0\frac{dy_n}{ (x_n-y_n)^{1+s}}
\right]\,dx_n\\&&\qquad=
{\frac{2 \pi ^{\frac {2n-1}{2}} \Gamma\left(\frac{1+s}2\right)}{s\,\Gamma \left({\frac {n}{2}}\right)\Gamma\left(\frac{n+s}2\right)}}\,r^{n-1}\,
\int_0^t \frac{dx_n}{ x_n^s}\\&&\qquad=
{\frac{2 \pi ^{\frac {2n-1}{2}} \Gamma\left(\frac{1+s}2\right)}{s\,(1-s)\,\Gamma \left({\frac {n}{2}}\right)\Gamma\left(\frac{n+s}2\right)}}\,r^{n-1}\,t^{1-s},
\end{eqnarray*}
as desired.\end{proof}

\begin{lem}\label{LE:AUSJM-1aux2}
Let~$r$, $t>0$, $s\in(0,1)$,
\begin{eqnarray*}&& D:=\big\{x=(x',x_n)\in\R^n\,:\,|x'|<r {\mbox{ and }}x_n\in(0,t)\big\}
\\ {\mbox{and }}&& F:=\big\{x=(x',x_n)\in\R^n\,:\,|x'|>r {\mbox{ and }}x_n\in(0,t)\big\}.\end{eqnarray*}
Then,
$$ \iint_{D\times F}\frac{dx\,dy}{|x-y|^{n+s}}\le Ct\, r^{n-1-s}
,$$
for some~$C>0$ depending only on~$n$ and~$s$.
\end{lem}

\begin{proof} Differently from the proof of Lemma~\ref{LE:AUSJM-1aux1},
here it is convenient to exploit the substitutions~$\alpha:=\frac{x_n}{|x'-y'|}$ and~$\beta:=\frac{y_n}{|x'-y'|}$.
In this way we see that
\begin{eqnarray*}&&
\iint_{D\times F}\frac{dx\,dy}{|x-y|^{n+s}}\\&&\qquad=\int_{\{|x'|<r\}}\left[\int_{\{|y'|>r\}}
\left[\int_0^{t/|x'-y'|} \left[\int_0^{t/|x'-y'|} \frac{d\beta}{|x'-y'|^{n+s-2}(1+(\alpha-\beta)^2)^{\frac{n+s}2}}\right]\,d\alpha\right]\,dy'\right]\,dx'\\&&\qquad\le\int_{\{|x'|<r\}}\left[\int_{\{|y'|>r\}}
\left[\int_0^{t/|x'-y'|} \left[\int_0^{+\infty} \frac{d\gamma}{|x'-y'|^{n+s-2}(1+\gamma^2)^{\frac{n+s}2}}\right]\,d\alpha\right]\,dy'\right]\,dx'\\&&\qquad=C\int_{\{|x'|<r\}}\left[\int_{\{|y'|>r\}}
\left[\int_0^{t/|x'-y'|} \frac{d\alpha}{|x'-y'|^{n+s-2}}\right]\,dy'\right]\,dx'\\&&\qquad=Ct
\int_{\{|x'|<r\}}\left[\int_{\{|y'|>r\}}
\frac{dy'}{|x'-y'|^{n+s-1}}\right]\,dx'\\&&\qquad=Ct\,r^{n-1-s}
\int_{\{|X'|<1\}}\left[\int_{\{|Y'|>1\}}
\frac{dY'}{|X'-Y'|^{n+s-1}}\right]\,dX'\\&&\qquad =Ct\,r^{n-1-s},
\end{eqnarray*}
where, as customary, we took the freedom of renaming~$C$ line after line.
\end{proof}

Now, in the forthcoming Lemma~\ref{OEHDpimnSBlkOGFS-012rikjmg} we present a further technical result that detects suitable cancellations
involving ``thin sets''. This is a pivotal result to account for the nonlocal scenario.
Indeed, in the classical capillarity theory, to look for a competitor for a given set,
one can dig out a (small deformation of a)
cylinder with base radius equal to~$\e$ and height~$\delta\e$
and then add a ball with the same volume. A very convenient fact in this scenario is that
the surface error produced by the cylinder is of order~$\e^{n-1}$, while the one produced by the balls
are of order~$(\delta\e^n)^{\frac{n-1}n}=\delta^{\frac{n-1}n}\e^{n-1}$.
That is, for~$\delta$ suitably small, the surface tension produced by the new ball is negligible with respect
to the surface tension of the cylinder, thus allowing us to construct competitors in a nice and simple way.

Instead, in the nonlocal setting, for a given value of the fractional parameter, the corresponding nonlocal
surface tension produced by cylinders and balls of the same volume are comparable. This makes
the idea of ``adding a ball to compensate the loss of volume caused by removing a cylinder'' not suitable for the nonlocal framework.
Instead, as we will see in the proof
of Theorem~\ref{ESEMPIOPTO},
the volume compensation should occur through the addition of a suitably thin set placed at a regular point of
the droplet. The fact that the corresponding nonlocal surface energy produces a negligible contribution
will rely on the following result:

\begin{lem}\label{OEHDpimnSBlkOGFS-012rikjmg}
Let~$s\in(0,1)$, $0<\e<\delta<1$ and~$\eta\in(0,1)$. Let~$f\in C^{1,\alpha}_0\left(\R^{n-1},\,\left(-\frac\delta2,\frac\delta2\right)\right)$ for some~$\alpha\in(0,1)$
and assume that~$f(0)=0$ and~$\partial_i f(0)=0$ for all~$i\in\{1,\dots,n-1\}$.

Let~$\varphi\in C^\infty(\R^{n-1},[0,+\infty))$ be such that~$\varphi(x')=0$ whenever~$|x'|\ge1$
and~$\int_{\R^{n-1}}\varphi(x')\,dx'=1$.

Let
\begin{eqnarray*}&& \psi(x'):=\frac{\eta}{\e^{n-1}}\varphi\left( \frac{x'}{\e}\right),\\
&& {\mathcal{P}}:=\big\{ x=(x',x_n)\in\R^n \,:\, |x'|<\delta {\mbox{ and }}x_n>f(x')+\psi(x')\big\},\\
&& {\mathcal{Q}}:=\big\{ x=(x',x_n)\in\R^n \,:\,|x'|<\delta {\mbox{ and }} x_n\in\big(f(x'),\,f(x')+\psi(x')\big)\big\}\\
{\mbox{and }}&& {\mathcal{R}}:=\big\{ x=(x',x_n)\in\R^n \,:\, |x'|<\delta {\mbox{ and }} x_n<f(x')\big\}.\end{eqnarray*}
Then, there exist~$\delta_0\in(0,1)$
and~$C>0$, depending only on~$n$, $s$, $\alpha$, $f$ and~$\varphi$, such that if~$\delta<\delta_0$ and~$\eta<\delta_0\e^n$ then
$$ \left| \iint_{ {\mathcal{P}}\times {\mathcal{Q}}}\frac{dx\,dy}{|x-y|^{n+s}}-
\iint_{ {\mathcal{R}}\times {\mathcal{Q}}}\frac{dx\,dy}{|x-y|^{n+s}}
\right|\le C \left( \delta^\alpha+\frac{\eta}{\e^n}\right)
\,\e^{(n-1)s}\,\eta^{1-s}
.$$
\end{lem}

\begin{proof} The gist of this proof is to use a suitable reflection to simplify most of the integral contributions.
For this, we consider the map
$$ T(x):=\left( -x', \,2f(x')+\psi(x')-x_n\right).$$
We observe that when~$|x'|<\delta$ the distance between the Jacobian of~$T$ and minus the identity matrix is bounded from above by
$$ C\sup_{|x'|<\delta}\big( |\nabla f(x')|+|\nabla\psi(x')|\big)\le
C\sup_{|x'|<\delta}\left( |\nabla f(x')-\nabla f(0)|+\frac{\eta}{\e^n}\right)\le
C \left( \delta^\alpha+\frac{\eta}{\e^n}\right)
,$$
and the latter is a small quantity, as long as~$\delta_0$ is chosen sufficiently small.

Moreover, the condition~$T(x)\in{\mathcal{Q}}$ is equivalent to~$|x'|<\delta$
and~$2f(x')+\psi(x')-x_n\in\big(f(x'),\,f(x')+\psi(x')\big)$, which is in turn equivalent to~$x\in{\mathcal{Q}}$.

Similarly, the condition~$T(x)\in{\mathcal{P}}$ is equivalent to~$x\in{\mathcal{R}}$, as well as
the condition~$T(x)\in{\mathcal{R}}$ is equivalent to~$x\in{\mathcal{P}}$.

{F}rom these observations and the change of variable~$(X,Y):=(T(x),T(y))$ we arrive at
$$ \iint_{ {\mathcal{P}}\times {\mathcal{Q}}}\frac{dx\,dy}{|x-y|^{n+s}}=
\left(1+O\left( \delta^\alpha+\frac{\eta}{\e^n}\right)\right)
\iint_{ {\mathcal{R}}\times {\mathcal{Q}}}\frac{dX\,dY}{|X-Y|^{n+s}}.
$$
As a result,
\begin{equation}\label{teguwjyemTaneTPKMS0-2i5r}\left| \iint_{ {\mathcal{P}}\times {\mathcal{Q}}}\frac{dx\,dy}{|x-y|^{n+s}}-
\iint_{ {\mathcal{R}}\times {\mathcal{Q}}}\frac{dx\,dy}{|x-y|^{n+s}}
\right|\le C \left( \delta^\alpha+\frac{\eta}{\e^n}\right)\,
\iint_{ {\mathcal{R}}\times {\mathcal{Q}}}\frac{dx\,dy}{|x-y|^{n+s}}.
\end{equation}

Now we consider the transformation~$S(x):=(x',x_n-f(x'))$.
When~$|x'|<\delta$ the distance between the Jacobian of~$S$ and the identity matrix is bounded from above by
$$ C\sup_{|x'|<\delta} |\nabla f(x')|\le C \delta^\alpha.$$
Besides, if~$x\in{\mathcal{R}}$ then~$S(x)\in\big\{x\in\R^n\,:\,|x'|<\delta{\mbox{ and }}x_n<0\big\}$.
Also, if~$x\in{\mathcal{Q}}$ then
$$S(x)\in\{x\in\R^n\,:\,|x'|<\delta{\mbox{ and }}x_n\in(0,\psi(x'))\}\subseteq
\left\{x\in\R^n\,:\,|x'|<\e{\mbox{ and }}x_n\in\left(0,\frac{C\eta}{\e^{n-1}}\right)\right\}.$$
We stress that we are using here the fact that~$\psi(x')=0$ when~$|x'|\ge\e$.

{F}rom these remarks and~\eqref{teguwjyemTaneTPKMS0-2i5r},
using now the change of variable~$(X,Y):=(S(x),S(y))$, it follows that
\begin{eqnarray*}&&\left| \iint_{ {\mathcal{P}}\times {\mathcal{Q}}}\frac{dx\,dy}{|x-y|^{n+s}}-
\iint_{ {\mathcal{R}}\times {\mathcal{Q}}}\frac{dx\,dy}{|x-y|^{n+s}}
\right|\\&&\qquad\le C \left( \delta^\alpha+\frac{\eta}{\e^n}\right)\,
\iint_{ \{X_n<0\}\times
\left\{|Y'|<\e,\;Y_n\in\left(0,\frac{C\eta}{\e^{n-1}}\right)\right\}
}\frac{dX\,dY}{|X-Y|^{n+s}}.
\end{eqnarray*}
We can thus employ Lemma~\ref{LE:AUSJM-1aux1} with~$r:=\e$ and~$t:=\frac{C\eta}{\e^{n-1}}$
and conclude that
\begin{eqnarray*}\left| \iint_{ {\mathcal{P}}\times {\mathcal{Q}}}\frac{dx\,dy}{|x-y|^{n+s}}-
\iint_{ {\mathcal{R}}\times {\mathcal{Q}}}\frac{dx\,dy}{|x-y|^{n+s}}
\right|\le C \left( \delta^\alpha+\frac{\eta}{\e^n}\right)
\,\e^{n-1}\,\left(\frac{\eta}{\e^{n-1}}\right)^{1-s},
\end{eqnarray*}
from which the desired result follows.\end{proof}

With this preliminary work, we can now complete the proofs of Theorems~\ref{ESEMPIOPTO} and~\ref{ESEMPIOPTO:CO}.

\begin{proof}[Proof of Theorem~\ref{ESEMPIOPTO}]
Up to a rigid motion we can suppose that~$p=e_n$. We let~$\e>0$ and~$\delta>0$, to be taken as small as we wish in what follows.
We also define
$$ {\mathcal{B}}:=\Big\{
x=(x',x_n)\in B_1\setminus B_{1-\delta\e} \, : \, x_n>0 \mbox{ and } |x'|<\e \Big\}.$$
We stress that~${\mathcal{B}}\subseteq B_{\e_0/2}(p)\cap B_1$ as long as~$\e$ is small enough. Also, we 
pick a point~$q\in\mathrm{Reg}_E\cap\Omega$ and we modify the surface of~$\partial E$ in the normal direction
in an~$\e$-neighborhood of~$q$
by a set~${\mathcal{B}}'$ with~$|{\mathcal{B}}'|=
|{\mathcal{B}}|$, see Figure~\ref{GRANDEaddt}
and notice that the geometry of Lemma~\ref{OEHDpimnSBlkOGFS-012rikjmg} can be reproduced, up to a rigid motion.
We stress that~$\eta$ in Lemma~\ref{OEHDpimnSBlkOGFS-012rikjmg} corresponds to the volume of the perturbation
induced by~$\psi$, therefore in this setting we will apply
Lemma~\ref{OEHDpimnSBlkOGFS-012rikjmg} with~$\eta:=|{\mathcal{B}}'|=
|{\mathcal{B}}|\le C\delta\e^n$.

\begin{figure}[h]
\includegraphics[width=0.45\textwidth]{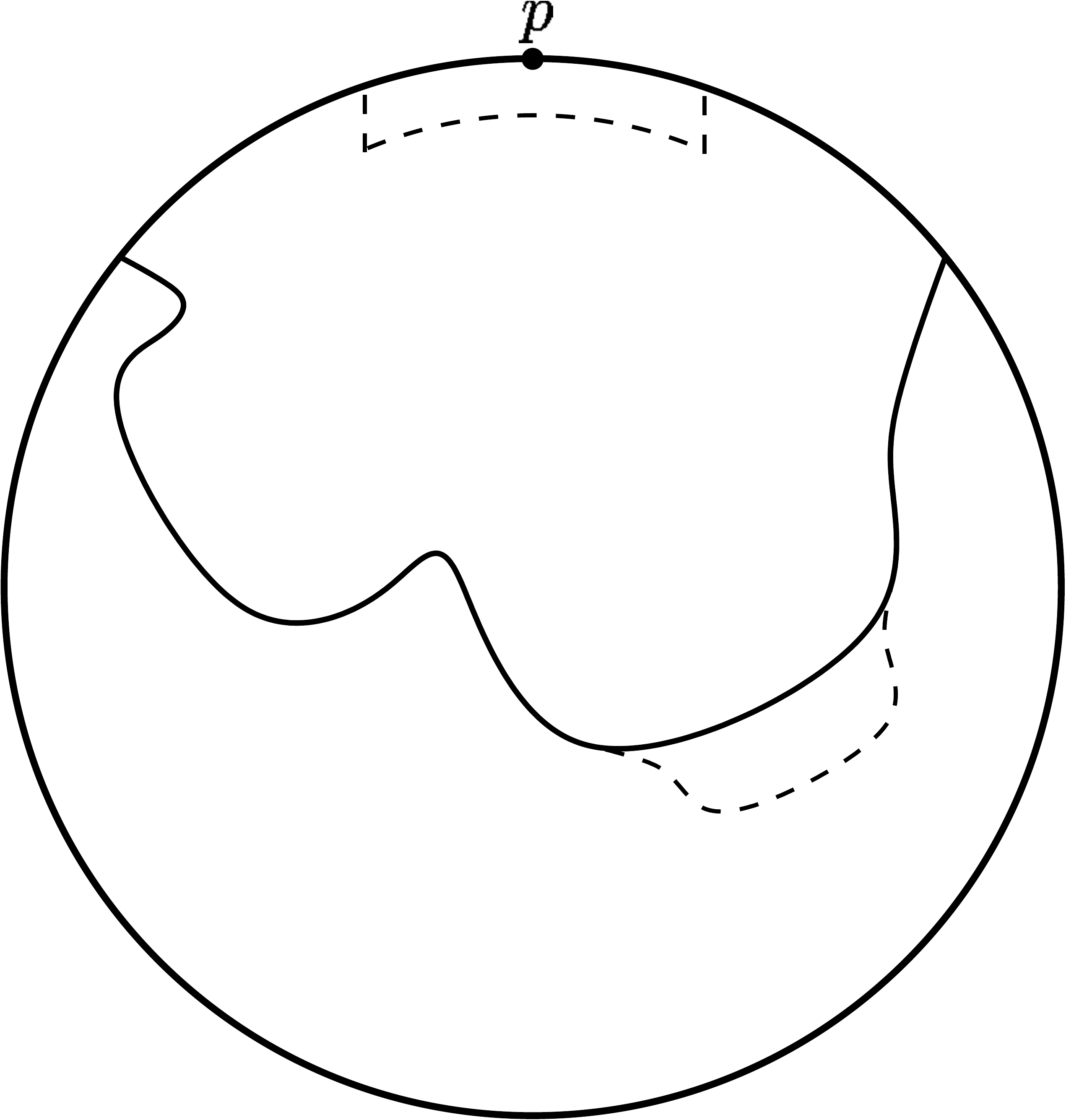}
\caption{Removing the thin set~${\mathcal{B}}$ to~$E$ near~$p$
and adding the thin set~${\mathcal{B}}'$ with the same volume.}
        \label{GRANDEaddt}
\end{figure}

We also denote by~$\Theta$ a cylinder centered at~$q$ (oriented by the normal of~${\mathcal{B}}'$ at~$q$)
of height equal to~$2\delta$ and radius of the basis equal to~$\delta$.
In this way, we have that if~$x\in{\mathcal{B}}'$ and~$y\in\R^n\setminus\Theta$
then~$|x-y|\ge|y-q|-|q-x|\ge\frac\delta2-C\e\ge\frac\delta4$, as long as~$\e$ is small enough, possibly in dependence of~$\delta$,
see Figure~\ref{GRANDEaddtTHETA}, whence
$$ I_1({\mathcal{B}}', B_1\setminus \Theta)\le 
C\int_{{\mathcal{B}}'\times B_1}
\frac{dx\,dy}{\delta^{n+s_1}}\le \frac{C\,|{\mathcal{B}}'|}{\delta^{n+s_1}}.
$$
Consequently,
\begin{equation}\label{JS-0jr-021roihfhFUJEWNfddmsa-1}
\begin{split}
I_1({\mathcal{B}}',B_1\setminus E\setminus {\mathcal{B}}')-I_1({\mathcal{B}}',E)&\le
I_1\big({\mathcal{B}}',(B_1\setminus E\setminus {\mathcal{B}}')\cap\Theta\big)-I_1({\mathcal{B}}',E\cap\Theta)+\frac{C |{\mathcal{B}}'|}{\delta^{n+s_1}}\\&\le
I_1\big({\mathcal{B}}',(B_1\setminus E\setminus {\mathcal{B}}')\cap\Theta\big)-I_1({\mathcal{B}}',E\cap\Theta)+\frac{C \e^n}{\delta^{n-1+s_1}},
\end{split}
\end{equation}
for some~$C>0$ that, as usual, gets renamed line after line.

\begin{figure}[h]
\includegraphics[width=0.45\textwidth]{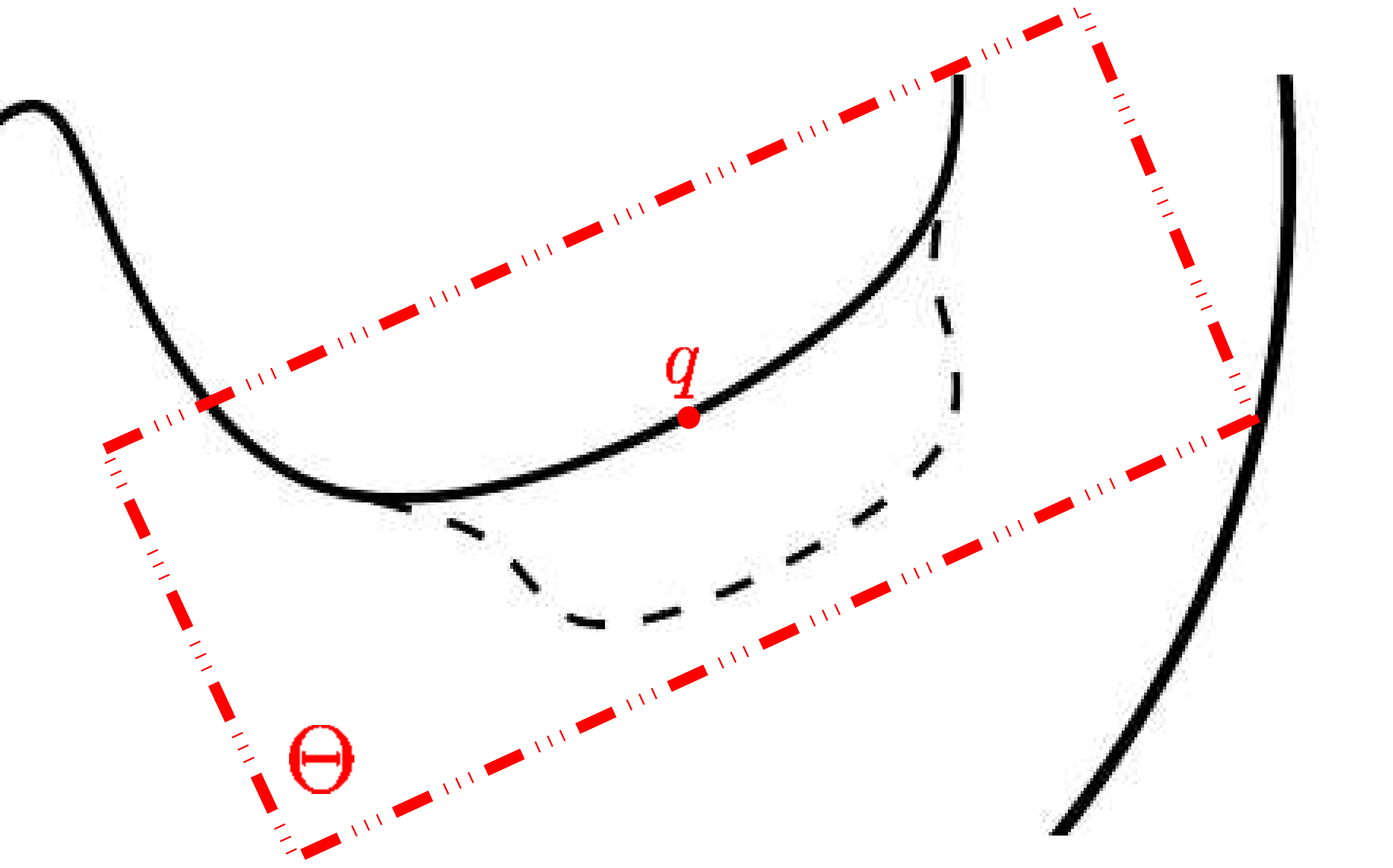}
\caption{Surrounding~${\mathcal{B}}'$ with a small cylinder~$\Theta$.}
        \label{GRANDEaddtTHETA}
\end{figure}

In view of Lemma~\ref{OEHDpimnSBlkOGFS-012rikjmg}, we also know that
$$ I_1\big({\mathcal{B}}',(B_1\setminus E\setminus {\mathcal{B}}')\cap\Theta\big)-I_1({\mathcal{B}}',E\cap\Theta)\le
C\delta^\alpha\e^{(n-1)s_1}\,(\delta\e^n)^{1-s_1}=C\delta^{1-s_1+\alpha}\e^{n-s_1}.
$$
This and~\eqref{JS-0jr-021roihfhFUJEWNfddmsa-1} lead to
\begin{equation}\label{JS-0jr-021roihfhFUJEWNfddmsa-2}
I_1({\mathcal{B}}',B_1\setminus E\setminus {\mathcal{B}}')-I_1({\mathcal{B}}',E)\le
C\delta^{1-s_1+\alpha}\e^{n-s_1}
+\frac{C \e^n}{\delta^{n-1+s_1}}.
\end{equation}

Now we claim that
\begin{equation}\label{0-uyfKJMS0DAVDM-14}\begin{split}&
I_1({\mathcal{B}},B_1\setminus E)+I_1({\mathcal{B}}',E)+
\sigma I_2({\mathcal{B}},B_1^c)\\&\qquad\le
I_1({\mathcal{B}},E\cup{\mathcal{B}}'\setminus{\mathcal{B}})
+I_1({\mathcal{B}}',B_1\setminus E\setminus {\mathcal{B}}')+I_1({\mathcal{B}},{\mathcal{B}}')
+\sigma I_2({\mathcal{B}}',B_1^c).\end{split}\end{equation}
To prove this, we construct a competitor for the minimal set~$E$ and compare their energies.
Indeed, the set~$\widetilde{E}:=(E\setminus{\mathcal{B}})\cup{\mathcal{B}}'$ is a competitor for~$E$, with the same volume of~$E$, and accordingly,
using the notation~$X:=E\setminus{\mathcal{B}}$ and~$Y:=B_1\setminus E\setminus{\mathcal{B}}'$,
\begin{eqnarray*}
0&\ge& {\mathcal{E}}({E})-{\mathcal{E}}(\widetilde{E})\\
&=& I_1(E,B_1\setminus E)-I_1(\widetilde E,B_1\setminus \widetilde E)
+\sigma I_2(E,B_1^c)-\sigma I_2(\widetilde E,B_1^c)\\
&=&  I_1(X\cup{\mathcal{B}},Y\cup{\mathcal{B}}')-I_1(X\cup{\mathcal{B}}',Y\cup{\mathcal{B}})
+\sigma I_2(X\cup{\mathcal{B}},B_1^c)-\sigma I_2(X\cup{\mathcal{B}}',B_1^c)\\&=&
I_1(X,{\mathcal{B}}')+
I_1({\mathcal{B}},Y)-I_1(X,{\mathcal{B}})-I_1({\mathcal{B}}',Y)
+\sigma I_2({\mathcal{B}},B_1^c)-\sigma I_2({\mathcal{B}}',B_1^c)\\&=&
\Big(I_1(E,{\mathcal{B}}')-I_1({\mathcal{B}},{\mathcal{B}}')\Big)+\Big(
I_1({\mathcal{B}},B_1\setminus E)-I_1({\mathcal{B}},{\mathcal{B}}')\Big)\\&&\qquad
-\Big(I_1(E\cup{\mathcal{B}}'\setminus{\mathcal{B}},{\mathcal{B}})-I_1({\mathcal{B}}',{\mathcal{B}}) \Big)
- I_1({\mathcal{B}}',B_1\setminus E\setminus{\mathcal{B}}')
+\sigma I_2({\mathcal{B}},B_1^c)-\sigma I_2({\mathcal{B}}',B_1^c)\\
&=&
I_1(E,{\mathcal{B}}')+
I_1({\mathcal{B}},B_1\setminus E)-I_1({\mathcal{B}},{\mathcal{B}}')
- I_1(E\cup{\mathcal{B}}'\setminus{\mathcal{B}},{\mathcal{B}})
- I_1({\mathcal{B}}',B_1\setminus E\setminus{\mathcal{B}}')\\&&\qquad
+\sigma I_2({\mathcal{B}},B_1^c)-\sigma I_2({\mathcal{B}}',B_1^c).
\end{eqnarray*}
This proves~\eqref{0-uyfKJMS0DAVDM-14}.

By combining~\eqref{JS-0jr-021roihfhFUJEWNfddmsa-2} and~\eqref{0-uyfKJMS0DAVDM-14} we find that
\begin{equation}\label{MSpr2idgekm:20e34r5o}\begin{split}&
I_1({\mathcal{B}},B_1\setminus E)+
\sigma I_2({\mathcal{B}},B_1^c)\\&\qquad\le
I_1({\mathcal{B}},E\cup{\mathcal{B}}'\setminus{\mathcal{B}})
+ I_1({\mathcal{B}},{\mathcal{B}}')
+\sigma I_2({\mathcal{B}}',B_1^c) +C\delta^{1-s_1+\alpha}\e^{n-s_1}
+\frac{C \e^n}{\delta^{n-1+s_1}}.\end{split}\end{equation}
Besides, since the distance between~${\mathcal{B}}'$ and~$B_1^c$ is bounded from below by a uniform
quantity, only depending on~$q$ and~$\e_0$ (and, in particular, independent of~$\e$), we have that
$$I_2({\mathcal{B}}',B_1^c)\le C|{\mathcal{B}}'|=C|{\mathcal{B}}|\le C\e^n,$$
for some~$C>0$ depending only on~$n$, $s_2$, $k_2$, $\e_0$, $q$ and the regularity of~$\partial E$
in the vicinity of~$q$. 
This and~\eqref{MSpr2idgekm:20e34r5o} yield that
\begin{equation}\label{MSpr2idgekm:20e34r5o:2}\begin{split}
\sigma I_2({\mathcal{B}},B_1^c)&\le I_1({\mathcal{B}},E\cup{\mathcal{B}}'\setminus{\mathcal{B}})
+I_1({\mathcal{B}},{\mathcal{B}}')
+C\e^n +C\delta^{1-s_1+\alpha}\e^{n-s_1}
+\frac{C \e^n}{\delta^{n-1+s_1}}
\\&\le I_1({\mathcal{B}},B_1\setminus{\mathcal{B}})
+I_1({\mathcal{B}},{\mathcal{B}}')
+C\e^n+C\delta^{1-s_1+\alpha}\e^{n-s_1}
+\frac{C \e^n}{\delta^{n-1+s_1}}
\\&\le I_1({\mathcal{B}},B_1\setminus{\mathcal{B}})+C\delta^{1-s_1+\alpha}\e^{n-s_1}
+\frac{C \e^n}{\delta^{n-1+s_1}}
,\end{split}\end{equation}
up to renaming~$C$ line after line.

Now, we use the change of variables~$X:=\frac{x-e_n}\e$ and~$Y:=\frac{y-e_n}\e$ to see that
\begin{equation}\label{CHDN:cipw-0202}
\begin{split}
& \e^{s_1-n} I_1({\mathcal{B}},B_1\setminus {\mathcal{B}})
= k_1 \,\e^{s_1-n}
\iint_{{\mathcal{B}}\times(B_1\setminus{\mathcal{B}})} \frac{dx\,dy}{|x-y|^{n+s_1}}
= k_1
\iint_{{\mathcal{Z}_\e}\times {\mathcal{A}}_\e} \frac{dX\,dY}{|X-Y|^{n+s_1}},\end{split}
\end{equation}
where
$$ {\mathcal{Z}}_\e:=\frac{ {\mathcal{B}}-e_n}\e=
\left\{ X\in\R^n\,:\,|X'|<1,\;X_n>-\frac1\e\,\;{\mbox{and}}\;\,
\left| X+\frac{e_n}\e\right|\in\left[\frac1\e-\delta,\frac1\e\right)
\right\}$$
and
$$ {\mathcal{A}}_\e:=\frac{ (B_1\setminus{\mathcal{B}})-e_n}\e=
{\mathcal{L}}_\e\cup{\mathcal{M}}_\e\cup{\mathcal{N}}_\e,$$
with
\begin{eqnarray*} 
{\mathcal{L}}_\e&:=&
\left\{ X\in\R^n\,:\,
\left| X+\frac{e_n}\e\right|<\frac1\e-\delta
\right\},
\\ {\mathcal{M}}_\e&:=&
\left\{ X\in\R^n\,:\,|X'|\ge1\;{\mbox{and}}\;\,
\left| X+\frac{e_n}\e\right|\in\left[\frac1\e-\delta,\frac1\e\right)
\right\}\\ {\mbox{and }}\qquad {\mathcal{N}}_\e&:=&
\left\{ X\in\R^n\,:\,|X'|<1,\;X_n\le-\frac1\e\,\;{\mbox{and}}\;\,
\left| X+\frac{e_n}\e\right|\in\left[\frac1\e-\delta,\frac1\e\right)
\right\}.
\end{eqnarray*}

Similarly,
\begin{equation}\label{MSpr2idgekm:20e34r5o:2-x-090}
\begin{split}
& \e^{s_2-n} I_2({\mathcal{B}},B_1^c)
= k_2\, \e^{s_2-n}
\iint_{{\mathcal{B}}\times B_1^c} \frac{dx\,dy}{|x-y|^{n+s_2}}
= k_2
\iint_{{\mathcal{Z}}_\e\times{\mathcal{O}}_\e} \frac{dX\,dY}{|X-Y|^{n+s_2}},\end{split}
\end{equation}
where
$$ {\mathcal{O}}_\e :=
\left\{ X\in\R^n\,:\,
\left| X+\frac{e_n}\e\right|\ge\frac1\e
\right\}.$$

Plugging~\eqref{CHDN:cipw-0202} and~\eqref{MSpr2idgekm:20e34r5o:2-x-090}
into~\eqref{MSpr2idgekm:20e34r5o:2},
we arrive at
\begin{equation}\label{CHDN:cipw-0201}
\sigma\,\e^{s_1-s_2} \,k_2
\iint_{{\mathcal{Z}}_\e\times{\mathcal{O}}_\e} \frac{dX\,dY}{|X-Y|^{n+s_2}}\le
k_1 \iint_{{\mathcal{Z}_\e}\times {\mathcal{A}}_\e} \frac{dX\,dY}{|X-Y|^{n+s_1}}+C\delta^{1-s_1+\alpha}
+\frac{C \e^{s_1}}{\delta^{n-1+s_1}}.\end{equation}

Now we claim that, if~$\e>0$ is suitably small, possibly
in depedence of~$\delta$, then
\begin{equation}\label{MSpr2idgekm:20e34r5o:2iu}
{\mathcal{B}}\subseteq
\big\{ x=(x',x_n)\in\R^n\,:\,|x'|<\e\;\;{\mbox{and}}\;\; x_n\in[1-(1+\delta)\delta\e,1)\big\}.
\end{equation}
Indeed, if~$x\in {\mathcal{B}}$ then
\begin{eqnarray*}&&
x_n=\sqrt{|x|^2-|x'|^2}\ge\sqrt{(1-\delta\e)^2-\e^2}
=\sqrt{1-2\delta\e +\delta^2\e^2-\e^2}\\&&\qquad\qquad\ge
\sqrt{1-2(1+\delta)\delta\e +(1+\delta)^2\delta^2\e^2}=
\sqrt{(1-(1+\delta)\delta\e)^2}=1-(1+\delta)\delta\e
\end{eqnarray*}
and this establishes~\eqref{MSpr2idgekm:20e34r5o:2iu}.

It follows from~\eqref{MSpr2idgekm:20e34r5o:2iu} that
\begin{equation}\label{MSpr2idgekm:20e34r5o:2iu:9u1j420}
{\mathcal{Z}}_\e\subseteq
\big\{ X=(X',X_n)\in\R^n\,:\,|X'|<1\;\;{\mbox{and}}\;\; X_n\in[-(1+\delta)\delta,0)\big\}=:{\mathcal{Z}}_\delta^\star.\end{equation}
Note also that
\begin{equation}\label{CHDN:cipw-0201-1}
\mathcal{O}_\e\supseteq \{ Y_n>0\}.\end{equation}
We now claim that
\begin{equation}\label{CHDN:cipw-0201-2}
{\mathcal{Z}}_\e\supseteq\left\{X\in\R^n\,:\, |X'|<1,\;X_n\in(-\delta,0)\,\;{\mbox{and}}\;\,
\left| X+\frac{e_n}\e\right|<\frac1\e
\right\}=:{\mathcal{W}}_\e.\end{equation}
To check this, suppose by contradiction that there exists~$X\in{\mathcal{W}}_\e$ with~$\left| X+\frac{e_n}\e\right|<\frac1\e-\delta$.
Then, we have that
\begin{eqnarray*}&&
0<\left(\frac1\e-\delta\right)^2-\left| X+\frac{e_n}\e\right|^2=
\frac1{\e^2}+\delta^2-\frac{2\delta}\e -|X'|^2-\left( X_n+\frac1\e\right)^2\\&&\qquad\qquad=
\delta^2-\frac{2\delta}\e -|X'|^2- X_n^2-\frac{2X_n}\e\le
\delta^2 -|X'|^2- X_n^2,
\end{eqnarray*}
that is~$|X|<\delta$,
and thus
\begin{eqnarray*}
\frac1\e-\delta>\left| X+\frac{e_n}\e\right| \ge\left|\frac{e_n}\e\right|-|X|=\frac1\e-|X|>
\frac1\e-\delta.
\end{eqnarray*}
This is a contradiction which establishes~\eqref{CHDN:cipw-0201-2}.

Hence, by~\eqref{CHDN:cipw-0201-1} and~\eqref{CHDN:cipw-0201-2},
we see that
\begin{equation}\label{CHDN:cipw-0201-3}
\begin{split}&
\iint_{{\mathcal{Z}}_\e\times{\mathcal{O}}_\e} \frac{dX\,dY}{|X-Y|^{n+s_2}}\ge
\iint_{{\mathcal{W}}_\e\times\{ Y_n>0\}} \frac{dX\,dY}{|X-Y|^{n+s_2}}\\&\qquad
\ge\iint_{{\mathcal{W}}_\delta^\star\times\{ Y_n>0\}} \frac{dX\,dY}{|X-Y|^{n+s_2}}-
\iint_{({\mathcal{W}}^\star_\delta\setminus{\mathcal{W}}_\e)\times\{ Y_n>0\}} \frac{dX\,dY}{|X-Y|^{n+s_2}},
\end{split}\end{equation}
where
$${\mathcal{W}}_\delta^\star:=
\big\{ X\in\R^n\,:\,|X'|<1\mbox{ and }X_n\in(-\delta,0)\big\}
.$$
Since
$$ \iint_{({\mathcal{W}}_\delta^\star\setminus{\mathcal{W}}_\e)\times\{ Y_n>0\}} \frac{dX\,dY}{|X-Y|^{n+s_2}}
\le \iint_{{\mathcal{W}}_\delta^\star\times\{ Y_n>0\}} \frac{dX\,dY}{|X-Y|^{n+s_2}}<+\infty$$
and
$$ \lim_{\e\searrow0}\big|{\mathcal{W}}_\delta^\star\setminus{\mathcal{W}}_\e\big|=0,$$
we have that
$$ \lim_{\e\searrow0}\iint_{({\mathcal{W}}_\delta^\star\setminus{\mathcal{W}}_\e)\times\{ Y_n>0\}} \frac{dX\,dY}{|X-Y|^{n+s_2}}=0$$
and, as a consequence, we infer from~\eqref{CHDN:cipw-0201-3} that
\begin{equation}\label{CHDN:cipw-0201-4} \liminf_{\e\searrow0}
\iint_{{\mathcal{Z}}_\e\times{\mathcal{O}}_\e} \frac{dX\,dY}{|X-Y|^{n+s_2}}\ge\iint_{{\mathcal{W}}_\delta^\star\times\{ Y_n>0\}} \frac{dX\,dY}{|X-Y|^{n+s_2}}.\end{equation}
We also note that if~$a\gg b>0$ then
$$ \left|\sqrt{a^2-b^2}-a+\frac{b^2}{2a}\right|=\left|a\sqrt{1-\frac{b^2}{a^2}}-a+\frac{b^2}{2a}\right|
\le \frac{b^4}{a^3},
$$
whence if~$X\in{\mathcal{N}_\e}$ then
\begin{eqnarray*}&&
-X_n-\frac2\e =
\left|X_n+\frac1\e\right|-\frac1\e=
\sqrt{\left|X+\frac{e_n}\e\right|^2-|X'|^2}-\frac1\e\\&&\quad\in\left[
\left|X+\frac{e_n}\e\right|-\frac{|X'|^2}{2\left|X+\frac{e_n}\e\right|}-
\frac{|X'|^4}{\left|X+\frac{e_n}\e\right|^3}
-\frac1\e,\;
\left|X+\frac{e_n}\e\right|-\frac{|X'|^2}{2\left|X+\frac{e_n}\e\right|}+
\frac{|X'|^4}{\left|X+\frac{e_n}\e\right|^3}
-\frac1\e
\right]\\&&\quad\subseteq\left[
-\delta-2\e
,\;2\e
\right]\subseteq[-2\delta,2\delta],
\end{eqnarray*} as long as~$\e$ is sufficiently small,
leading to
\begin{equation}\label{MZlijMSTZZmsdMAkdDy0-masc67bijnsdaAc}
|{\mathcal{N}}_\e|\le 
\left|
\left\{ X\in\R^n\,:\,|X'|<1,\,\;{\mbox{and}}\;\,
X_n\in\left[-\frac2\e-2\delta,-\frac2\e+2\delta\right]
\right\}
\right|\le
C\delta.\end{equation}

Furthermore, if~$X\in{\mathcal{Z}_\e}$ then~$X_n\ge-(1+\delta)\delta$, thanks to~\eqref{MSpr2idgekm:20e34r5o:2iu},
and therefore if~$Y\in {\mathcal{N}}_\e$ we have that
$$ |X-Y|\ge X_n-Y_n\ge -(1+\delta)\delta+\frac1\e\ge \frac{1}{2\e}.$$
This and~\eqref{MZlijMSTZZmsdMAkdDy0-masc67bijnsdaAc} yield that
\begin{equation}\label{98uytgfguiuhgvcvghjhbvb092r5tPKJNSkdmwa}
\iint_{{\mathcal{Z}_\e}\times {\mathcal{N}}_\e} \frac{dX\,dY}{|X-Y|^{n+s_1}}\le
C\e^{n+s_1}\,|{\mathcal{Z}_\e}|\;| {\mathcal{N}}_\e|\le C\e^{n+s_1}.
\end{equation}
Now we set
$$ {\mathcal{M}}_\e':={\mathcal{M}}_\e\cap B_2\qquad{\mbox{and}}\qquad
{\mathcal{M}}_\e'':={\mathcal{M}}_\e\setminus B_2.$$
We remark that, if~$\e>0$ is suitably small, possibly
in depedence of~$\delta$, then
\begin{equation}\label{PALSKMxdcio8we9ghknqkdhfFosdfv}
{\mathcal{M}}_\e'\subseteq\big\{X\in\R^n\,:\,|X'|\in[1,2]{\mbox{ and }}X_n\in[-(1+\delta)\delta,0)\big\}=:
{\mathcal{M}}^\star_\delta.
\end{equation}
Indeed, if~$X\in{\mathcal{M}}_\e'$ then~$|X'|\ge1$ and~$|X'|\le|X|<2$. Furthermore,
$$ 1+\left| X_n+\frac{1}\e\right|^2\le |X'|^2+\left| X_n+\frac{1}\e\right|^2=
\left| X+\frac{e_n}\e\right|^2\le\frac1{\e^2}$$
which gives that~$X_n<0$.

Moreover,
\begin{eqnarray*}
4+\left| X_n+\frac{1}\e\right|^2\ge
|X'|^2+\left| X_n+\frac{1}\e\right|^2=
\left| X+\frac{e_n}\e\right|^2\ge\left(\frac1\e-\delta\right)^2.
\end{eqnarray*}
Since~$X_n\ge-|X|\ge-2$, this gives that
\begin{eqnarray*}&&
X_n+\frac{1}\e
=\sqrt{\left| X_n+\frac{1}\e\right|^2}
\ge\sqrt{\left(\frac1\e-\delta\right)^2-4}=
\sqrt{\frac1{\e^2}-\frac{2\delta}\e+\delta^2-4}\\&&\qquad=\frac1\e
\sqrt{1- 2\delta \e+\delta^2\e^2-4\e^2}\ge\frac1\e(1-(1+\delta)\delta\e)\end{eqnarray*}
and accordingly~$X_n\ge-(1+\delta)\delta$. These observations complete the proof of~\eqref{PALSKMxdcio8we9ghknqkdhfFosdfv}.

We now use~\eqref{PALSKMxdcio8we9ghknqkdhfFosdfv}
in combination with~\eqref{MSpr2idgekm:20e34r5o:2iu:9u1j420}.
In this way,
we see that
\begin{equation}\label{PALSKMxdcio8we9ghknqkdhfFosdfv2}
\iint_{{\mathcal{Z}_\e}\times {\mathcal{M}}'_\e} \frac{dX\,dY}{|X-Y|^{n+s_1}}\le
\iint_{{\mathcal{Z}_\delta^\star}\times {\mathcal{M}}^\star_\delta} \frac{dX\,dY}{|X-Y|^{n+s_1}}.
\end{equation}

Besides, if~$X\in{\mathcal{Z}_\e}$ and~$Y\in {\mathcal{M}}''_\e$ then~$|X-Y|\ge|Y|-|X|\ge 2-\frac32=\frac12$
and, as a result,
\begin{eqnarray*}
\iint_{{\mathcal{Z}_\e}\times {\mathcal{M}}''_\e} \frac{dX\,dY}{|X-Y|^{n+s_1}}\le
C \,|{\mathcal{Z}_\e}|\,\int_{\R^n\setminus B_{1/2}}\frac{dZ}{|Z|^{n+s_1}}\le C\delta.
\end{eqnarray*}
Combining this and~\eqref{PALSKMxdcio8we9ghknqkdhfFosdfv2} we conclude that
\begin{equation*}
\iint_{{\mathcal{Z}_\e}\times {\mathcal{M}}_\e} \frac{dX\,dY}{|X-Y|^{n+s_1}}\le
\iint_{{\mathcal{Z}_\delta^\star}\times {\mathcal{M}}^\star_\delta} \frac{dX\,dY}{|X-Y|^{n+s_1}}+C\delta.
\end{equation*}
Using the latter inequality and~\eqref{98uytgfguiuhgvcvghjhbvb092r5tPKJNSkdmwa}
we obtain that
\begin{equation}\label{98uytgfguiuhgvcvghjhbvb092r5tPKJNSkdmwa-F}
\begin{split}&
\limsup_{\e\searrow0}
\iint_{{\mathcal{Z}_\e}\times {\mathcal{A}}_\e} \frac{dX\,dY}{|X-Y|^{n+s_1}}\\&\qquad\qquad\le
\iint_{{\mathcal{Z}_\delta^\star}\times {\mathcal{M}}^\star_\delta} \frac{dX\,dY}{|X-Y|^{n+s_1}}+C\delta+
\limsup_{\e\searrow0}
\iint_{{\mathcal{Z}_\e}\times {\mathcal{L}}_\e} \frac{dX\,dY}{|X-Y|^{n+s_1}}
.\end{split}\end{equation}

Now we consider the map
$$ \big\{\R^n:\,|X'|<2\big\}\ni X=(X',X_n)\longmapsto T(X):=\left(X',\,
X_n-\sqrt{\left(\frac1\e-\delta\right)^2-|X'|^2}+\frac1\e
\right)
$$
and we observe that if~$X\in{\mathcal{Z}}_\e$ then~$\underline{X}:=T(X)$
satisfies~$|\underline{X}'|<1$
and
\begin{eqnarray*}&&
\underline{X}_n=\left|X_n+\frac1\e\right|-\sqrt{\left(\frac1\e-\delta\right)^2-|X'|^2}
=\sqrt{\left|X+\frac{e_n}\e\right|^2-|X'|^2}-\sqrt{\left(\frac1\e-\delta\right)^2-|X'|^2}\\&&\qquad\qquad
\in\left[0,\,\sqrt{\frac1{\e^2}-|X'|^2}-\sqrt{\left(\frac1\e-\delta\right)^2-|X'|^2}\right)\subseteq
[0,(1+\delta)\delta].
\end{eqnarray*}
In addition, if~$Y\in {\mathcal{L}}_\e':= {\mathcal{L}}_\e\cap B_2$
and~$\underline{Y}:=T(Y)$, we have that~$|\underline{Y'}|<2$ and
\begin{eqnarray*}&&
\underline{Y}_n\le\left|Y_n+\frac1\e\right|-\sqrt{\left(\frac1\e-\delta\right)^2-|Y'|^2}
=\sqrt{\left|Y+\frac{e_n}\e\right|^2-|Y'|^2}-\sqrt{\left(\frac1\e-\delta\right)^2-|Y'|^2}\le0.
\end{eqnarray*}
We also observe that the distance of the Jacobian matrix of~$T$ from the identity is bounded from above by
$$ C\left|D_{X'} \sqrt{\left(\frac1\e-\delta\right)^2-|X'|^2}\right|\le
\frac{C|X'|}{ \sqrt{\left(\frac1\e-\delta\right)^2-|X'|^2}}\le C\e,
$$
yielding that, in the above notation,~$|\underline{X}-\underline{Y}|\le(1+ C\e)|X-Y|$, with the freedom, as usual,
of renaming~$C$.

These observations allow us to conclude that
\begin{equation}\label{KJLND7fewishdkfbweoikhf3oiwgrtp32uty5739p4uiwegfibuwejksbv}
\iint_{{\mathcal{Z}_\e}\times {\mathcal{L}}_\e'} \frac{dX\,dY}{|X-Y|^{n+s_1}}\le
(1+C\e)\iint_{{\mathcal{X}}_\delta^\star\times {\mathcal{Y}}^\star} \frac{d\underline{X}\,d\underline{Y}}{|\underline{X}-\underline{Y}|^{n+s_1}}
\end{equation}
where
\begin{eqnarray*}&&
{\mathcal{X}}_\delta^\star:=\big\{ X\in\R^n\,:\,|X'|<1 {\mbox{ and }} X_n\in(0,(1+\delta)\delta)\big\}\\
{\mbox{and }}&&{\mathcal{Y}}^\star:=\big\{ X\in\R^n\,:\,|X'|<2 {\mbox{ and }} X_n<0\big\}.
\end{eqnarray*}
Also, setting~${\mathcal{L}}_\e'':= {\mathcal{L}}_\e\setminus B_2$,
we have that
$$ \iint_{{\mathcal{Z}_\e}\times {\mathcal{L}}_\e''} \frac{dX\,dY}{|X-Y|^{n+s_1}}\le
C\,|{\mathcal{Z}_\e}|\int_{\R^n\setminus B_{1/2}}\frac{dZ}{|Z|^{n+s_1}}\le C\delta.$$
Combining this inequality and~\eqref{KJLND7fewishdkfbweoikhf3oiwgrtp32uty5739p4uiwegfibuwejksbv}
we find that
\begin{equation*}
\iint_{{\mathcal{Z}_\e}\times {\mathcal{L}}_\e} \frac{dX\,dY}{|X-Y|^{n+s_1}}\le(1+
C\e)\iint_{{\mathcal{X}}_\delta^\star\times {\mathcal{Y}}^\star} \frac{d\underline{X}\,d\underline{Y}}{|\underline{X}-\underline{Y}|^{n+s_1}}+C\delta.
\end{equation*}
{F}rom this and~\eqref{98uytgfguiuhgvcvghjhbvb092r5tPKJNSkdmwa-F} we arrive at
\begin{equation*}
\begin{split}&
\limsup_{\e\searrow0}
\iint_{{\mathcal{Z}_\e}\times {\mathcal{A}}_\e} \frac{dX\,dY}{|X-Y|^{n+s_1}}\\&\qquad\qquad\le
\iint_{{\mathcal{Z}_\delta^\star}\times {\mathcal{M}}^\star_\delta} \frac{dX\,dY}{|X-Y|^{n+s_1}}+
\limsup_{\e\searrow0}(1+C\e)\iint_{{\mathcal{X}}_\delta^\star\times {\mathcal{Y}}^\star} \frac{d\underline{X}\,d\underline{Y}}{|\underline{X}-\underline{Y}|^{n+s_1}}+C\delta
.\end{split}\end{equation*}
Thus, given~$\delta>0$, to be taken conveniently small, we consider the limit~$\e\searrow0$
and we deduce from the latter inequality, \eqref{CHDN:cipw-0201} and~\eqref{CHDN:cipw-0201-4}
that, as~$\e\searrow0$,
\begin{equation}\label{CHDN:cipw-0201-4AB}\begin{split}&
\sigma\,\e^{s_1-s_2} \,k_2\left(
\iint_{{\mathcal{W}}_\delta^\star\times\{ Y_n>0\}} \frac{dX\,dY}{|X-Y|^{n+s_2}}+o(1)
\right)\\&\quad\le
k_1 \left(
\iint_{{\mathcal{Z}_\delta^\star}\times {\mathcal{M}}^\star_\delta} \frac{dX\,dY}{|X-Y|^{n+s_1}}+(1+C\e)\iint_{{\mathcal{X}}_\delta^\star\times {\mathcal{Y}}^\star} \frac{d\underline{X}\,d\underline{Y}}{|\underline{X}-\underline{Y}|^{n+s_1}} \right)\\&\qquad\qquad+C\delta+C\delta^{1-s_1+\alpha}
+\frac{C \e^{s_1}}{\delta^{n-1+s_1}}.\end{split}\end{equation}
This yields that necessarily
\begin{equation}\label{CHDN:cipw-0201-55}
s_1\ge s_2.\end{equation}

Furthermore, if~$s_1=s_2$ then we obtain, passing to the limit~\eqref{CHDN:cipw-0201-4AB}
as~$\e\searrow0$, that
\begin{equation}\label{CHDN:cipw-0201-55-AB}\begin{split}&
\sigma \,k_2
\iint_{{\mathcal{W}}_\delta^\star\times\{ Y_n>0\}} \frac{dX\,dY}{|X-Y|^{n+s_1}}\\&\quad\le
k_1 \left(
\iint_{{\mathcal{Z}_\delta^\star}\times {\mathcal{M}}^\star_\delta} \frac{dX\,dY}{|X-Y|^{n+s_1}}+\iint_{{\mathcal{X}}_\delta^\star\times {\mathcal{Y}}^\star} \frac{d\underline{X}\,d\underline{Y}}{|\underline{X}-\underline{Y}|^{n+s_1}} \right)
+C\delta+C\delta^{1-s_1+\alpha}.\end{split}\end{equation}
We are now ready to send~$\delta\searrow0$. To this end, we multiply~\eqref{CHDN:cipw-0201-55-AB}
by~$\delta^{s_1-1}$ and we make use of Lemmata~\ref{LE:AUSJM-1aux1} and~\ref{LE:AUSJM-1aux2} to find that
\begin{eqnarray*}
c_\star\,\sigma \,k_2&=&
\lim_{\delta\searrow0}\sigma \,k_2\,\delta^{s_1-1}
\iint_{{\mathcal{W}}_\delta^\star\times\{ Y_n>0\}} \frac{dX\,dY}{|X-Y|^{n+s_1}}\\&\le&\lim_{\delta\searrow0}\left[
k_1\,\delta^{s_1-1} \left(
\iint_{{\mathcal{Z}_\delta^\star}\times {\mathcal{M}}^\star_\delta} \frac{dX\,dY}{|X-Y|^{n+s_1}}+\iint_{{\mathcal{X}}_\delta^\star\times {\mathcal{Y}}^\star} \frac{d\underline{X}\,d\underline{Y}}{|\underline{X}-\underline{Y}|^{n+s_1}} \right)
+C\delta^{s_1}+C\delta^{\alpha}\right]\\&\le&\lim_{\delta\searrow0}\left[
C\delta^{s_1}(1+\delta)+c_\star\, k_1\,(1+\delta)^{1-s_1}+C\delta^{s_1}+C\delta^{\alpha}\right]\\&=&c_\star\, k_1
\end{eqnarray*}
and therefore~$\sigma k_2\le k_1$.
Thanks to this, we have that,
to complete the proof of Theorem~\ref{ESEMPIOPTO}, it only remains to rule out
the case~$s_1=s_2$ and~$k_1=\sigma k_2$. In this situation,
$$ {\mathcal{C}}(F)=\mathcal{E}(F)=k_1\iint_{F\times F^c}\frac{dx\,dy}{|x-y|^{n+s_1}},$$
hence all the minimizers with prescribed volume correspond to balls, thanks to~\cite{MR2469027}.
But this violates the assumptions about the point~$p$ in Theorem~\ref{ESEMPIOPTO}.\end{proof}

\begin{proof}[Proof of Theorem~\ref{ESEMPIOPTO:CO}]
This can be seen as a counterpart of Theorem~\ref{ESEMPIOPTO} based on complementary sets.
For this argument, we denote by~${\mathcal{C}}_\sigma$, instead of~$\mathcal{C}$, the functional in~\eqref{variationalproblem},
in order to showcase explicitly its dependence on the
relative adhesion coeﬃcient~$\sigma$. Thus, in the setting of Theorem~\ref{ESEMPIOPTO:CO},
if~$F\subseteq\Omega$ and~$\widetilde{F}:=\Omega\setminus F$,
\begin{eqnarray*}
{\mathcal{C}}_\sigma(\widetilde{F})&=&
I_1\big(\Omega\setminus F, (\Omega\setminus F)^c\cap\Omega\big)+\sigma\, I_2(\Omega\setminus F,\Omega^c)
\\&=&I_1(\Omega\setminus F, F)+\sigma\, I_2(\Omega\setminus F,\Omega^c)\\&=&
{\mathcal{C}}_{-\sigma}(F)+\sigma\, I_2( F,\Omega^c)+\sigma\, I_2(\Omega\setminus F,\Omega^c)\\
&=&{\mathcal{C}}_{-\sigma}(F)+\sigma\, I_2( \Omega,\Omega^c).
\end{eqnarray*}
Since the latter term does not depend on~$F$, we see that if~$E$, as in the statement of Theorem~\ref{ESEMPIOPTO:CO}, is
a volume-constrained minimizer of~$\mathcal{C}_{\sigma}$, then~$\widetilde{E}:=\Omega\setminus E$ is
a volume-constrained minimizer of~$\mathcal{C}_{-\sigma}$. Now, the set~$\widetilde{E}$ fulfills the assumptions of
Theorem~\ref{ESEMPIOPTO} with~$\sigma$ replaced by~$-\sigma$. It follows that
either~$s_1>s_2$, or~$s_1=s_2$ and~$k_1>-\sigma k_2$, as desired.
\end{proof}

\section{Unique determination of the contact angle and proof of Theorem~\ref{CO:AN:CO}}\label{CO:AN:CO:S}

Here we discuss the existence and uniqueness theory for the equation that prescribes
the nonlocal angle of contact between the droplet and the container. This analysis will ultimately lead to the proof
of Theorem~\ref{CO:AN:CO}: for this, it is convenient to perform some integral computations
in order to appropriately rewrite integral interactions involving cones, detecting cancellations,
using a dimensional reduction argument and a well designed notation of polar angle with respect to the kernel singularity.
The details go as follows.

\begin{lem}\label{JSIMDESEM}
In the notation of~\eqref{Jtheta1theta2},
\eqref{Jtheta1theta2-Di2}, \eqref{Jtheta1theta2-Di3} and~\eqref{Jtheta1theta2-Di4}, if~$\vartheta\in(0,\pi)$, then
\begin{equation}\label{D13A5D8-KS-pokme-1}\begin{split}&
\int_{J_{\vartheta,\pi}}\frac{a_1(\overrightarrow{x-e(\vartheta)})}{|x-e(\vartheta)|^{n+s_1}}\, dx-
\int_{J_{0,\vartheta}}\frac{a_1(\overrightarrow{x-e(\vartheta)})}{|x-e(\vartheta)|^{n+s_1}}\, dx\\&\qquad=
\frac{1}{s_1 (\sin\vartheta)^{s_1}}\left(
\int_{0}^\vartheta\phi_1(\alpha)\,(\sin\alpha)^{s_1}\,d\alpha
-\int_{\vartheta}^\pi \phi_1(\alpha)\,(\sin\alpha)^{s_1}\,d\alpha\right).\end{split}\end{equation}
\end{lem}

\begin{figure}[h]
\includegraphics[width=0.85\textwidth]{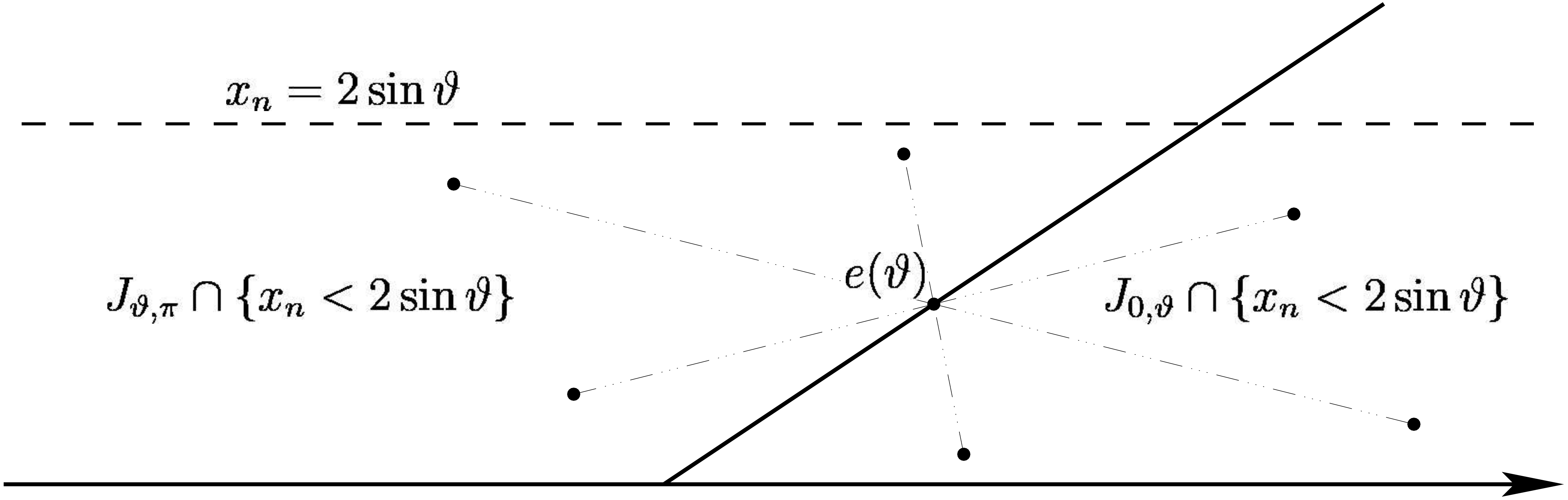}
\caption{A geometric argument involved in the proof of Lemma~\ref{JSIMDESEM}.}
        \label{JSIMDESEM2}
\end{figure}

\begin{proof} We stress that each of the integrals on the left hand side of~\eqref{D13A5D8-KS-pokme-1} is divergent, hence the two terms have to be considered together, in the principal value sense.
However, for typographical convenience, we will formally
act on the integrals by omitting the principal value notation and
perform the cancellations necessary to have only finite contributions to obtain the desired result.

To this end, we recall~\eqref{Jtheta1theta2}
and observe that~$x\in J_{0,\vartheta}\cap\{ x_n<2\sin\vartheta\}$ if and only if~$z:=2e(\vartheta)-x\in
J_{\vartheta,\pi}\cap\{ x_n<2\sin\vartheta\}$, see Figure~\ref{JSIMDESEM2}. Hence,
by the symmetry of~$a_1$,
\begin{eqnarray*}
&&\int_{J_{0,\vartheta}\cap\{ x_n<2\sin\vartheta\}}\frac{a_1(\overrightarrow{x-e(\vartheta)})}{|x-e(\vartheta)|^{n+s_1}}\, dx
=\int_{J_{\vartheta,\pi}\cap\{ z_n<2\sin\vartheta\}}\frac{a_1(\overrightarrow{z-e(\vartheta)})}{|z-e(\vartheta)|^{n+s_1}}\, dz.
\end{eqnarray*}
Consequently, if we denote by~$\Upsilon$ the left hand side of~\eqref{D13A5D8-KS-pokme-1}, we see after a cancellation that
\begin{equation}\label{01930-1ifjHDknfmr22PIKa-lmcwutgf980ihg}
\begin{split}
\Upsilon=
\int_{J_{\vartheta,\pi}\cap\{ x_n>2\sin\vartheta\}}\frac{a_1(\overrightarrow{x-e(\vartheta)})}{|x-e(\vartheta)|^{n+s_1}}\, dx-
\int_{J_{0,\vartheta}\cap\{ x_n>2\sin\vartheta\}}\frac{a_1(\overrightarrow{x-e(\vartheta)})}{|x-e(\vartheta)|^{n+s_1}}\, dx.
\end{split}
\end{equation}
It is useful now to reduce the problem to that in dimension~$2$. To this end, we adopt the notation in~\eqref{Jtheta1theta2-Di2} and~\eqref{Jtheta1theta2-Di3} and note that
\begin{equation}\label{KSM X-10oqlrnfXcsw}
\begin{split}
& \int_{J_{\vartheta,\pi}\cap\{ x_n>2\sin\vartheta\}}\frac{a_1(\overrightarrow{x-e(\vartheta)})}{|x-e(\vartheta)|^{n+s_1}}\, dx \\=&
\iiint\limits_{\{(x_1,x_n)\in J_{\vartheta,\pi}^\star,\;\bar{x}\in\R^{n-2},\; x_n>2\sin\vartheta\}}\frac{a_1\Big(\overrightarrow{
(x_1-\cos\vartheta)e_1+(x_n-\sin\vartheta)e_n+(0,\bar{x},0)}\Big)}{\Big(
(x_1-\cos\vartheta)^2+(x_n-\sin\vartheta)^2+|\bar{x}|^2\Big)^{\frac{n+s_1}2}
}\,d\bar{x}\, dx_1\,dx_n\\=&
\iint\limits_{\{y=(y_1,y_2)\in J_{\vartheta,\pi}^\star,\;\bar{y}\in\R^{n-2},\; y_2>2\sin\vartheta\}}\frac{a_1\Big(\overrightarrow{
(y_1-\cos\vartheta)\, e_1+
(y_2-\sin\vartheta)\,e_n+|y-e^\star(\vartheta)|(0,\bar{y},0)}\Big)}{
|y-e^\star(\vartheta)|^{{2+s_1}}\;
\big(1+|\bar{y}|^2\big)^{\frac{n+s_1}2}
}\,d\bar{y}\, dy
\\=&
\int_{J_{\vartheta,\pi}^\star\cap\{ y_2>2\sin\vartheta\}}\frac{a_1^\star(\overrightarrow{
y-e^\star(\vartheta)})}{
|y-e^\star(\vartheta)|^{{2+s_1}}}\,dy.
\end{split}\end{equation}
Similarly,
$$ \int_{J_{0,\vartheta}\cap\{ x_n>2\sin\vartheta\}}\frac{a_1(\overrightarrow{x-e(\vartheta)})}{|x-e(\vartheta)|^{n+s_1}}\, dx=
\int_{J_{0,\vartheta}^\star\cap\{ y_2>2\sin\vartheta\}}\frac{a_1^\star(\overrightarrow{
y-e^\star(\vartheta)})}{
|y-e^\star(\vartheta)|^{{2+s_1}}}\,dy.$$
Thanks to these observations, we rewrite~\eqref{01930-1ifjHDknfmr22PIKa-lmcwutgf980ihg} in the form
\begin{equation}\label{01930-1ifjHDknfmr22PIKa-lmcwutgf980ihg-RE3}
\begin{split}
\Upsilon=
\int_{J_{\vartheta,\pi}^\star\cap\{ x_2>2\sin\vartheta\}}\frac{a_1^\star(\overrightarrow{x-e^\star(\vartheta)})}{|x-e^\star(\vartheta)|^{2+s_1}}\, dx-
\int_{J_{0,\vartheta}^\star\cap\{ x_2>2\sin\vartheta\}}\frac{a_1^\star(\overrightarrow{x-e^\star(\vartheta)})}{|x-e^\star(\vartheta)|^{2+s_1}}\, dx.
\end{split}
\end{equation}

Now we use polar coordinates centered at~$e^\star(\vartheta)$.
For this, if~$x\in J_{0,\vartheta}^\star\cap\{ x_2>2\sin\vartheta\}$, we write~$x=(\cos\vartheta,\sin\vartheta)+\rho(\cos\alpha,\sin\alpha)$ with~$\alpha\in(0,\vartheta)$ and~$\rho>\frac{\sin\vartheta}{\sin\alpha}$.
Similarly, if~$x\in J_{\vartheta,\pi}^\star\cap\{ x_2>2\sin\vartheta\}$, we write~$x=(\cos\vartheta,\sin\vartheta)+\rho(\cos\beta,\sin\beta)$ with~$\beta\in(\vartheta,\pi)$ and~$\rho>\frac{\sin\vartheta}{\sin\beta}$, see Figure~\ref{JSIMDESEM2.2}.

\begin{figure}[h]
\includegraphics[width=0.85\textwidth]{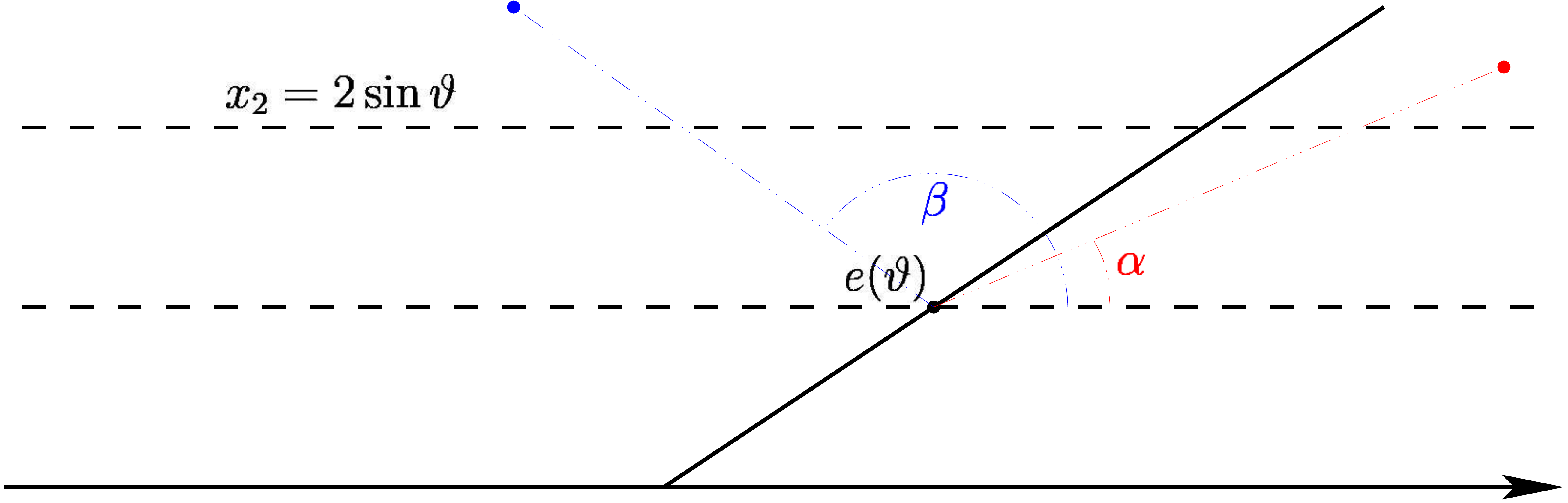}
\caption{Another geometric argument involved in the proof of Lemma~\ref{JSIMDESEM}.}
        \label{JSIMDESEM2.2}
\end{figure}

As a result, using the notation in~\eqref{Jtheta1theta2-Di4}, we deduce
from~\eqref{01930-1ifjHDknfmr22PIKa-lmcwutgf980ihg-RE3} that
\begin{equation*}
\begin{split}
\Upsilon&=
\iint_{(0,\vartheta)\times\left(\frac{\sin\vartheta}{\sin\alpha},+\infty\right)}\frac{\phi_1(\alpha)}{\rho^{1+s_1}}\, d\alpha\,d\rho-
\iint_{(\vartheta,\pi)\times\left(\frac{\sin\vartheta}{\sin\beta},+\infty\right)}\frac{\phi_1(\beta)}{\rho^{1+s_1}}\, d\beta\,d\rho\\&=
\frac{1}{s_1 (\sin\vartheta)^{s_1}}\left(
\int_{0}^\vartheta\phi_1(\alpha)\,(\sin\alpha)^{s_1}\,d\alpha
-\int_{\vartheta}^\pi \phi_1(\beta)\,(\sin\beta)^{s_1}\,d\beta\right),
\end{split}
\end{equation*}
which establishes~\eqref{D13A5D8-KS-pokme-1}.
\end{proof}

\begin{lem}\label{JSIMDESEMa2}
Let the notation in~\eqref{Jtheta1theta2},
\eqref{Jtheta1theta2-Di2}, \eqref{Jtheta1theta2-Di3} and~\eqref{Jtheta1theta2-Di4} hold true.
Then,
\begin{equation}\label{KSM X-10oqlrnfXcsw-2}
\int_{H^c}\frac{a_2(\overrightarrow{e(\vartheta)-x})}{|e(\vartheta)-x|^{n+s_1}}\,dx=
\frac{1}{s_1(\sin\vartheta)^{s_1}}\int_{-\pi}^0
\phi_2(\alpha)\,|\sin\alpha|^{s_1}\,d\alpha.\end{equation}
\end{lem}

\begin{figure}[h]
\includegraphics[width=0.85\textwidth]{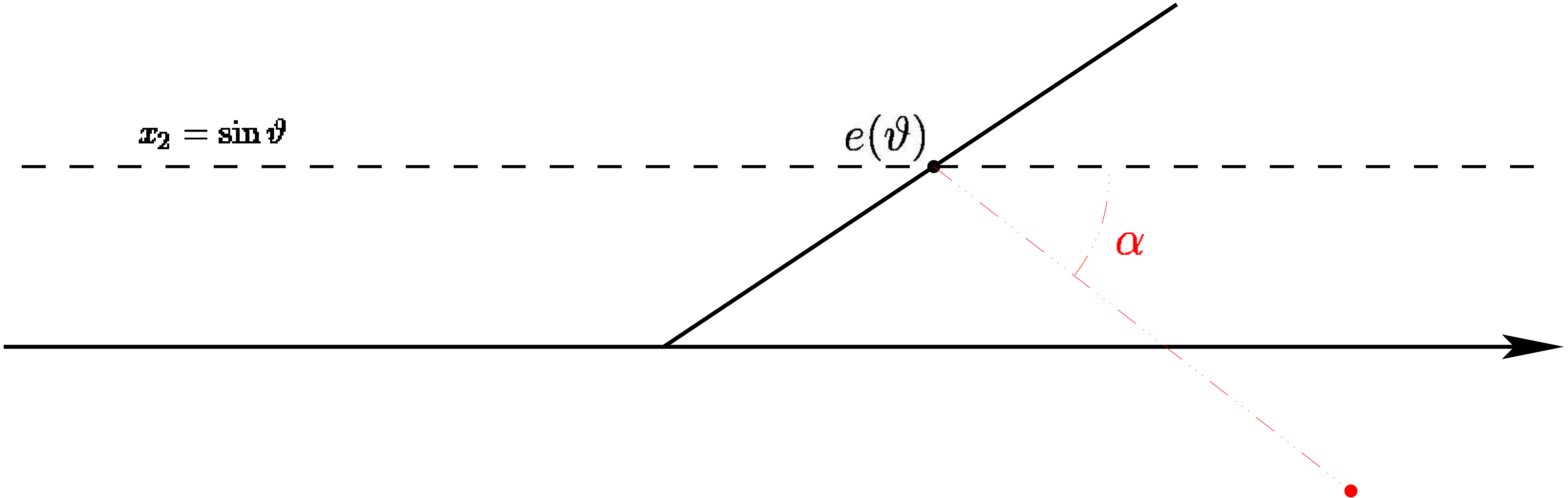}
\caption{A geometric argument involved in the proof of Lemma~\ref{JSIMDESEMa2}.}
        \label{JSIMDESEM2a2}
\end{figure}

\begin{proof} As in~\eqref{KSM X-10oqlrnfXcsw}, we have that the left hand side of~\eqref{KSM X-10oqlrnfXcsw-2} equals to
$$ \Lambda:=
\int_{\R\times(-\infty,0)}\frac{a_2^\star(\overrightarrow{
y-e^\star(\vartheta)})}{
|y-e^\star(\vartheta)|^{{2+s_1}}}\,dy.$$
Now we use polar coordinates centered at~$e^\star(\vartheta)$
by considering~$y=(\cos\vartheta,\sin\vartheta)+\rho(\cos\alpha,\sin\alpha)$ with~$\alpha\in(-\pi,0)$ and~$\rho>\frac{\sin\vartheta}{|\sin\alpha|}$,
see Figure~\ref{JSIMDESEM2a2}. In this way, and recalling~\eqref{Jtheta1theta2-Di4},
it follows that
\begin{eqnarray*} \Lambda=
\iint_{(-\pi,0)\times\left( \frac{\sin\vartheta}{|\sin\alpha|},+\infty\right)}\frac{\phi_2(\alpha)}{
\rho^{{1+s_1}}}\,d\alpha\,d\rho=\frac{1}{s_1(\sin\vartheta)^{s_1}}\int_{-\pi}^0
\phi_2(\alpha)\,|\sin\alpha|^{s_1}\,d\alpha,
\end{eqnarray*}
as desired.
\end{proof}

With this, we can uniquely determine the contact angle, as presented in
Theorem~\ref{CO:AN:CO}:

\begin{proof}[Proof of Theorem~\ref{CO:AN:CO}]
We let
$$ {\mathcal{W}}(\vartheta):=s_1 (\sin\vartheta)^{s_1}\left(
\int_{J_{\vartheta,\pi}}\frac{a_1(\overrightarrow{e(\vartheta)-x})}{|e(\vartheta)-x|^{n+s_1}}\, dx
-\int_{J_{0,\vartheta}}\frac{a_1(\overrightarrow{e(\vartheta)-x})}{|e(\vartheta)-x|^{n+s_1}}\, dx
-\sigma  \int_{H^c}\frac{a_2(\overrightarrow{e(\vartheta)-x})}{|e(\vartheta)-x|^{n+s_1}}\,dx\right)
$$
and we observe that solutions of~\eqref{sigmagenerale-0} correspond to zeros of~${\mathcal{W}}$
in~$[0,\pi]$.

Also, by Lemmata~\ref{JSIMDESEM} and~\ref{JSIMDESEMa2},
and recalling~\eqref{AIPA2},
\begin{equation}\label{eq:W:CO:W:W} \begin{split}
{\mathcal{W}}(\vartheta)\,&=
\int_{0}^\vartheta\phi_1(\alpha)\,(\sin\alpha)^{s_1}\,d\alpha
-\int_{\vartheta}^\pi \phi_1(\alpha)\,(\sin\alpha)^{s_1}\,d\alpha
-\sigma\int_{-\pi}^0
\phi_2(\alpha)\,|\sin\alpha|^{s_1}\,d\alpha\\&=
\int_{0}^\vartheta\phi_1(\alpha)\,(\sin\alpha)^{s_1}\,d\alpha
-\int_{\vartheta}^\pi \phi_1(\alpha)\,(\sin\alpha)^{s_1}\,d\alpha
-\sigma\int_{-\pi}^0
\phi_2(\pi+\alpha)\,(\sin(\pi+\alpha))^{s_1}\,d\alpha\\&=
\int_{0}^\vartheta\phi_1(\alpha)\,(\sin\alpha)^{s_1}\,d\alpha
-\int_{\vartheta}^\pi \phi_1(\alpha)\,(\sin\alpha)^{s_1}\,d\alpha
-\sigma\int_0^{\pi}\phi_2(\alpha)\,(\sin \alpha)^{s_1}\,d\alpha.
\end{split}\end{equation}
In particular, ${\mathcal{W}}$ is continuous in~$[0,\pi]$,
differentiable in~$(0,\pi)$ and,
for each~$\vartheta\in(0,\pi)$,
$$ {\mathcal{W}}'(\vartheta)=
2\,\phi_1(\vartheta)\,(\sin\vartheta)^{s_1}>0,$$
which shows that~${\mathcal{W}}$ admits at most one zero in~$(0,\pi)$.
This establishes the uniqueness result stated in Theorem~\ref{CO:AN:CO}.

Now we show the existence result claimed in Theorem~\ref{CO:AN:CO} under assumption~\eqref{0oujfn-29roh-32eirj-9034o5t-PK}. To this end, it suffices to notice that,
by~\eqref{0oujfn-29roh-32eirj-9034o5t-PK} and~\eqref{eq:W:CO:W:W}, we have that
$${\mathcal{W}}(0)=
-\int_{0}^\pi\phi_1(\alpha)\,(\sin\alpha)^{s_1}\,d\alpha
-\sigma\int_0^{\pi}\phi_2(\alpha)\,(\sin \alpha)^{s_1}\,d\alpha
<0$$
and
$${\mathcal{W}}(\pi)=
\int_{0}^\pi\phi_1(\alpha)\,(\sin\alpha)^{s_1}\,d\alpha
-\sigma\int_0^{\pi}\phi_2(\alpha)\,(\sin \alpha)^{s_1}\,d\alpha
>0.$$
{F}rom this and the continuity of~${\mathcal{W}}$, we obtain the existence of a zero of~${\mathcal{W}}$
in~$(0,\pi)$.
\end{proof}

\begin{rem}\label{REM:barvartheta} {\rm
We stress that the strict positivity of the kernel is essential for the uniqueness result in Theorem~\ref{CO:AN:CO}:
indeed, if one allows degenerate kernels in which~$a_1$ is only nonnegative, such a uniqueness claim
can be violated. As an example, consider~$\sigma:=0$ and pick~$\vartheta_0\in\left(0,\frac\pi2\right)$.
Let~$\phi_1\in C^\infty(\R)$ be such that~$\phi_1(\alpha):=0$ for all~$\alpha\in [\vartheta_0,\pi-\vartheta_0]$. Assume also that~$\phi_1\left(\frac\pi2+\alpha\right)=\phi_1\left(\frac\pi2-\alpha\right)$ for all~$\alpha\in\left(0,\frac\pi2\right)$ and that~$\phi_1(\alpha+\pi)=\phi_1(\alpha)$ for all~$\alpha\in(0,\pi)$.
See e.g. Figure~\ref{FIGURA3MUESE} for a sketch of this function.

\begin{figure}[h]
\includegraphics[width=0.75\textwidth]{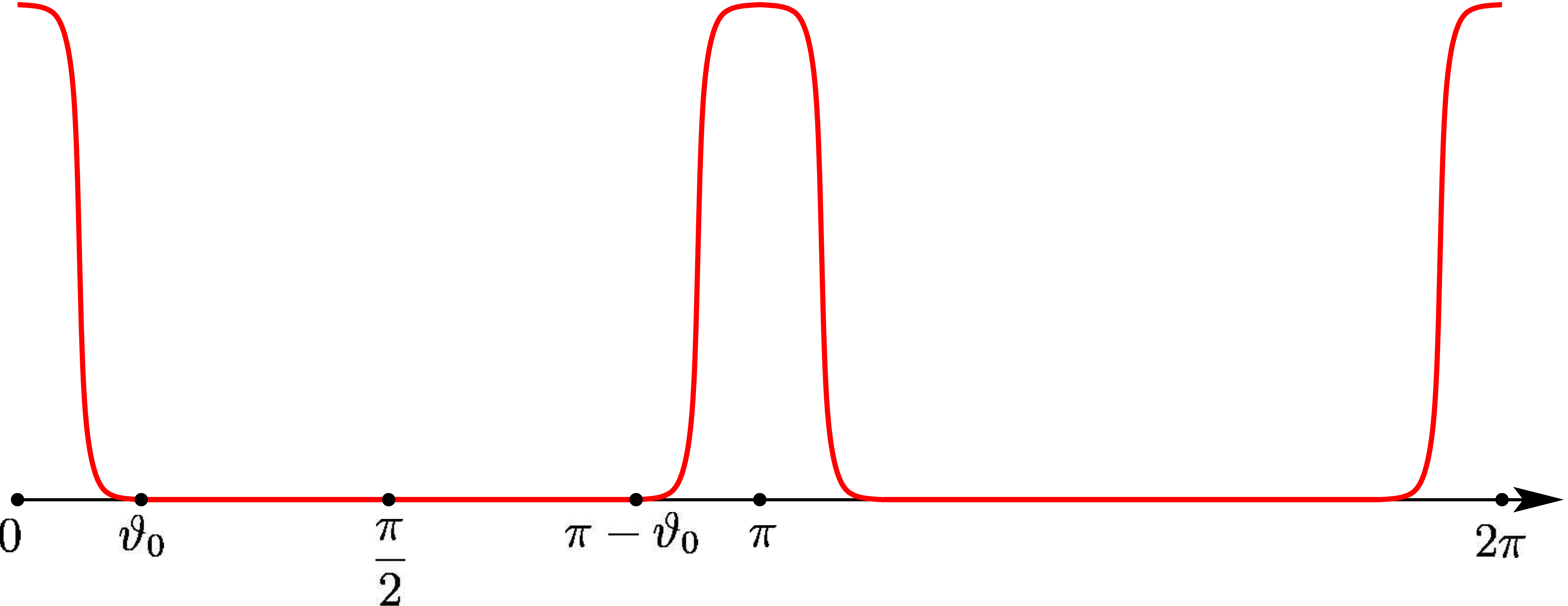}
\caption{A degenerate example of~$\phi_1$ leading to a multiplicity of the contact angle in~\eqref{sjwicru348bt4v8b4576848poiuytrewqasdfghjkmnbvs}.}
        \label{FIGURA3MUESE}
\end{figure}

Then, by~\eqref{eq:W:CO:W:W}, for every~$\bar\vartheta\in\left[\vartheta_0,\frac\pi2\right]$,
\begin{eqnarray*}
{\mathcal{W}}(\bar\vartheta)&=&
\int_{0}^{\bar\vartheta}\phi_1(\alpha)\,(\sin\alpha)^{s_1}\,d\alpha
-\int_{\bar\vartheta}^\pi \phi_1(\alpha)\,(\sin\alpha)^{s_1}\,d\alpha\\&=&
\int_{0}^{\vartheta_0}\phi_1(\alpha)\,(\sin\alpha)^{s_1}\,d\alpha
-\int_{\pi-\vartheta_0}^\pi \phi_1(\alpha)\,(\sin\alpha)^{s_1}\,d\alpha\\&=&
\int_{0}^{\vartheta_0}\phi_1(\alpha)\,(\sin\alpha)^{s_1}\,d\alpha
-\int_0^{\vartheta_0} \phi_1(\pi-\beta)\,(\sin(\pi-\beta))^{s_1}\,d\beta\\&=&
\int_{0}^{\vartheta_0}\phi_1(\alpha)\,(\sin\alpha)^{s_1}\,d\alpha
-\int_0^{\vartheta_0} \phi_1\left(\frac\pi2+\frac\pi2-\beta\right)\,(\sin \beta)^{s_1}\,d\beta\\&=&
\int_{0}^{\vartheta_0}\phi_1(\beta)\,(\sin\beta)^{s_1}\,d\beta
-\int_0^{\vartheta_0} \phi_1\left(\frac\pi2-\left(\frac\pi2-\beta\right)\right)\,(\sin \beta)^{s_1}\,d\beta\\&=&0,
\end{eqnarray*}
which shows that in this degenerate case every angle~$\bar\vartheta\in\left[\vartheta_0,\frac\pi2\right]$ would
be a zero of~${\mathcal{W}}$, hence
a solution of the contact angle equation in~\eqref{sigmagenerale-0}. Accordingly, the assumption of strict
positivity of the kernel cannot be dropped in
Theorem~\ref{CO:AN:CO}.}\end{rem}

\begin{appendix}

\section{Existence of minimizers and proof of Proposition~\ref{existenceofminimizer}}\label{APPENDI-A}

The proof of the existence result in Proposition~\ref{existenceofminimizer}
is based on a semicontinuity argument and on a direct minimization procedure.
We first provide the following lower semicontinuity lemma.

\begin{lem}[Semicontinuity of the energy]\label{lowersemic}  If $I_2(\Omega,\Omega^c)<+\infty$,
$E_j\subseteq \Omega$ and~$E_j\rightarrow E$ in $L^1(\Omega)$, then 
\[
\liminf_{j\rightarrow +\infty}\mathcal{E}(E_j) \geq \mathcal{E}(E).
\]
\end{lem}

\begin{proof}
If $\sigma\geq 0$, the proof follows by Fatou's Lemma. 
If instead~$\sigma<0$, then we observe that 
\[
I_2(\Omega,\Omega^c)=I_2(E,\Omega^c)+I_2(E^c\cap\Omega,\Omega^c),
\]
and therefore, using that $\sigma=-|\sigma|$,
we can write 
\begin{equation*}
\begin{split}
\mathcal{E}(E)&=I_{1}(E,E^c\cap \Omega)-|\sigma|I_{2}(E, \Omega^c)+(|\sigma|+1)I_2(\Omega,\Omega^c)
-(|\sigma|+1)I_2(\Omega,\Omega^c)
\\
&=I_{1}(E,E^c\cap \Omega)+I_{2}(E, \Omega^c)
+(|\sigma|+1)I_{2}(E^c\cap\Omega, \Omega^c)-(|\sigma|+1)I_2(\Omega,\Omega^c).
\end{split}
\end{equation*}
As a consequence, we can exploit Fatou's Lemma and obtain the desired result.
\end{proof}

With this we are able to prove Proposition~\ref{existenceofminimizer}:

\begin{proof}[Proof of Proposition~\ref{existenceofminimizer}]
We observe that, if~$K_1\in \mathbf{K}(n,s_1,\lambda,\varrho)$, then, for any $p\in\mathbb{R}^n$,
\begin{equation}\label{estimatefrombelow}
I_1(F,F^c)\geq \frac{1}{\lambda}I_{s_1}(F\cap B_{\varrho/2}(p),F^c\cap B_{\varrho/2}(p)), \qquad
{\mbox{for every }}F\subseteq \mathbb{R}^n.
\end{equation}
To prove it, we notice that if $x,y\in B_{\varrho/2}(p)$, then $|x-y|\leq |x-p|+|p-y|<\varrho$,
and therefore, recalling~\eqref{stimakernel},
\begin{eqnarray*}
I_1(F,F^c)&\geq& \int_{F\cap B_{\varrho/2}(p)}\int_{F^c\cap B_{\varrho/2}(p)}K_1(x-y)\, dx\, dy\\
&
\geq& \frac{1}{\lambda}\int_{F\cap B_{\varrho/2}(p)}\int_{F^c\cap B_{\varrho/2}(p)}\frac{dx\, dy}{|x-y|^{n+s_1}},
\end{eqnarray*}
which establishes~\eqref{estimatefrombelow}.

Now, if $H$ is a half-space such that $|H\cap \Omega|=m$ and $R>0$ is such that $\Omega\subseteq B_R$, then, using~\eqref{stimakernel}, we see that
$$ I_1(H\cap B_R, H^c\cap B_R)=C\,R^{n-s_1},$$
for some~$C>0$ depending only on~$n$ and~$s_1$,
and therefore
\begin{eqnarray*}
\mathcal{E}(H\cap\Omega)&=&I_1(H\cap \Omega, (H\cap \Omega)^c\cap\Omega)+\sigma I_2(H\cap \Omega, \Omega^c)\\&
=&I_1(H\cap \Omega, H^c\cap\Omega)+\sigma I_2(H\cap \Omega, \Omega^c)\\&
\leq& I_1(H\cap B_R, H^c\cap B_R)+|\sigma|\,I_2(\Omega,\Omega^c)\\&<&+\infty.
\end{eqnarray*}
As a consequence, we find that
$$\gamma:=\inf \left\{ {\mathcal{C}}(E)\; :\; E\subseteq\Omega,\; |E|=m\right\}
<+\infty.$$
Let now~$E_j\subseteq \Omega$ be such that $|E_j|=m$ and~${\mathcal{C}}(E_j)=\mathcal{E}(E_j)+\int_{E_j}g\rightarrow \gamma$ as~$j\to+\infty$. Then, if $j$ is large enough, we have that 
\[
\gamma+1+\int_{\Omega}|g|\geq {\mathcal{E}}(E_j)=I_1(E_j,E_j^c\cap\Omega)+\sigma I_2(E_j,\Omega^c)\geq I_1(E_j,E_j^c\cap\Omega)- |\sigma| \,I_2(\Omega,\Omega^c).
\]
Consequently
\[
I_1(E_j,E_j^c)=I_1(E_j,E_j^c\cap\Omega)+I_1(E_j,E_j^c\cap\Omega^c)\leq \gamma+1+\int_{\Omega}|g|+
I_1(\Omega,\Omega^c)+
|\sigma|I_2(\Omega,\Omega^c).
\]
Since $E_j\subseteq B_R$, using~\eqref{estimatefrombelow} and the fact that the space $W^{s_1,2}(B_R)$ is compactly embedded in~$L^1(B_R)$, we find that, up to
a subsequence, $E_j\rightarrow E$ in $L^1_{\mathrm{loc}}(\mathbb{R}^n)$ for some~$E\subseteq \Omega$ with~$|E|=m$. Hence, using the semicontinuity property in Lemma~\ref{lowersemic},
we conclude that $E$ is a minimizer.

We also remark that 
\begin{equation*}
I_2(E,\Omega^c)\leq I_2(\Omega,\Omega^c)<+\infty,
\end{equation*}
and therefore, since~$\mathcal{E}(E)<+\infty$, we conclude that 
\[
I_1(E,E^c\cap \Omega)<+\infty,
\]
as desired. 
\end{proof}

\end{appendix}

\bibliography{biblio}

\end{document}